%% file: tgv_manifold.tex
\documentclass[a4paper]{article}
\usepackage{amsthm}
\usepackage{amsmath}
\usepackage{amssymb}
\usepackage{tikz}
\usepackage{float}
\usepackage{algpseudocode}
\usepackage{setspace} 
\usepackage{booktabs}
\usepackage{geometry}
\usepackage{appendix}

\usepackage{paralist}
\floatstyle{ruled}
	\newfloat{algorithm}{h}{loa}
\floatname{algorithm}{Algorithm}

\include{definitions}

\usepackage{tikz}
\usepackage{pgfplots}
\usetikzlibrary{spy}

\usepackage[font=small,labelfont=bf]{caption}
\usepackage[font=footnotesize]{subcaption}

\definecolor{y}{rgb}{0.2902,1.0000,0.0788}
\definecolor{cb}{rgb}{0,0.1399,0.9316}
\definecolor{cr}{rgb}{0.7218,0,0}

\definecolor{darkgreen}{rgb}{0,0.6,0}

\newcommand{\deltaSNR}{\mathrm{\Delta SNR}}
\newcommand{\dB}{\mathrm{\,dB}}
\newcommand{\Pos}{\mathrm{Pos}}

\newlength{\pix}
\usepackage{url}

\title{Total Generalized Variation for Manifold-valued Data}

\author{K. Bredies\footnote{K. Bredies and M. Holler are with Institute of Mathematics and Scientific Computing of the University of Graz. The institute is a member of NAWI Graz (\texttt{www.nawigraz.at}). K.~B. and M.~H. are members of BioTechMed Graz (\texttt{www.biotechmed.at}); 
email:~\texttt{kristian.bredies@uni-graz.at, martin.holler@uni-graz.at}		.} \and M. Holler\footnotemark[1] \and 
	M. Storath\footnote{Martin Storath is with the Image Analysis and Learning Group, Heidelberg
		Collaboratory for Image Processing, Universit\"at Heidelberg, Germany;
		email:~\texttt{martin.storath@iwr.uni-heidelberg.de}.}
	 \and 
	A. Weinmann\footnote{Andreas Weinmann is with the Institute of Computational Biology, Helmholtz
		Zentrum M\"unchen, and with the Department of Mathematics and Natural
		Sciences, Hochschule Darmstadt, Germany; email:~\texttt{andreas.weinmann@h-da.de}}.}
\date{}
\begin{document}

\newtheorem{prop}{Proposition}[section]
\newtheorem*{prop*}{Proposition}
\newtheorem{rem}[prop]{Remark}
\newtheorem*{rem*}{Remark}
\newtheorem{thm}[prop]{Theorem}
\newtheorem*{thm*}{Theorem}
\newtheorem{defn}[prop]{Definition}
\newtheorem*{defn*}{Definition}
\newtheorem{lem}[prop]{Lemma}
\newtheorem*{lem*}{Lemma}
\newtheorem{cor}[prop]{Corollary}
\newtheorem*{cor*}{Corollary}

\maketitle

\begin{abstract}
In this paper we introduce the notion of second-order total generalized variation (TGV) regularization for manifold-valued data in a discrete setting.
We provide an axiomatic approach to formalize reasonable generalizations of TGV to the manifold setting and present two possible concrete instances that fulfill the proposed axioms. We provide well-posedness results and present algorithms for a numerical realization of these generalizations to the manifold setup. Further, we provide experimental results for synthetic and real data to further underpin the proposed generalization numerically and show its potential for applications with manifold-valued data.
\end{abstract}


\section{Introduction}

In this work we introduce and explore a generalization of second-order total generalized variation (TGV) regularization for manifold-valued data. 
The TGV functional has originally been introduced in \cite{Bredies10} for the vector space setting as generalization of the total variation (TV) functional \cite{rudin1992nonlinear}, which is extensively used for regularization in image processing and beyond. 
The advantage of TV regularization compared to, e.g., classical $H^1$ regularization approaches, is that jump discontinuities can be much better reconstructed. This can be seen in the function space setting since functions of bounded variation, as opposed to Sobolev functions, can have jump discontinuities. It is also reflected in numerical realizations where TV minimization allows to effectively preserve sharp interfaces. A disadvantage of TV regularization, however, is its tendency to produce piecewise constant results even in non-piecewise-constant regions, which is known as the staircasing effect. Employing a regularization with higher order derivatives, such as second-order TV regularization, overcomes this drawback, but again does not allow for jump discontinuities. As a result, a lot of recent research aims at finding suitable extensions of TV that overcome the staircasing effect, but still allow for jumps \cite{Chambolle97,setzer2011infimal,Bergounioux15infcon_analysis}. While infimal-convolution-type approaches can be seen as the first methods to achieve this, the introduction of the TGV functional (of arbitrary order $k$) in 2010 finally provided a complete model for piecewise smooth data with jump discontinuities. This is achieved by an optimal balancing between first and higher order derivatives (up to the order $k$), which is carried out as part of the evaluation of the functional. We refer to \cite{Bredies10} for more motivation and to \cite{BH_inverse} for a detailed analysis of TGV in the context of inverse problems. 
In particular, second-order TGV, which balances between first and second-order derivatives, achieves piecewise linear -- as opposed to piecewise constant -- image reconstructions while it still allows for jumps. This renders second-order TGV a well suited model for piecewise smooth images and can be seen as the motivation of the use of second-order TGV regularization in a plethora of applications \cite{knoll2011tgv2mri,bredies2015decompressiontgv1,bredies2015decompressiontgv2,holler16mri_pet,niu2014sparse}. Up to now, TGV regularization was only available for vector space data and applications are hence limited to this situation.

In various problems of applied sciences, however, the data do not take values in a linear space but in a nonlinear space such as a smooth manifold.
Examples of manifold-valued data are circle and sphere-valued data as appearing in interferometric SAR imaging \cite{massonnet1998radar}, wind directions \cite{schultz1990circular}, orientations of flow fields \cite{adato2011polar, storath2017fastMedian},
and color image processing \cite{chan2001total,vese2002numerical,kimmel2002orientation, lai2014splitting}.
Other examples are data taking values in the special orthogonal group $SO(3)$ expressing vehicle headings, aircraft orientations or camera positions \cite{rahman2005multiscale},
Euclidean motion group-valued data \cite{rosman2012group} as well as shape-space data \cite{Michor07}.
Another prominent manifold is the space of positive (definite) matrices
endowed with the Fisher-Rao metric  \cite{radhakrishna1945information}. 
This space is the data space in DTI  \cite{pennec2006riemannian}.
It is a Cartan-Hadamard manifold and as such it has particularly nice differential-geometric properties. 
DTI allows to quantify the diffusional characteristics of a specimen non-invasively \cite{basser1994mr,johansen2009diffusion}
which is helpful in the context of neurodegenerative pathologies such as schizophrenia \cite{foong2000neuropathological} or autism 
\cite{alexander2007diffusion}. 
Because of the natural appearance of these nonlinear data spaces quite a lot of recent work deals with them. 
Examples are wavelet-type multiscale transforms for manifold-valued data 
\cite{rahman2005multiscale,grohs2009interpolatory, weinmannConstrApprox, wallner2011convergence, weinmann2012interpolatory}
as well as manifold-valued partial differential equations  
\cite{tschumperle2001diffusion, chefd2004regularizing, grohs2013optimal}.
Work on statistics on Riemannian manifolds can be found in 
\cite{oller1995intrinsic,bhattacharya2003large,fletcher2004principal,%
bhattacharya2005large,pennec2006intrinsic,fletcher2007riemannian}.
Optimization problems for manifold-valued data are for example the topic of \cite{absil2009optimization,absil2004riemannian},
of \cite{grohs2016varepsilon} 
and of \cite{hawe2013separable} with a view towards learning in manifolds. We also mention related work on optimization in Hadamard spaces \cite{bavcak2016second, bacak2014convex, bavcak2013proximal} and on the Potts and Mumford-Shah models for manifold-valued data  \cite{weinmann2015mumford,storath2017jump}.

TV functionals for manifold-valued data have been considered from the analytical side in   
~\cite{GM06,GM07,GMS93}; in particular, the existence of minimizers of certain TV-type energies has been shown. 
A convex relaxation based algorithm for TV regularization for $\mathbb S^1$-valued data was considered in~\cite{SC11,CS13}. 
Approaches for TV regularization for manifold-valued data are considered in \cite{LSKC13}
which proposes a reformulation as multi-label optimization problem and a convex relaxation,
in \cite{grohs2016total} which proposes iteratively reweighted minimization, 
and in \cite{weinmann2014total} which proposes cyclic and parallel proximal point algorithms.
An exact solver for the TV problem for circle-valued signals has been proposed in \cite{storath2016exact}.
Furthermore, \cite{steidl16half_quadratic} considers half-quadratic minimization approaches, which may be seen as an extension of \cite{grohs2016total}, and \cite{Steidl16dr_manifold} considers an extension of the Douglas-Rachford algorithm for manifold-valued data.
Applications of TV regularization to shape spaces may be found in \cite{baust2015total, stefanoiu2016joint}.
TV regularization with a data term involving an imaging type operator has been considered in \cite{baust2016combined}
where a forward-backward type algorithm is proposed to solve the corresponding inverse problem involving manifold-valued data.

As with vector space data, TV regularization for manifold-valued data
has a tendency to produce piecewise constant results in regions where
the data is smooth.  As an alternative which prevents this,
second-order TV type functionals for circle-valued data have been
considered in \cite{bergmann2014second,bergmann2016second} and, for
general manifolds, in \cite{bavcak2016second}.  However, similar to
the vector space situation, regularization with second-order TV type
functionals tends towards producing solutions which do not preserve
the edges as desired.  To address this drawback, the vector space
situation guides us to considering TGV models and models based on
infimal convolution to address this issue.  The most difficult part in
this respect is to define suitable notions in the manifold setup which
have both reasonable analytic properties on the one hand and which are
algorithmically realizable on the other hand.  
  Concerning infimal-convolution type functionals such as
  $\TV$--$\TV^2$ infimal convolutions, a first effort towards a
  generalization to the manifold setting has been made in the recent
  conference proceeding \cite{Steidl17_infcon_manifold} which has
  later been extended in an \emph{arXiv} preprint
  \cite{bergmann2017priors}. The present manuscript
  \cite{bredies2017total}, which was submitted to \emph{arXiv} at the
  same time as \cite{bergmann2017priors}, emerged independently.  
  In this work, we focus on TGV and aim at providing a thorough study of
  reasonable generalizations of TGV for piecewise smooth,
  manifold-valued data by first investigating crucially required
  properties of generalizations of TGV to the manifold setting.  Then
we propose suitable extensions that fulfill these properties and which
are, in addition, computationally feasible.
In this respect, it is important to note that due to the cascadic nature of TGV (as opposed to infimal-convolution-type functionals), its generalization to manifolds requires a conceptually new approach. For this reason, both the techniques we propose for a generalization of TGV as well as our algorithmic setting are rather different from the one of \cite{Steidl17_infcon_manifold}, which uses a mid-point approach to generalize $\TV$--$\TV^2$ infimal convolutions and gradient descent for numerical minimization.

\subsection{Contributions}

The contributions of the paper are as follows: \emph{(i)}
we lay the foundations for -- and provide concrete realizations of -- a suitable variational model for second-order total generalized variation for manifold-valued data,
\emph{(ii)} we provide algorithms for the proposed models,
\emph{(iii)} we show the potential of the proposed algorithms by applying them to synthetic and real data.
Concerning \emph{(i)}, we use an axiomatic approach.
We first formalize reasonable fundamental properties of vector-valued TGV which should be conserved in the manifold setting. 
Then we propose two concrete realizations which we show to fulfill the axioms. The one is based on parallel transport and the other is motivated by the Schild's approximation of parallel transport. We obtain well-posedness results for TGV-based denoising of manifold valued data for both variants.
Concerning \emph{(ii)} we provide the details for a numerical realization of variational regularization for general manifolds using either of the two proposed generalizations of TGV. We build on the well-established concept of cyclic proximal point algorithms (using the inexact variant). The challenging part and the main contribution consists in the computation of the proximal mappings of the TGV atoms involved.
In particular, for the class of symmetric spaces, we obtain closed form expressions of the related derivatives in terms of Jacobi fields for the Schild variant;
for the parallel transport variant, we obtain closed form expressions for the related derivatives for general symmetric spaces up to a derivative 
of a certain parallel transport term for which we can still analytically compute the derivative for the class of manifolds considered in the paper.
Concerning \emph{(iii)}, we provide a detailed numerical investigation of the proposed generalization and its fundamental properties.
Furthermore, we provide experiments with real and synthetic data and a comparison to existing methods.

\subsection{Outline of the paper}

The paper is organized as follows. 
The topic of Section~\ref{sec:defModel} is the derivation of suitable TGV models for manifold-valued data. We start out with a detailed discussion of fundamental properties expected by a reasonable TGV version for manifold-valued data.
To this end, we first reconsider vector-space TGV in a form suitable for our purposes in Subsection~\ref{sec:TGVdefvecSpace}.
Then we derive an axiomatic extension of $\TGVat$ to the manifold setting where we -- for a better understanding -- first consider the univariate case in  Subsection~\ref{sec:TGVdefuni}. 
Next we suggest two realizations which we call the Schild variant of TGV and the parallel transport variant of TGV  and show that these realizations indeed fulfill all desired properties.
Then, in Subsection~\ref{sec:axomatic_extension_bivariate}, we extend the axiomatic framework to the bivariate setting and provide bivariate versions of 
the Schild variant of TGV and the parallel transport variant of TGV, respectively, that we show to fulfill the required axioms.
In Section \ref{sec:existence_results} we then provide existence and lower semi-continuity results for the proposed variants of TGV from which the existence of minimizers for the TGV-regularized denoising problems of manifold-valued data follows.
The topic of Section~\ref{sec:algorithm} is the algorithmic realization of the proposed models.
We start out recalling the concept of a cyclic proximal point algorithm (CPPA). 
Then, in Subsection~\ref{sec:AlgBasicUni}, we consider the implementation of the CPPA for manifold-valued TGV.
We identify certain TGV atoms whose proximal mappings are challenging to compute. 
The necessary derivations needed for their computation are provided
for the Schild variant of TGV in Subsection~\ref{subsec:AlgSchildUniv}, 
and, for the parallel transport variant of TGV, in Subsection~\ref{subsec:AlgParUniv}.
The topic of Section~\ref{sec:numericalResults} is the numerical evaluation of the proposed schemes.
There, we first carry out a detailed numerical evaluation of the proposed model for synthetic data. We test extreme parameter choices and ensure consistency of the results of our numerical method with a reference implementation for vector spaces. Then we consider various application scenarios and compare to existing methods.
Finally, we draw conclusions in Section~\ref{sec:Conclusion}.

\subsection{Basic Notation and Concepts from Differential Geometry}\label{sec:notation}

Throughout this work, $\M \subset \R ^K$ will always denote a complete manifold with a Riemannian structure
(with its canonical metric connection, its Levi-Civita connection). Let us explain these notions more precisely.
We consider a manifold $\M$ with a Riemannian metric which is a smoothly varying inner product
	 $\langle \cdot,\cdot \rangle_p$
	  in the tangent space of each point $p$. 
	  (We note that any paracompact manifold can be equipped with a Riemannian structure.) 
	  As usual, we will frequently omit the dependence on $p$ in the notation in the following.
 According to the Hopf-Rinow theorem, the complete manifold $\M$ is geodesically complete in the sense
that any geodesic can be prolongated arbitrarily. Here a geodesic is a curve $\gamma$ of zero acceleration, i.e.,
$\frac{D}{dt} \frac{d}{dt} \gamma = 0,$ where $\frac{D}{dt}$ denotes the covariant derivative along the curve $\gamma$ (see below for details on the covariant derivative). Geodesics are invariant under affine reparametrizations and usually identified with their image in $\M$.
(We parametrize them on $[0,1]$ or, more generally, on $[t_0,t_1]$ with $t_0<t_1$, depending on the context.) 
We always denote by $d: \M\times \M \rightarrow [0,\infty)$ the distance on $\M$ which is induced by the Riemannian structure and note that, since $\M$ is complete, for any $a,b \in \M$ there is always a geodesic from $a$ to $b$ whose length equals $d(a,b)$. Such geodesics are called length minimizing geodesics between $a$ and $b$. By $T_a\M$ we denote the tangent space of $\M$ at $a \in \M$ and by $T\M$ the tangent bundle of $\M$. 
Further, we need the notion of parallel transport. The parallel transport of a vector $v \in T_a \M$ with $a \in \M$ to a point $b\in \M$ along a curve $\gamma:[0,1] \rightarrow \M$ such that $\gamma(0) = a$, $\gamma(1) = b$ is the vector $V_1 = V(1) \in T_b \M$, where $V:[0,1] \rightarrow T\M$ is given as the solution of the ODE $\frac{D}{dt}V(t) = 0$ on $[0,1]$  with the initial condition $V(0)=v,$ where the covariant derivative $\frac{D}{dt}$ is taken along the curve $\gamma$.
The covariant derivative $\frac{D}{dt}$ is induced by the Levi-Civita connection $\nabla_X Y$ of the Riemannian manifold $\M$;
a connection $\nabla_X Y$ is a $C^\infty$-linear mapping w.r.t.\ the vector field $X$ and a derivation w.r.t.\ the vector field $Y.$  (Connections are needed since, in a general manifold $\M$, there is no a priori canonical way to define the directional derivative in direction $X$ of a vector field $Y$.) 
For the covariant derivative $\frac{D}{dt}Y$ along a curve $\gamma$, a direction $X$ along the curve $\gamma$ is given by $\frac{d}{dt}\gamma$ and the vector field $Y$ is given along the curve; hence, if $\frac{d}{dt}\gamma \neq 0,$
$\frac{D}{dt}Y$ is just given by $\nabla_{\frac{d}{dt}\gamma} Y.$
In a Riemannian manifold  $\M,$ there is a (uniquely determined) canonical metric connection, its Levi-Civita connection.
The Levi-Civita connection is the only connection which is symmetric (i.e.,$\nabla_X Y = \nabla_Y X$ for any vector fields $X,Y$ which commute) and which is compatible with the metric (i.e., in terms of the induced covariant derivative which we need later on, $\frac{d}{dt} \langle X,Y \rangle =$ $\langle\frac{D}{dt}X,Y\rangle$ + $\langle X, \frac{D}{dt}Y\rangle$ for any two vector fields $X,Y$ along a curve $\gamma,$ along which the covariant derivatives are taken.) 
For an account on Riemannian geometry we refer to the books \cite{spivak1975differential,do1992riemannian} or to
\cite{kobayashi1969foundations}.

For later use, we fix some further notation. To this end, let
	 $a,b \in \M,$  $v \in T_a \M,$ and $\gamma:[0,L] \rightarrow \M$ be a curve connecting $a$ and $b$ such that $\gamma(t_0) = a$, $\gamma(t_1) = b$ with $t_0<t_1$, $t_0,t_1 \in [0,L].$ We define
	\begin{enumerate}
		\item $\exp_a(v) = \psi(1)$ where $\psi:[0,1]\rightarrow \M$ is the unique geodesic such that $\psi(0) = a$, $\frac{d}{dt}\psi(0) = v$,
		\item $\exp(v) = \exp_a(v)$ when it is clear from the context that $v \in T_a\M$,
		\item $\log_a(b) = \{ z \in T_a\M \st  \exp(z) = b \text{ and } [0,1] \ni t \mapsto \exp(tz) \text{ is a length-minimizing geodesic}\}$,
		
		\item in case $\gamma$ is a geodesic and length-minimizing between $a$ and $b$ on $[t_0,t_1]$, we denote 
		$\log^\gamma_a(b) = (t_1 - t_0)\frac{d}{dt}\gamma(t_0),$
		i.e., the vector in $\log_a(b)$ that is parallel to $\frac{d}{dt}\gamma(t_0)$ and such that $\exp_a(\log^\gamma_a(b)) = b$,
		\item $\pt^\gamma_{a,b}(v) = z \in T_b M$, where $z$ is the vector resulting from the parallel transport of $v$ from $a$ to $b$ along the curve $\psi$ that reparametrizes $[t_0,t_1]$ to $[0,1]$ in an affine way, i.e., $\psi(t) = \gamma(t_0 + t(t_1 - t_0))$ such that $\psi(0) = a$, $\psi(1) = b$,
		\item $\pt_{a,b}(v) = \{ \pt^\psi_{a,b}(v)  \st \psi \text{ is a length-minimizing geodesic connecting } a \text{ and } b\}$,
		\item $\pt_b(v) = \pt_{a,b}(v)$, $\pt^\gamma_{a,b}(v) = \pt^\gamma_b(v)$ when it is clear from the context that $v \in T_a\M$,
		\item for $t \in [0,1]$,  \[[a,b]_t = \{ \psi (t) \,|\, \psi:  [0,1] \rightarrow \M \text{ is a length-minimizing geodesic with } \psi(0) = a, \, \psi(1) = b \}.\] 
		We extend this definition for $t \in \R\setminus [0,1]$ by extending the corresponding geodesic.
	\end{enumerate}
We also note that throughout the paper we identify sets having only one element with the corresponding element.

We further need the concept of geodesic variations for the existence results in Section~\ref{sec:existence_results} and the derivation of the algorithms in Section~\ref{sec:algorithm}. A variation of a curve $\gamma:I \to \M$ defined on an interval $I$ is a smooth mapping $V:I \times J \to \M,$ $J$ an interval containing $0$, such that $V(s,0) = \gamma(s)$ for all $s \in I.$ A variation is called geodesic, if all curves $s \mapsto V(s,t)$ are geodesics for any $t$ in $J.$

\section{Definition of TGV for manifold-valued data}\label{sec:defModel}
The goal of this section is to define a discrete total generalized variation (TGV) functional of second order for manifold-valued data. To this end, we first state some fundamental properties of the TGV functional in infinite and finite dimensional vector spaces. The definition of a generalization of TGV to the manifold setting will then be driven by the goal of preserving these fundamental properties of vector-space TGV.

\subsection{TGV on vector spaces}\label{sec:TGVdefvecSpace}    
We first recall the definition of TGV on infinite-dimensional vector spaces via its minimum representation which is, according to results in \cite{Bredies11_inverse,Valkonen12,BH_inverse,bredies2015decompressiontgv1} covering the second- and general order case as well as the scalar and vectorial case, equivalent to the original definition as provided in \cite{Bredies10}. 
Then we present a discretization which is slightly different from the one typically used.

\begin{defn}[Minimum representation of TGV in vector spaces]\label{def:tgv_continuous} For $\alpha_0,\alpha_1 \in (0,\infty)$,  $u \in L^1_{\loc}(\Omega)^K$ with $\Omega \subset \R^d$ a bounded Lipschitz domain, we define 
\begin{equation} \label{eq:tgv_continuous_setting}
\TGVat (u) = \min _{w \in \BD(\Omega,\R^{d})^K} \alpha _1 \|\Wrt u - w\|_{\M} + \alpha_0 \|\symgrad w\|_{\M} .
\end{equation}
Here , $\|\cdot \|_{\M}$ denotes the Radon norm in the space of Radon measures $\M(\Omega,X)^K:= (C_0(\Omega,X)^K)^*$ with $X \in \{ \R^{d},S^{d\times d}\}$, $S^{d\times d}$ the space of symmetric matrices and further $\BD(\Omega,\R^d)^K:= \{ w \in L^1(\Omega,\R^d)^K \st \symgrad w \in \M(\Omega,S^{d\times d})^K\}$. The derivatives $\Wrt u \in \mathcal{D}(\Omega,\R^{d})^K$ and $\symgrad w \in \mathcal{D}(\Omega,S^{d\times d })^K$ are defined in the weak sense by 
\begin{align*}
\langle \Wrt u ,\varphi \rangle &= -\langle  u , \dive \varphi \rangle , \quad \varphi  \in C^\infty _c (\Omega,\R^{d})^K, \\ 
\langle \symgrad w ,\varphi \rangle &= -\langle  w , \dive \varphi \rangle , \quad \varphi  \in C^\infty _c (\Omega,S^{d\times d})^K ,
\end{align*}
with $\dive \varphi = ( \dive \varphi^1,\ldots,\dive \varphi^K)\in C^\infty _c (\Omega)^K$ for $\varphi = (\varphi^1,\ldots,\varphi^K) \in C^\infty _c (\Omega,\R^{d})^K$ and, for $\varphi = (\varphi^ 1,\ldots,\varphi^ K) \in C^\infty _c (\Omega,S^{d\times d})^K$ with $\varphi^ i  = (\varphi^ i_1,¸\ldots,\varphi^ i_d) \in C^\infty _c (\Omega,S^{d\times d})$, we denote $\dive\varphi = (\dive \varphi^1,\ldots ,\dive \varphi^K) \in C^\infty _c (\Omega,\R^d)^K$ with $\dive \varphi^i = (\dive \varphi^i_1,\ldots,\dive \varphi^i_d)\in C^\infty _c (\Omega,\R^d)$. See \cite{BH_inverse,bredies2015decompressiontgv1} for details.
\end{defn}
Note that $\TGVat(u)$ is finite if and only if $u\in \BV(\Omega)^K$ and, in this case, the minimum is actually obtained \cite{BH_inverse}. Hence the term $\min$ in the above definition is justified.

For a discretization and generalization later on, it is convenient to list some of the main properties of second-order $\TGV$ in function space (see \cite{BH_inverse,bredies2015decompressiontgv1}):
\begin{itemize}
\item[\textbf{(P1)}] If the minimum in \eqref{eq:tgv_continuous_setting} is obtained at $w=0$, then $\TGVat(u) = \alpha_1 \TV(u)$,
\item[\textbf{(P2)}] If the minimum in \eqref{eq:tgv_continuous_setting} is obtained at $w=\Wrt u$, then $\TGVat(u) = \alpha_0 \TV^ 2(u)$,
\item[\textbf{(P3)}] $\TGVat(u) = 0$ if and only if $u$ is affine.
\end{itemize}
Here, $\TV^ 2$ denotes a second-order TV functional which can be defined as $\TV^ 2(u)  = \|\Wrt ^ 2 u\|_{\M} $
where $\Wrt ^ 2u$ denotes the second-order distributional derivative of $u$ and the above quantities are finite if and only if $\Wrt^ 2u \in \M(\Omega,S^ {d\times d})^K$. We note that $w=0$ and $w=\Du$ can trivially be obtained when $u$ is constant and affine, respectively, but $w=0$ is for instance also obtained under some symmetry conditions when $u$ is piecewise constant or when the ratio $\alpha_0/\alpha_1$ is large enough \cite{Bredies10,Papafitsoros2015}.

Using the minimum representation above, a discretization of $\TGVat$ in vector spaces on two-dimensionals grids can be given as follows.

\begin{defn}[Discrete isotropic and anisotropic TGV in vector spaces]\label{def:discrete_tgv_vector_space} Set $U = \R^K$. For $u=(u_{i,j})_{i,j} \in U^{N\times M}$ with $u_{i,j} \in U$, we define the discrete second-order TGV functional as
\begin{equation} \label{eq:tgv_discrete_vector_space}
 \TGVat(u) = \min _{w \in U^{(N-1)\times M }  \times U^{N\times (M-1)} }\alpha_1 \|\Wrt u - w\|_1 + \alpha_0 \|\symgrad w\|_1 
 \end{equation}
where 
\[
\Wrt u := (\delta_{x+} u,\delta_{y+} u) \in U^{(N-1)\times M }  \times U^{N\times (M-1) }\]
and 
\[
\symgrad w := \symgrad (w^1,w^2):= ( \delta_{x-} w^1,\delta_{y-} w^2,\tfrac{\delta_{y-} w^1 + \delta_{x-} w^2 }{2}) \in U^{(N-2)\times M } \times U^{N\times (M-2) } \times U^{(N-1)\times (M-1) }
\]
with $\delta_{x+}$, $\delta_{y+}$ and $\delta_{x-}$, $\delta_{y-}$ being standard forward and backward differences, respectively. Introducing a parameter $p \in [1,\infty)$, the one-norms are given as
\[ \|w\|_1 =  \sum _{i,j} |(w_{i,j}^1,w_{i,j}^2)|_p \text{ for } w \in U^{(N-1)\times M }  \times U^{N\times (M-1)} \]
and 
\[ \|z\|_1 =  \sum _{i,j} \left |\left( \begin{matrix}
z_{i,j}^1 & z_{i,j}^3 \\ z_{i,j}^3 & z_{i,j}^2
\end{matrix}\right) \right|_p \text{ for } z \in U^{(N-2)\times M } \times U^{N\times (M-2) } \times U^{(N-1)\times (M-1)}, \]
where we extend the signals by zero to have the same size. Here, $|(w^1,w^2)|_p:= \left( |w^1|^p + |w^2|^p\right)^{1/p}$ and $\left|\left(\begin{matrix}
z^1 & z^3 \\ z^3 & z^2
\end{matrix}\right)\right|_p:= \left( |z^1|^p + |z^2|^p + 2|z^3|^p\right)^{1/p}$ with $|\cdot| $ the Euclidean norm on $U=\R^K$.
\end{defn}

Note that we incorporate a pointwise $\ell^p$-norm with $p \in [1,\infty)$ in the definition of $\TGVat$. In the function space setting of Definition \ref{def:tgv_continuous}, this corresponds to choosing the appropriate dual norm on $X \in \{\R^d,S^{d\times d}\}$. There, any choice of point-wise norm (also beyond $\ell^p$-norms) yields equivalent functionals and function spaces, with the Euclidean norm being used in the first paper on TGV \cite{Bredies10} and the most popular choices being $\ell^p$ norms with $p \in \{1,2\}$. The focus of this paper is the case $p=1$ since this is more frequently used in the manifold context because of being more amenable to numerical realization. In the theory part we include the case of $p \in [1,\infty)$ since it can be done without additional effort.

We remark that the above discretization of $\TGVat$ is slightly different from the standard one as provided, e.g., in \cite{Bredies10}. The purpose of this is to achieve consistency of the zero set of both the continuous and discrete version as follows. Also, note that, notation-wise, we do not distinguish between the continuous and the discrete version of TGV (and of TV and $\TV^2$). In the following, we will always refer to discrete versions.
\begin{prop}[Zero set of $\TGVat$ in vector spaces]\label{prop:discrete_vector_space_kernel} For $u\in U^{N\times M}$ we have that $\TGVat(u) = 0$ if and only if $u$ is affine, i.e., there exist $a,b,c \in U$ such that $u_{i,j} = ai + bj + c $.
\begin{proof}
We provide only the basic steps: Setting $w = \Wrt u$ we get that 
\[\symgrad w = \symgrad \Wrt u = \left(
\delta_{x-}\delta_{x+} u ,  \delta_{y-}\delta_{y+} u , \tfrac{ \delta_{y-}\delta_{x+} u + \delta_{x-}\delta_{y+} u }{2} 
\right).
\]
Now it is easy to see that if $u$ is linear as above, $\symgrad \Wrt u = 0$ and hence the choice $w = \Wrt u$ renders $\TGVat(u) $ to be zero. Conversely, $\TGVat(u) = 0$ implies $w = \Wrt u$ and hence $\symgrad \Wrt u = 0$. It is easy to see that $\delta_{x-}\delta_{x+} u =   \delta_{y-}\delta_{y+} u = 0$ implies that $u$ is of the form 
$ u_{i,j} = rij + ai + bj + c $
with $r,a,b,c \in U$. But in this case, the last component of $\symgrad \Wrt u$ being zero implies that $ r = \delta_{y-}\delta_{x+} u = -\delta_{x-}\delta_{y+} u = -r $, hence $r=0$.
\end{proof}
\end{prop}
The latter proposition shows that, also after discretization, the kernel of $\TGVat$ consists exactly of (discrete) affine functions.
In  fact,  this  is one of the fundamental properties of $\TGV$ (corresponding to \textbf{(P3)}) which  should  also  be  transferred  to  an  appropriate  generalization  of TGV for manifolds.
Regarding appropriate counterparts of \textbf{(P1)} and \textbf{(P2)}, we introduce the following definition.

\begin{defn} \label{def:discreete_tv_tv2} Using the notation of Definition \ref{def:discrete_tgv_vector_space}, we define $\Wrt^ 2u = \symgrad Du$ and
\[ \TV(u) = \|\Wrt u\|_1, \quad \text{and} \quad  \TV^ 2(u) = \|\Wrt^ 2u\|_1. \]

\end{defn}
In summary, using the discretized version of $\TV$, $\TV^2$ and $\TGVat$ and the discrete notion of affine functions as in Proposition \ref{prop:discrete_vector_space_kernel}, we can conclude that the properties \textbf{(P1)} to \textbf{(P3)} of TGV in continuous vector spaces also transfer to its discretization.
   
\subsection{Univariate $\TGVat$ on manifolds}\label{sec:TGVdefuni}    
When moving from the vector space to the manifold setting, a main difference is that vectors representing derivatives can no longer be made independent of their location, but are attached to a base point on the manifold. In other words, in the Euclidean setting, all tangent spaces at all locations can be identified, which is not possible for manifolds in general.
Accordingly, when aiming to extend $\TGV$ to the manifold setting, important questions are how to introduce the vector fields $(w_{i,j})_{i,j}$ appearing in the minimum representation, how to define $\symgrad w $ for such elements, 
and, most importantly, how to define a suitable distance-type function between such tangent vectors sitting in different points.

In the following, we describe our main ideas to resolve these questions.
In order to allow the reader to understand the underlying ideas more easily, we consider the univariate setting first.
Let $u = (u_i)_i$ be a finite sequence of points in a manifold $\M$ with distance $d:\M\times \M \rightarrow [0,\infty)$. 
Our first goal is to suitably extend the notion of forward differences $(\delta_{x+}u)_i = u_{i+1} - u_i$ and introduce auxiliary variables $w_i$ which they can be compared to. To this end, the central idea is to identify tangent vectors $w_i$ in the space $T_{u_i}\M$ with point-tuples, i.e., $w_i \simeq [u_i,y_i]$ with $y_i \in \M$. In the vector-space case, this can be done via the correspondence $w_i = y_{i} - u_i$. For manifolds, the correspondence can be established via the exponential and the logarithmic map. That is, any $w_i \in T_{u_i}\M$ can be assigned to a unique point-tuple $[u_i,y_i]$ such that $\exp_{u_i}(w_i) = y_i$. Conversely, any point-tuple $[u_i,y_i]$ can be assigned to (generally multiple) tangent vectors $w_i$ such that $w_i \in \log_{u_i}(y_i)$.
 (Note that the ambiguities in logarithmic map result from non-uniqueness of distance-minimizing geodesics, which is a rather degenerate situation in the sense that it only occurs for a set of points with measure zero \cite{itoh1998dimension,cheeger1975comparison}.) Figure \ref{fig:point_tuple_example_univariate} visualizes the correspondence between tangent vectors $w$ and point-tuples $[u,y]$.

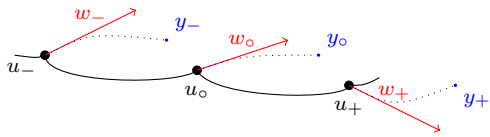
\begin{figure}
\footnotesize
\center
\begin{tikzpicture}[scale=0.2]

\draw (-12,1) .. controls (-10.5,0.8) and (-10.5,0.8) .. (-10,1) .. controls (-10,-1) and (-1,-1) .. (0,0) .. controls (1,-2) and (10,-2) .. (10,-1)  .. controls (11,-1) ..  (12,-0.5);

\def\cw{8pt} 
\filldraw (-10,1) circle (\cw) node[black,below left]{$u_{-}$};
\filldraw (0,0) circle (\cw) node[black,below=0.1cm]{$u_{\circ}$};
\filldraw (10,-1) circle (\cw) node[black,below=0.1cm]{$u_{+}$};

\def\cw{2pt} 
\filldraw[blue] (-2,2) circle (\cw) node[blue,above right]{$y_{-}$};
\filldraw[blue] (8,1) circle (\cw) node[blue,above right]{$y_{\circ}$};
\filldraw[blue] (17,-1) circle (\cw) node[blue,below right]{$y_{+}$};

\draw[red,->] (-10,1) -- (-4,4) node [midway, above] {$w_{-}$};
\draw[red,->] (0,0) -- (6,2) node [midway, above] {$w_{\circ}$};
\draw[red,->] (10,-1) -- (16,-4) node [midway, above] {$w_{+}$};

\draw[blue,dotted] (-10,1) .. controls (-7,2.5) .. (-2,2);
\draw[blue,dotted] (0,0) .. controls (3,1) .. (8,1);
\draw[blue,dotted] (10,-1) .. controls (13,-2.5) .. (17,-1);

\end{tikzpicture}
\caption{A three-point section of a signal $u$ together with tangent vectors $w$ represented by the endpoints $y$. The red vectors show the tangent vectors at the signal points, the blue dotted lines indicate the geodesics $t \mapsto \exp_u(tw)$ and the $y$ are the endpoints $\exp_u(w)$. The black lines indicate geodesic interpolations between the signal points and are for visualization purposes only.\label{fig:point_tuple_example_univariate}}
\end{figure}

Exploiting this correspondence, our approach is to work with a discrete ``tangent bundle'' that is defined as the set of point-tuples $[u,y]$ with $u,y \in \M$, rather than with the continuous one. We refer to the elements in this discrete ``tangent bundle'' as tangent tuples. 
The identification of forward differences with a tangent tuple is then naturally given via $(\delta_{x+}u)_i = [u_i,u_{i+1}]$.

The discretized vector-space version of second-order TGV (for which there is a one-to-one correspondence of tangent vectors and point tuples) can then in the univariate case be rewritten as
\begin{align*}
 \TGVat(u)  
& =  \min _{(w_i)_i } \sum _{i} \alpha _1 \big| (\delta_{x+} u)_i - w_i \big|  + \alpha _0 \big| (\delta_{x-}w)_{i} \big| \\
& =  \min _{ (y_i)_i } \sum _{i} \alpha _1 D([u_i,u_{i+1}],[u_i,y_i])  + \alpha _0 D([u_i,y_i] , [u_{i-1},y_{i-1}]),
 \end{align*}
 where we define $D([x,y],[u,v]) = \big |[x,y] - [u,v]\big | $ and $[x,y] - [u,v] = (y-x) - (v-u)$. Note that, in this context, $(w_i)_i$, $(y_i)_i$ always denote finite sequences of points with their length being the same as the one of $\delta_{x+}u$ and, whenever out of bound indices are accessed in such a summation, we assume the signals to be extended such that the evaluation of $D$ yields zero cost.
 
Hence, in order to extend TGV to the manifold-setting, we are left to appropriately measure the distance of two tangent tuples, i.e., generalize the expression $D([x,y],[u,v]) = |[x,y] - [u,v]|$. In case both tuples have the same base-point on the manifold, a simple and rather natural generalization is to measure their distance by the distance of their endpoints, that is, we set $D([x,y],[x,v]) = d(y,v)$. 
In the other case, i.e., for evaluating $D([u_i,y_i] , [u_{i-1},y_{i-1}])$ for $u_i \neq u_{i-1}$, there is no such simple or even unique generalization since the two tangent tuples belong to different (tangential) spaces.
A quite direct approach to overcome this is to first employ the concept of parallel transport on manifolds to shift both tangent tuples to the same location, and then to measure the distance of their endpoints. 
One possible generalization for TGV proposed in this paper builds on this idea. 
This is not the only possible way of measuring the distance of two tangent tuples.
Alternatively, one can for instance build on a discrete approximation of parallel transport, i.e., on a variant of the Schild approximation,
to measure distance of point tuples at different locations.
This second approach is simpler to realize numerically, and it can be build on simpler differential-geometric tools, while the first one is more straight-forward.
Both approaches have their motivation and, given 
this ambiguity in possible generalizations, we take a more systematic approach to the topic. That is, we first formulate fundamental properties that we require from reasonable generalizations of TGV to the manifold setting.
Then we translate these properties to requirements on admissible distance-type functions for measuring the distance of two tangent tuples. Having established the later, we propose two possible concrete realizations that fulfill the required properties.

\noindent \textbf{Axiomatic extension.} Assuming $D:\M^2 \times \M^2 \rightarrow [0,\infty)$ to be a distance-type function for tangent tuples, we define a generalization of second-order TGV for univariate manifold-valued data as
\begin{equation} \label{eq:tgv_manifold_general}
\MTGV(u)  
 =  \min _{(y_i)_i  } \sum _{i} \alpha _1 d(u_{i+1},y_i)  + \alpha _0 D( [u_i,y_i],[u_{i-1},y_{i-1}]).
\end{equation} In this context, we note that, while with TGV for discrete and continuous vector spaces it is known that the minimum in the above expression is attained, this is not clear a-priori for our generalization. In order not to lose focus, we shift the discussion on existence to Section \ref{sec:existence_results}, noting at this point that for all versions of $\MTGV$ proposed in this work, existence can be shown.

Accounting for the fact that the tangent tuples represent vectors in the tangent space, some basic identifications and requirements for the distance-type function $D$ follow naturally. Zero elements in the discrete ``tangent bundle'' for instance correspond to tangent tuples of the form $[u,u] $ with $u \in \M$ and the distance of $[x,x]$ and $[y,y]$ should be zero.
Also, the distance of two identical elements should also be zero, i.e., $D([x,y],[x,y]) = 0$, and, in the vector space case, the distance function should reduce to the difference of the corresponding tangent vectors.

Our goal is now to further restrict possible choices of distance-type function in order to obtain a meaningful generalization of TGV. To this aim, we want to ensure that appropriate counterparts of the properties \textbf{(P1)} to \textbf{(P3)} are preserved for the resulting version of $\MTGV$ also in the manifold setting. In order to make this more precise, we first have to generalize the involved concepts. We start with the notion of ``affine''.

Given that the geodesics in a manifold play the role of straight lines in vector space, a natural generalization for $u=(u_i)_i $ a finite sequence in $ \M $ to be affine is to require that all points $(u_i)_i$ are on a single geodesic at equal distance. A difficulty that arises in connection with this definition is that, in general, as opposed to the vector space case, a geodesic connecting two points is not necessarily unique. As a consequence, even though all points might be at the same distance when following a joint geodesic, the distance of each single pair of points in the manifold is not necessarily equal. To account for that, we require in addition that the geodesic connecting all points also realizes the shortest connection between all involved points on the geodesic locally.

\begin{defn} \label{def:geodesic_function} Let $u = (u_i)_{i=0}^ n$ be a finite sequence of points on a manifold $\M$. 
We say that $u$ is geodesic if there exists a geodesic $\gamma:[0,L] \rightarrow \M$ parametrized with unit speed such that $\gamma(i L/n) = u_i$ for $i=0,\ldots,n$ and $d(u_i,u_{i+1}) =  L/n$ for $i=0,\ldots,n-1$, i.e., $\gamma |_{[iL/n,(i+1)L/n]}$ is a geodesic of minimal length connecting $u_i$ and $u_{i+1}$ for all $i$. 
\end{defn}

A second issue arising from non-uniqueness of geodesics is the fact that, even though every triplet of points in a signal $(u_i)_i$ might be connected by a geodesic, we cannot conclude that all points are connected by a unique geodesic. Consequently, as a reasonable generalization of TGV will typically act local, in particular will be based on three point stencils, we cannot hope to obtain more than a local assertion for signals in the zero set of TGV. To account for that, we introduce the following notion.
\begin{defn} \label{def:locally_geodesic_function} Let $u = (u_i)_{i=0}^ n$ be a finite sequence of points on a manifold $\M$. We say that  $u$ is locally geodesic if for each $i \in \{2,\ldots,n-1\}$, the sequence $(u_{j})_{j=i-1}^{i+1}$ is geodesic.
\end{defn}

As one might expect, in the situation that subsequent points of a signal are sufficiently close such that they admit a unique connecting geodesic we obtain equivalence of the notions of geodesic and locally geodesic signals.
In this respect, we again note that, in the general non-local situation, the existence of unique minimizing geodesics  
is true for any set of points outside a set of measure zero \cite{itoh1998dimension,cheeger1975comparison}.

\begin{lem} \label{lem:locally_globally_geodesic} If a sequence of points $u = (u_i)_i $ in $ \M$ is such that
the length minimizing geodesic connecting each pair $u_i$, $u_{i'}$ with $|i-i'|\leq 1$ is unique, then $u$ is locally geodesic if and only if it is geodesic.
\begin{proof}
If $u$ is geodesic, it is obviously locally geodesic. Now assume that $u$ is locally geodesic and that
the length minimizing geodesic connecting each pair $u_i$, $u_{i'}$ with $|i-i'|\leq 1$ is unique.
Then there exists a geodesic connecting $ u_0$, $ u_1$, $ u_2$ at equal distance, i.e., the distance between two subsequent points equals the length of the geodesic segment. We proceed recursively: Now assume that $ u_{i-2}$, $ u_{i-1}$, $ u_{i}$ are connected by a geodesic at equal distance. As $u$ is locally geodesic also $ u_{i-1}$, $ u_{i}$, $ u_{i+1}$ are connected by a geodesic. Now by uniqueness, the two geodesics between $u_i$ and $u_{i-1}$ must coincide. Hence all points $u_{i-2}$ till $u_{i+1}$ are on the same geodesic at equal speed. Proceeding iteratively, the result follows.
\end{proof}
\end{lem}

In order to investigate the counterparts of properties \textbf{(P1)} and \textbf{(P2)}, we need to define generalizations of the $\TV$ and $\TV^2$ functionals to the manifold setting. For $\TV$, a natural generalization is to set, for $u=(u_i)_i$ in $\M$, 
\begin{equation} \label{eq:tv_manifold_univariate}
\TV(u) = \sum\nolimits_i d(u_{i+1},u_i) ,
\end{equation} 
see also \cite{weinmann2014total}.
For $\TV^2$, a generalization is not that immediate. In \cite{bavcak2016second}, second-order TV was essentially (and up to a constant) defined as $\TV^2(u) = 2\sum_i \inf_{c \in [u_{i-1},u_{i+1}]_{\frac{1}{2}} } d(c,u_{i}), $
which, in the vector space setting reduces to 
\begin{equation} \label{eq:tv2_univariate_vector_space}
\TV^2(u) = 2\sum\nolimits_i d(\tfrac{u_{i-1}+u_{i+1}}{2},u_{i}) = \sum\nolimits_i \big|u_{i-1} -2u_i + u_{i+1} \big|. 
\end{equation} 
However, as can be deduced by considering $\TV^2$ as special case of the different versions of TGV proposed below, this is not the only generalization which fulfills such a property. We will call a function an {\em admissible generalization of $\TV^2$} whenever, in the vector space setting, it reduces to $\TV^ 2$ as in Equation \eqref{eq:tv2_univariate_vector_space} above.
 
Using these prerequisites, we now formulate our requirements on an appropriate generalization of $\TGV$ to the manifold setting of the form \eqref{eq:tgv_manifold_general}. 
\begin{itemize}
\item[\textbf{(M-P1)}] In the vector space setting, $\MTGV$ reduces to the univariate version of \eqref{eq:tgv_discrete_vector_space}.
\item[\textbf{(M-P2)}] If the minimum in \eqref{eq:tgv_manifold_general} is attained at $(y_i)_i = (u_i)_i$, i.e., the tangent tuples $[u_i,y_i]$ all correspond to zero vectors, then $\MTGV(u) = \alpha_1 \TV(u)$ with TV as in \eqref{eq:tv_manifold_univariate}.
\item[\textbf{(M-P3)}] If the minimum in \eqref{eq:tgv_manifold_general} is attained at $(y_i)_i = (u_{i+1})_i$, i.e., the tangent tuples $[u_i,y_i]$ all correspond to $(\delta_{x+} u)_i$, then $\MTGV(u) = \alpha_0 \TV^ 2(u)$ with $\TV^ 2$ an admissible generalization of $\TV^ 2$.
\item[\textbf{(M-P4)}] $\MTGV(u) = 0$ if and only if $u$ is locally geodesic according to Definition \ref{def:locally_geodesic_function}.
\end{itemize}

As the following Proposition shows, the Properties \textbf{(M-P1)} to \textbf{(M-P3)} already follow from basic requirements on the distance function that arise from the intuition that tangent tuples represent tangent vectors. The most restrictive property is \textbf{(M-P4)}, which requires an additional assumption on the distance function.

\begin{prop} \label{prop:mp1_to_mp4_hold_general} Assume that the function $D:\M^2 \times \M^2 \rightarrow [0,\infty)$ is such that
\begin{itemize}
\item $D([x,x],[u,u]) = 0$ for any $x,u \in \M$,
\item $D([x,y],[u,v]) = \big| (y-x) - (v-u)\big|$ in case $\M = \R^K$.
\end{itemize}
Then, for $\MTGV$ as in \eqref{eq:tgv_manifold_general}, the properties \textbf{(M-P1)} to \textbf{(M-P3)} hold. If we further assume that for any geodesic three-point signal $(v_j)_{j=i-1}^{i+1}$ it follows that $D([v_i,v_{i+1}],[v_{i-1},v_i]) = 0$, then $\MTGV(u) = 0$ for any locally geodesic signal $u = (u_i)_i$. Conversely, assume that for any three-point signal $(v_j)_{j=i-1}^{i+1}$, such that the geodesic connecting each pair $v_i$, $v_{i'}$ with $|i-i'|\leq 1 $ is unique, $D([v_i,v_{i+1}],[v_{i-1},v_i]) = 0$ implies that $(v_j)_{j=i-1}^{i+1}$ is geodesic. Then, for any signal $(u_i)_i$ where the geodesic connecting each pair $u_i$, $u_{i'}$ with $|i-i'|\leq 1$ is unique, $\MTGV(u) = 0$ implies that $u$ is geodesic.
\end{prop}
\begin{proof}
In the vector space case we get by our assumptions and since $d(x,u) = |x-u|$, that
\begin{equation} 
\begin{aligned}
\MTGV(u)  
 = \min _{(y_i)_i } &\sum\nolimits _{i} \alpha _1 \big|u_{i+1} - u_i - ( y_i - u_i)\big| + \alpha _0 \big|(y_i - u_i) - (y_{i-1} - u_{i-1})\big| \\
  =\min _{(w_i)_i } &\sum\nolimits _{i} \alpha _1 \big| (\delta_{x+} u)_i - w_i\big| + \alpha _0 \big|(\delta_{x-}w)_i  \big|
 \end{aligned}
\end{equation}
which coincides with the univariate version of $\TGV$ as in \eqref{eq:tgv_discrete_vector_space}, hence \textbf{(M-P1)} holds.

Now in case the minimum in \eqref{eq:tgv_manifold_general} is achieved for $(y_i)_i = (u_i)_i$ we get that
\[ \MTGV(u) = \sum\nolimits_i\alpha_1 d(u_{i+1},u_i) + \alpha_0 D([u_i,u_i],[u_{i-1},u_{i-1}]) = \alpha_1 \TV(u) \]
and \textbf{(M-P2)} holds. Similarly, if the minimum in \eqref{eq:tgv_manifold_general} is achieved for $(y_i)_i = (u_{i+1})_i$ we get that
\begin{equation} \label{eq:properties_proof_case_y_i_u_ip}
 \MTGV(u) = \sum\nolimits_i \alpha_0 D([u_i,u_{i+1}],[u_{i-1},u_{i}]) 
 \end{equation}
which is an admissible generalization of $\TV^2$ since, by assumption, $D([u_i,u_{i+1}],[u_{i-1},u_{i}]) = \big| u_{i+1}-2u_i + u_{i-1}\big|$ in case $\M = \R^K$. Hence \textbf{(M-P3)} holds.

Now in case $(u_i)_i$ is locally geodesic, we can choose $(y_i)_i = (u_{i+1})_i $ to estimate $\MTGV$ from above with the right-hand side in \eqref{eq:properties_proof_case_y_i_u_ip} and obtain that $\MTGV(u) =0 $. Conversely, in case $\MTGV(u)=0$ we get that necessarily $(y_i)_i = (u_{i+1})_i $ and $\MTGV$ reduces to the right hand side in \eqref{eq:properties_proof_case_y_i_u_ip}. By the assumption on $D$, this implies that $u$ is locally geodesic and by Lemma \ref{lem:locally_globally_geodesic} the result follows.
\end{proof}

Guided by the aim of fulfilling the requirements of Proposition \ref{prop:mp1_to_mp4_hold_general}, in the following, we propose two possible choices for $D(\cdot,\cdot)$ which result in two different concrete versions of $\MTGV$.
Both variants follow the general idea of first transporting different tangent tuples to the same location in $\M$ and then measuring the distance there. The main difference between the two variants will be the way the transport is carried out.
The first concept is motivated by the so called Schild's ladder approximation of parallel transport. Its implementation requires only basic differential geometric concepts and is therefore presented first. The second realization, which we call the parallel transport variant, requires more differential geometric concepts 
\cite{spivak1975differential,do1992riemannian} and is therefore presented afterwards. However, the Schild's ladder variant can be seen as an approximation of the parallel transport variant as explained below.

\noindent \textbf{Realization via Schild's-approximation.}
Let $[x,y]$ and $[u,v]$ be two tuples in $\M^ 2$ for which we want to define $D([x,y],[u,v])$ and assume for the moment that distance-minimizing geodesics are unique. 
Motivated by the Schild's ladder approximation of parallel transport \cite{schild72_ladder,kheyfets2000schild}, 
we consider the following construction (see Figure \ref{fig:schilds_ladder}): take $c$ to be the midpoint of the points $v$ and $x$, i.e., $c = [x,v]_{\frac{1}{2}};$ set $y' = [u,c]_2$, i.e., reflect $u$ at $c$. 
Then, we claim that the distance-type function $D([x,y],[u,v]) = d(y,y')$ fulfills the requirements of Proposition \ref{prop:mp1_to_mp4_hold_general}.

The above construction may be motivated as follows. 
To compare the tangent vectors $\log_x(y)$ and $\log_u(v)$, which correspond to the tangent tuples $[x,y]$ and $[u,v]$, we could move them to the same point using parallel transport
and then take the norm induced by the Riemannian metric in the corresponding tangent space. Using the above construction to obtain $y'$ from $x$ and $[u,v]$, the tangent vector $\log_x(y')$ can be seen as an approximation of the parallel transport of $\log_u(v)$ to $x$ (in the limit $x \rightarrow u$) \cite{schild72_ladder,kheyfets2000schild}. The difference of this vector to $\log_x(y)$ is then approximated by $d(y,y')$.
We notice that, as can be easily seen, this Schild's ladder approximation of parallel transport, i.e., the construction of $y'$ and $\log_x(y')$ as above, is exact in the vector space case.

\begin{figure}
\footnotesize
\center
\begin{tikzpicture}[scale=0.2]

\def\cw{8pt}

\filldraw (0,0) circle (\cw) node[left]{$u$};
\filldraw[blue] (20,0) circle (\cw) node[right,blue]{$v$};
\filldraw (0,10) circle (\cw) node[left]{$x$};
\filldraw[blue] (20,13) circle (\cw) node[right,blue]{$y$};

\draw[blue,dotted] (0,0) .. controls (4,-1) and (8,-2) .. (20,0);
\draw[blue,dotted] (0,10) .. controls (4,12) and (8,14) .. (20,13);
\draw (0,0) .. controls (-2,3) and (-2,5) .. (0,10);

\filldraw[darkgreen] (8.5,3.5) circle (\cw) node[darkgreen,above right]{$c$};
\filldraw[darkgreen] (19,8) circle (\cw) node[darkgreen,above right]{$y'$};
\draw[dotted,darkgreen] (0,10) .. controls (4,6) and (8,2) .. (20,0);
\draw[dotted, darkgreen] (0,0) .. controls (6,4) and (8,2) .. (19,8);

\draw[red,->] (0,0) -- (8,-2) node[below] {$w= \log_u(v)$};
\draw[red,->] (0,10) -- (7,8.5) node[above right=0cm] {$\log_x(y') \approx \pt_x(w)$};

\draw[darkgreen,dotted] (0,10) .. controls (5,9) and (7,8) .. (19,8);

\end{tikzpicture}
\caption{Approximate parallel transport of $\log_u(v)$ to $x$ via the Schild's ladder construction. \label{fig:schilds_ladder}}
\end{figure}
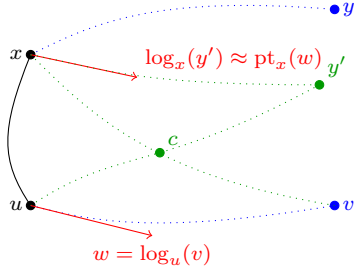

Hence we propose to measure the deviation between the tangent tuples $[x,y]$ and $[u,v]$ by transporting $[u,v]$ to $x$ using the construction above to obtain the tangent tuple $[x,y']$ and then to compare the resulting tangent tuples sitting in the same point $x$ by measuring the distance of their endpoints $y,y'$. Since in general geodesics are not unique, we have to minimize over all possible constructions as above, i.e., consider all midpoints of all length-minimizing geodesics. This yields the following distance-type function
\begin{equation} \label{eq:schilds_distance}
 \ds([x,y],[u,v]) = \min _{y' \in \M} d(y',y) \text{ such that } y' \in [u,c]_2 \text{ with } c \in [x,v]_{\frac{1}{2}}.
\end{equation}
It is immediate that $\ds$ is positive and zero for identical elements. Furthermore if $\ds([x,y],[u,v])= 0$ then $[x,y]$ can be interpreted to be equal to the approximate parallel transport of $[u,v]$ to $x$. Indeed, in the vector space case,  $\ds([x,y],[u,v])= 0$ if and only if $v-u$ is parallel to $y-x$ and has the same length and direction.

Plugging in $\ds$ in \eqref{eq:tgv_manifold_general} we define
\begin{equation} \label{eq:tgv_univariate_schilds}
\begin{aligned}
\STGV(u)  
 = \min _{(y_i)_i } &\sum\nolimits _{i} \alpha _1 d(u_{i+1},y_i) + \alpha _0 \ds\big([u_i,y_{i}],[u_{i-1},y_{i-1}]\big). \\
\end{aligned}
\end{equation}
Regarding the properties \textbf{(M-P1)} to \textbf{(M-P4)}, we then get the following result.

\begin{thm}\label{thm:d_s_properties} The $\STGV$ functional as in \eqref{eq:tgv_univariate_schilds} satisfies \textbf{(M-P1)} and \textbf{(M-P3)}. If a finite sequence of points $u=(u_i)_i$ is locally geodesic, then $\STGV(u) = 0$. Further, if $\STGV(u) = 0$ and 
the geodesic connecting each pair $u_i$, $u_{i'}$ with $|i-i'|\leq 1$ is unique, then $u$ is geodesic.
\begin{proof} We verify the assumptions of Proposition \ref{prop:mp1_to_mp4_hold_general} for $D = \ds$.
It can be easily seen that $\ds([x,x],[u,u]) = \ds([x,y],[x,y]) = 0$ for $x,u,y \in \M$. Also, in the case $\M = \R ^K$, the approximate parallel transport as in the definition of $\ds$ above coincides with the parallel shift of vectors. Hence $\ds([x,y],[u,v]) = \ds( [0,y-x],[0,v-u]) = \big|(y-x) - (v-u)  \big| $ and \textbf{(M-P1)} to \textbf{(M-P3)} holds. 
Now if $v = (v_j)_{j=i-1}^ {i+1}$ is locally geodesic, then $v_i \in [v_{i},v_{i}]_{\frac{1}{2}}$ and $v_{i+1} \in [v_{i-1},v_i]_2$, hence $\ds([v_i,v_{i+1}],[v_{i-1},v_i]) = 0.$
Conversely, $\ds([v_i,v_{i+1}],[v_{i-1},v_i]) = 0$ implies that there exists $y' \in [v_{i-1},v_i]_2$ such that $d(y',v_{i+1}) = 0$. But this implies that $v_{i+1} \in [v_{i-1},v_i]_2$ and, by the assumption on unique geodesics, it follows that $v$ is geodesic.
\end{proof}
\end{thm}
\begin{rem}
Assuming that, for each $i$, $u_i$ and $y_i$ are sufficiently close such that they are connected by a unique length minimizing geodesic, we get that $[u_i,[u_i,y_i]_{\frac{1}{2}}]_2 = y_i$ and hence $ \ds([u_i,u_{i+1}],[u_i,y_i]) = d(u_{i+1},y_i)$. Consequently, in this situation, an equivalent definition of $\STGV $ can be given as
\begin{equation*} 
\begin{aligned}
\STGV(u)  
 = \min _{(y_i)_i } &\sum\nolimits _{i} \alpha _1 \ds([u_i,u_{i+1}],[u_i,y_i]) + \alpha _0 \ds\big([u_i,y_{i}],[u_{i-1},y_{i-1}]\big) \\
\end{aligned}
\end{equation*}
This definition regards the mapping $(u_i)_i\mapsto ([u_i,u_{i+1}])_i$ as discrete gradient operator, mapping from $\M$ to the discrete tangent space, and exclusively works in the discrete tangent space with $\ds$ the canonical distance-type function.
\end{rem}

\begin{rem} We note that an alternative choice for $\ds$ would have been to transport $[u_{i},y_{i}]$ to $u_{i-1}$ rather than $[u_{i-1},y_{i-1}] $ to $u_i$. However, since the distance $d(u_{i+1},y_i)$ in $\STGV$ as in \eqref{eq:tgv_univariate_schilds} can be interpreted as evaluating the difference of the forward difference $(\delta_{x+}u)_i $ with $[u_i,y_{i}]$ in the center point $u_i$, it seems natural to evaluate also the backward difference of the signal $([u_i,y_i])_i$ at the center point. 
\end{rem}

\begin{rem} We also note that, as opposed to the vector space setting, the distance function $\ds$ is in general not symmetric, i.e., $\ds([x,y],[u,v]) \neq \ds([u,v],[x,y])$. 
 To obtain symmetry, alternative definitions of $\ds$ could be given as
$ \tilde{\ds}([x,y],[u,v]) = \ds([x,y],[u,v]) + \ds([u,v],[x,y]) $
or
\[ \tilde{\ds}([x,y],[u,v]) =   \min _{c_1,c_2} \, d(c_1,c_2) \quad \text{subject to } c_1 \in  [x,v]_{\frac{1}{2}} \text{ and } c_2 \in  [u,y]_{\frac{1}{2}}.\]
For the sake of simplicity and in order to obtain the relation with parallel transport, however, we have defined $\ds$ as in \eqref{eq:schilds_distance}.
\end{rem}

\noindent \textbf{Realization via parallel transport.}
The Schild's ladder construction defined above can be seen as a discrete approximation of parallel transport. 
Alternatively, we can use the identification of point-tuples $[u,v],[x,y]$ with vectors in the tangent space at $u$ and $x,$ respectively, and use the parallel transport directly to transport the respective vectors to a common base point.

That is, assuming -- for the moment -- uniqueness of length-minimizing geodesics, we can identify each $[u,v]$ with $w \in T_u \M $ such that $v = \log_u (w)$ and compare $\pt_x(w)$, the vector resulting from the parallel transport of $w$ to $T_x \M$, to $\log_x(y)\in T_x \M$. This yields a distance-type function for two point-tuples $[u,v]$, $[x,y]$ as
\begin{equation}\label{eq:DefDpt}
	\dpt ([x,y],[u,v]) = \big| \log_x(y) - \pt_x(\log_u(v))\big|_x
\end{equation}
where $\big|\cdot\big|_x$ denotes the norm in $T_x\M$. 
Comparing this to $\ds$ we note that, besides using a different notion of transport, we now measure the distance with the norm in the tangent space rather than with the distance of endpoints. The reason for doing so is that, in this situation, approximating the norm in the tangent space via the distance of endpoints does not yield further simplification since the the parallel transport forces us to work in the tangent space anyway.

Now in general, length minimizing geodesics are not necessarily unique. To deal with this issue, we have defined the $\log$ mapping and the parallel transport to be set-valued (see Section \ref{sec:notation}), i.e., for $u,v \in \M$, not necessarily close, and $w \in T_v\M  $ we define $\pt_u(w) \subset T_u\M $ to be the set of all vectors in $T_u\M$ which can be obtained by parallel-transporting $w$ to $T_u\M$ along a length minimizing geodesic. Note that by isometry of the parallel transport, the length of all such vectors is the same but by holonomy their orientation might be different. 
In order to adapt the above definition of $\dpt{}$ to this situation, we generalize the distance function $\dpt{}$ for tangent tuples $[x,y]$ and $[u,v]$ as
\begin{equation} \label{eq:definition_transport_distance}
\dpt([x,y],[u,v]):= \min_{\substack{z_1 \in \log_x(y) \\ z_2 \in T_x\M} } \big|z_1 - z_2\big|_{x} \text{ such that } z_2 \in \pt_x(w) \text{ with } w \in \log_u(v).
\end{equation}
Using this notation, a parallel-transport-based version of $\MTGV$, denoted by $\PTTGV$, can then be defined for the general situation as
\begin{equation} 
\PTTGV(u)  
 =  \min _{(y_i)_i } \sum\nolimits _{i} \alpha _1 d(u_{i+1},y_i) + \alpha _0 \dpt( [u_i,y_i],[u_{i-1},y_{i-1}]) .
\end{equation}
Similarly to the Schild's-version, we then get the following result.
\begin{thm} \label{thm:pt_properties} The functional $\PTTGV$ satisfies \textbf{(M-P1)} and \textbf{(M-P3)}. If a finite sequence of points $u=(u_i)_i$ is locally geodesic, then $\PTTGV(u) = 0$. Further, if $\PTTGV(u) = 0$ and the geodesic connecting each pair $u_i$, $u_{i'}$ with $|i-i'|\leq 1$ is unique, then $u$ is geodesic.
\proof The proof is similar to the one of Theorem \ref{thm:pt_properties_multi} below considering the bivariate setting.
\end{thm}

\begin{rem} Since the expression $d(u_{i+1},y_i)$ can be seen as approximation of $|\log_{u_i}(u_{i+1}) - w_i|_{u_i}$, where $w_i \in \log_{u_i}(y_i)$, an alternative definition of $\PTTGV$ (assuming uniqueness of geodesics for simplicity) only in terms of tangent vectors is given as
\begin{equation} 
\widetilde{\PTTGV}(u)  
 =  \min _{w_i \in T_{u_i}\M} \sum\nolimits _{i} \alpha _1 \big|\log_{u_i}(u_{i+1})  -  w_i \big|_{u_i} + \alpha _0 \big| w_i - \pt_{u_i}(w_{i-1})\big|_{u_i} .
\end{equation}
We believe, however, that the originally proposed version is preferable since the fact that the first term in $\PTTGV$ only involves the standard manifold distance of two points simplifies the numerical realization. That is, with the standard distance we are able to solve certain subproblems (proximal mappings) in the algorithm explicitly whereas the other version would require additional inner loops (see Section \ref{sec:algorithm} for details).
\end{rem}

\subsection{Bivariate $\TGVat$ on manifolds}
\label{sec:axomatic_extension_bivariate}
The goal of this section is to extend $\MTGV$ to the bivariate case. To this aim, we first observe that the bivariate version of $\TGVat$ for vector spaces as in \eqref{eq:tgv_discrete_vector_space} can be written as
\begin{multline} \label{eq:tgv_vector_manifold_general_multi}
 \TGVat(u)  = \min _{(w_{i,j}^1)_{i,j},(w_{i,j}^2)_{i,j}}  \sum\nolimits _{i,j} \alpha_1 \bigg( \big | (\delta_{x+}u)_{i,j}- w_{i,j}^ 1\big |^p + \big | (\delta_{y+}u)_{i,j}- w_{i,j}^ 2\big |^p  \bigg)^{1/p} \\
  +\alpha_0 \bigg( \big | w^1 _{i-1,j} - w^1_{i,j}\big |^p + \big | w^2 _{i,j-1} - w^2_{i,j}\big |^p  + 2^{1-p} \big | (w^1 _{i,j-1} - w^1_{i,j}) + (w^2 _{i-1,j} - w^2_{i,j})\big |^p \bigg) ^{1/p},
\end{multline}
where we set all norms $|\cdot |^p$ to be zero whenever out of bound indices are involved.

The first four summands in the above definition of $\TGVat$ can all be transferred to the manifold setting using the previously introduced distance-type functions for tangent tuples. The additional difficulty of the bivariate situation arises form the fifth term, which first combines both differences and sums of tangent tuples and then measures the norm. Again possible generalizations are not unique and we employ an axiomatic approach to propose reasonable choices.

\noindent \textbf{Axiomatic extension.} Denote by $D(\cdot,\cdot)$ one of the two previously introduced distance functions for tangent tuples and assume that $D^{\text{sym}} : \M^ 2 \times \M^ 2\times \M^ 2\times \M^ 2 \rightarrow \R$ generalizes the fifth term in \eqref{eq:tgv_vector_manifold_general_multi}, which corresponds to the mixed derivatives. A bivariate version of $\MTGV$ can then be given as
\begin{equation} \label{eq:tgv_manifold_general_multi}
\begin{aligned}
\text{M-TGV}_\alpha^ 2(u)  
 = \min _{y^ 1_{i,j},y^ 2_{i,j}}  &\alpha_1 \sum\nolimits _{i,j}  \Big(  d(u_{i+1,j},y^ 1_{i,j})^p + d(u_{i,j+1},y^ 2_{i,j})^p \Big) ^{1/p}\\
  +& \alpha _0 \sum\nolimits_{i,j}  \Big( D\big([u_{i,j},y^1_{i,j}],[u_{i-1,j},y^1_{i-1,j}]\big)^p + D\big([u_{i,j},y^2_{i,j}],[u_{i,j-1},y^2_{i,j-1}]\big)^p  \\
  &\qquad  + 2^{1-p} D^{\text{sym}} ([u_{i,j},y^1_{i,j}],[u_{i,j},y^2_{i,j}],[u_{i,j-1},y^1_{i,j-1}],[u_{i-1,j},y^2_{i-1,j}])^p \Big)^{1/p}.
\end{aligned}
\end{equation}
The basis for this generalization is again the representation of tangent vectors with tangent tuples, only that now for each $u_{i,j}$ we consider two tangent vectors $w^1_{i,j}, w^ 2_{i,j}$ and corresponding points $y^1_{i,j}, y^ 2_{i,j}$ in order to represent horizontal and vertical derivatives, see Figure \ref{fig:point_tuples_bivariate}.
Remember that in this paper, we focus on the case $p=1,$ for which we derive an algorithmic realization later on. However, since it causes no additional effort in this section, we provide a formulation for all
$p \in [1,\infty)$ here.

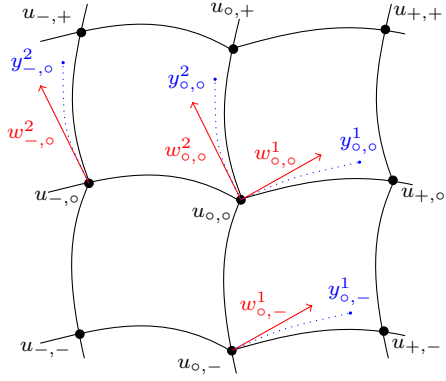
\begin{figure}
\footnotesize
\center
\begin{tikzpicture}[scale=0.2]

\foreach \x [count=\xi] in {-10,0,10}
{
	\pgfmathsetmacro\result{-0.05*abs(\x)}
	\draw (-12+\result,\x-1) ..controls (-8+\result,-3+\x+3) and (-4+\result,-1.8+\x+3) .. (0.5+\result,-1.5+\x);
	\draw (0.5+\result,-1.5+\x) ..controls (4+\result,-3+\x+2.5) and (7.5+\result,-2.5+\x+3) .. (12+\result,\x-0.5);
	\pgfmathsetmacro\result{-0.1*abs(\x)}
	\draw (\x+0.4,-13-\result) ..controls (\x-1,-7-\result) and (\x-0.5,-4-\result) .. (0.8+\x,-1.2-\result) ;
	\draw (\x+0.8,-1.5-\result) ..controls (\x-1,-7-\result+10) and (\x-0.5,-4-\result+10) .. (0.6+\x,10.5-\result);
}

\def\cw{8pt} 
\filldraw (0.7,-1.5) circle (\cw) node[black,below left]{$u_{\circ,\circ}$};
\filldraw (10.7,-0.2) circle (\cw) node[black,below right]{$u_{+,\circ}$};
\filldraw (-9.3,-0.4) circle (\cw) node[black,below left]{$u_{-,\circ}$};

\filldraw[yshift=10cm,xshift=-0.5cm] (0.7,-1.5) circle (\cw) node[black,above=0.3cm]{$u_{\circ,+}$};
\filldraw[yshift=10cm,xshift=-0.5cm] (10.7,-0.2) circle (\cw) node[black,above right]{$u_{+,+}$};
\filldraw[yshift=10cm,xshift=-0.5cm] (-9.3,-0.4) circle (\cw) node[black,above left]{$u_{-,+}$};

\filldraw [yshift=-10cm,xshift=-0.6cm] (0.7,-1.5) circle (\cw) node[black,below left]{$u_{\circ,-}$};
\filldraw [yshift=-10cm,xshift=-0.6cm] (10.7,-0.2) circle (\cw) node[black,below right]{$u_{+,-}$};
\filldraw [yshift=-10cm,xshift=-0.6cm] (-9.3,-0.4) circle (\cw) node[black,below left]{$u_{-,-}$};

\draw[red,->] (0.7,-1.5) -- (6,1.5) node [left=0.2cm] {$w_{\circ,\circ}^1$};
\draw[red,->,yshift=-10cm,xshift=-0.6cm] (0.7,-1.5) -- (6,1.5) node [left=0.2cm] {$w_{\circ,-}^1$};

\draw[red,->] (0.7,-1.5) -- (-2.5,5) node [midway, left] {$w_{\circ,\circ}^2$};
\draw[red,->,yshift=1.1cm,xshift=-10cm] (0.7,-1.5) -- (-2.5,5) node [midway, left] {$w_{-,\circ}^2$};

\def\cw{2pt} 
\filldraw[blue] (8.5,1) circle (\cw) node[blue,above]{$y^1_{\circ,\circ}$};
\draw[blue,dotted] (0.7,-1.5) .. controls (4,0) .. (8.5,1);
\filldraw[blue,yshift=-10cm,xshift=-0.6cm] (8.5,1) circle (\cw) node[blue,above]{$y^1_{\circ,-}$};
\draw[blue,dotted,yshift=-10cm,xshift=-0.6cm] (0.7,-1.5) .. controls (4,0) .. (8.5,1);

\filldraw[blue] (-1,6.5) circle (\cw) node[blue,left]{$y^2_{\circ,\circ}$};
\draw[blue,dotted] (0.7,-1.5) .. controls (-1,2) .. (-1,6.5);
\filldraw[blue,yshift=1.1cm,xshift=-10cm] (-1,6.5) circle (\cw) node[blue,left]{$y^2_{-,\circ}$};
\draw[blue,dotted,yshift=1.1cm,xshift=-10cm] (0.7,-1.5) .. controls (-1,2) .. (-1,6.5);

\end{tikzpicture}
\caption{A three-by-three point section of a bivariate signal $u$ together with tangent vectors $w$ represented by endpoints $y$. The blue, dotted lines indicate the geodesics $t \mapsto \exp(tw)$ connecting the signal points $u$ with the endpoints $y$. The black lines indicate a piecewise geodesic interpolation of the signal points and are for visualization purposes only.
\label{fig:point_tuples_bivariate}}
\end{figure}

We now extend the requirements \textbf{(M-P1)} to \textbf{(M-P4)} to the bivariate case.
The generalization of TV to bivariate manifold-valued data is quite straightforward \cite{weinmann2014total}. For $u = (u_{i,j})_{i,j}$ we define
\begin{equation} \label{eq:tv_manifold_multi}
 \TV(u) = \sum\nolimits_{i,j} \Big( d(u_{i+1,j},u_{i,j})^p  + d(u_{i,j+1},u_{i,j})^p \Big)^{1/p} .
\end{equation} 

A generalization of second-order $\TV$ is again less straightforward and we call any functional $\TV^2$ acting on $u = (u_{i,j})_{i,j}$ an admissible generalization of $\TV^2$ if it reduces, in the vector space setting, to $\TV^2 $ as given in Definition \ref{def:discreete_tv_tv2}.
Note that our version of $\TV^2$ differs from the definition of $\TV^ 2$ as given in \cite{bavcak2016second} since we use a symmetrization of the mixed derivatives. 

A generalization of affine functions for the bivariate manifold setting is given as follows.

\begin{defn} \label{def:locally_globally_geodesic_function_bivariate} Let $u = (u_{i,j})_{i,j}$ be a finite sequence of points on a manifold $\M$. We say that  $u$ is (locally) geodesic if, for each $(i_0,j_0)$, the univariate signals $ (u_{i,j_0})_{i} $, $ (u_{i_0,j})_{j}$ and $(u_{i_0+k,j_0-k})_{k}$ are (locally) geodesic. 
\end{defn}

Note that this indeed generalizes the notion of affine for vector space data. While the first two conditions ensure that $u$ is of the form $u_{i,j} = aij + b i + c j + d$, the last condition ensures that $a = 0$, i.e., no mixed terms occur. In the vector space case the third condition is equivalent to requiring that $(u_{i_0+k,j_0+k})_{k}$ is (locally) geodesic. In the manifold setting, this is not equivalent in general and hence our definition of geodesic is somewhat anisotropic. The reason for using the anti-diagonal ($(i_0 + k,j_0-k)$) and not the diagonal ($(i_0 + k,j_0+k)$) direction is that the former arises naturally from the usage of forward-backward differences, as will become clear in the discussion after Proposition \ref{prop:mp1_to_mp4_hold_general_bivariate}. 
An alternative would also be to require both $(u_{i_0+k,j_0+k})_{k}$ and $(u_{i_0+k,j_0-k})_{k}$ to be geodesic. This, however, seems rather restrictive and for general manifolds such signals might even not exist. Hence we do not impose this additional restriction.

As a direct consequence of the corresponding result in the univariate setting, we obtain equivalence of the notion of locally geodesic and geodesic if neighboring points are sufficiently close.
\begin{lem} \label{lem:locally_globally_geodesic_bivariate} Let $u = (u_{i,j})_{i,j} $ be a finite sequence in $ \M$. If the length minimizing geodesic connecting any two points $u_{i,j}$ and $u_{i',j'}$ with $\max\{|i-i'|,|j-j'|\} \leq 1$ is unique, then $u$ is locally geodesic if and only if it is geodesic.
\end{lem}
With these prerequisites, we now state our requirements for a reasonable generalization of TGV in the bivariate case.
\begin{itemize}
\item[\textbf{(M-P1')}] In the vector space setting, $\MTGV$ reduces to vector-space TGV as in  \eqref{eq:tgv_discrete_vector_space}.
\item[\textbf{(M-P2')}] If the minimum in \eqref{eq:tgv_manifold_general_multi} is attained at $(y^ 1_{i,j})_{i,j} =(y^ 2_{i,j})_{i,j} =  (u_{i,j})_{i,j}$, i.e., the tangent tuples all correspond to zero vectors, then $\MTGV(u) = \alpha_1 \TV(u)$ with TV as in \eqref{eq:tv_manifold_multi}.
\item[\textbf{(M-P3')}] 
If the minimum in \eqref{eq:tgv_manifold_general_multi} is attained at $(y^ 1_{i,j})_{i,j} = (u_{i+1,j})_{i,j}$ and $(y^ 2_{i,j})_{i,j} = (u_{i,j+1})_{i,j}$, i.e., the tangent tuples $[u_{i,j},y^1_{i,j}]$ and $[u_{i,j},y^2_{i,j}]$ all correspond to $(\delta_{x+} u)_{i,j}$ and $(\delta_{y+} u)_{i,j}$, respectively, then $\MTGV(u) = \alpha_0 \TV^ 2(u)$ with $\TV^ 2$ an admissible generalization of $\TV^ 2$.
\item[\textbf{(M-P4')}] $\MTGV(u) = 0$ if and only if $u$ is locally geodesic according to Definition \ref{def:locally_globally_geodesic_function_bivariate}.
\end{itemize}
Those properties translate to requirements for the involved distance-type functions as follows.
\begin{prop} \label{prop:mp1_to_mp4_hold_general_bivariate} Assume that the function $D:\M^2 \times \M^2 \rightarrow [0,\infty)$ satisfies the assumptions of Proposition \ref{prop:mp1_to_mp4_hold_general} and assume that $\dsym$ is such that
$\dsym([x,x],[x,x],[u,u],[u,u]) = 0$ for any $x,u \in \M$ and, 
in case $\M = \R^K$,
\begin{multline*}
\dsym ([\ucc,\ycco],[\ucc,\ycct],[\ucm,\ycmo],[\umc,\ymct])  \\ =  \big | \ycco - \ucc - (\ycmo - \ucm) + \ycct - \ucc - (\ymct - \umc)\big |.
\end{multline*}
Then, for $\MTGV$ as in \eqref{eq:tgv_manifold_general}, the properties \textbf{(M-P1')} to \textbf{(M-P3')} hold. If further 
\[\dsym ([v_{i,j},v_{i+1,j}],[v_{i,j},v_{i,j+1}],[v_{i,j-1},v_{i+1,j-1}],[v_{i-1,j},v_{i-1,j+1}]) = 0,\]
for any geodesic three-by-three signal $v= (v_{i,j})_{i,j}$, then $\MTGV(u) = 0$ for any locally geodesic $u$. 
Conversely, assume that for any three-by-three signal in $v= (v_{i,j})_{i,j}$ where the geodesic connecting each pair $v_{i,j}$, $v_{i',j'}$ are unique and where $(v_{i,j})_i$ and $(v_{i,j})_j$ are geodesic, it holds that
\[\dsym ([v_{i,j},v_{i+1,j}],[v_{i,j},v_{i,j+1}],[v_{i,j-1},v_{i+1,j-1}],[v_{i-1,j},v_{i-1,j+1}]) = 0,\]
implies also $(v_{i+k,j-k})_{k=-1}^1$ being geodesic. Then, for any $u= (u_{i,j})_{i,j}$
such that the geodesic connecting each pair $u_{i,j}$ and $u_{i',j'}$ with $\max\{|i-i'|,|j-j'|\} \leq 2$ is unique, we get that $\MTGV(u) = 0$ implies $u$ being geodesic. 
\proof See Section \ref{sec:proofs_section2} in the Appendix. \qedhere
\end{prop}

Similar to the univariate case, the most restrictive requirement for $\dsym$ is that, in case of uniqueness of length-minimizing geodesics, for $(u_{i,j})_{i,j}$ being locally geodesic in horizontal and vertical direction, it holds that 
\begin{multline}
\dsym ([u_{i,j},u_{i+1,j}],[u_{i,j},u_{i,j+1}],[u_{i,j-1},u_{i+1,j-1}],[u_{i-1,j},u_{i-1,j+1}]) = 0  \\ 
\Longleftrightarrow(u_{i+k,j-k})_k \text{ is locally geodesic}.
\end{multline}

Let us discuss, again temporarily assuming uniqueness of geodesics, a general strategy to design the function $\dsym$ such that this property holds. 
We consider the situation around a point $\ucc$ and denote by $w^1_{\circ,\circ},w^2_{\circ,\circ},w^1_{\circ,-},w^2_{-,\circ}$ four tangent vectors corresponding to the four tangent tuples $[\ucc,\ycco],$ $[\ucc,\ycct],$ $[\ucm,\ycmo],$ $[\umc,\ymct]$ at which $\dsym $ is evaluated, see Figure \ref{fig:point_tuples_bivariate}. The evaluation of $\dsym$ at these points should, in the vector space case, correspond to one of the equivalent formulations
\[ \big | w^1_{\circ,\circ}- w^1_{\circ,-} + w^2_{\circ,\circ}- w^2_{-,\circ} \big | = \big | \big( w^1_{\circ,\circ} + w^2_{\circ,\circ} \big) - \big( w^1_{\circ,-}  + w^2_{-,\circ}\big) \big | = \big | \big( w^1_{\circ,\circ}- w^1_{\circ,-}\big)  -\big(   w^2_{-,\circ} - w^2_{\circ,\circ}  \big) \big |. \]
Disregarding for the moment the fact that the corresponding tangent tuples live on different locations, the difficulty here is how to define algebraic operations, i.e., sum or difference operations, on two tangent tuples. For the sum, there is no direct way of doing so, 
the difference operation, however, quite naturally transfers to tangent tuples as
\[ [x,y_1] - [x,y_2] = (y_1 - x) - (y_2 - x) = y_1 - y_2  = [y_2,y_1]. \]
Hence, in the situation that all tangent tuples live on the same location, we can define $\dsym$ by first evaluating pair-wise differences of the four point tuples and then measuring the distance of the two resulting tuples by $D(\cdot,\cdot)$. 

Now in the general situation, we first need to transport tangent tuples to the same location. As described in Section \ref{sec:TGVdefuni}, we have means of doing this that build either on parallel transport or its Schild's approximation.
Taking this into account, we note that the evaluation of $\dsym$ in the vector space case can be written purely in terms of distances of tangent-vector-differences in various ways, e.g.,
\begin{equation} \label{eq:reformulations_of_symgrad}
\big | w^1_{\circ,\circ}- w^1_{\circ,-} + w^2_{\circ,\circ}- w^2_{-,\circ} \big | =  \big | \big(w^1_{\circ,\circ}- w^1_{\circ,-}\big) -\big(w^2_{-,\circ}- w^2_{\circ,\circ} \big) \big | = \big | \big(w^1_{\circ,-} - w^2_{\circ,\circ}\big) - \big( w^1_{\circ,\circ}  - w^2_{-,\circ}    \big)   \big |.
\end{equation}
While any order of taking the differences is equivalent in the vector space case, in the manifold setting the order defines how the corresponding tangent tuples need to be transported to the same location. Looking again at Figure \ref{fig:point_tuples_bivariate}, the simplest idea seems to be to transport the tangent tuple $[u_{\circ,-},y_{\circ,-}^1]$ corresponding to $w_{\circ,-}^ 1$ and the tangent tuple $[u_{-,\circ,-},y_{-,\circ}^2]$ corresponding to $w_{-,\circ}^ 2$ to the point $\ucc$ and carry out all operations there.

The drawback of this simple solution can be found when looking at the conditions of Proposition \ref{prop:mp1_to_mp4_hold_general_bivariate}. There, the main difficulty is to ensure that \[\dsym ([\ucc,\upc],[\ucc,\ucp],[\ucm,\upm],[\umc,\ump]) = 0\] allows to conclude that the signal is geodesic in the anti-diagonal direction. 
The situation that is relevant for this condition is when both $(u_{i,j_0})_i$ and $(u_{i_0,j})_j$ are locally geodesic for each $(i_0,j_0)$ and the endpoints $y$ coincide with the respective signal points, i.e., $y_{i,j}^1 = u_{i+1,j}$ and $y_{i,j}^2 = u_{i,j+1}$ (see Figure \ref{fig:dsym_motivation}, left).
In this situation, in order to fulfill the above condition, we need to know something about the transported tangent tuples, e.g., 
if $[u_{\circ,-},y_{\circ,-}^1]$ and $[u_{-,\circ},y_{-,\circ}^2]$ corresponding to $w_{\circ,-}^1$ and $w_{-,\circ}^2$ are both transported to $\ucc$, we would need to know for instance that the tangent tuple $[\ucc,x]$ resulting from the transport of $[u_{\circ,-},y_{\circ,-}^1]$ to $\ucc$ points to $\upc$, i.e., $x = \upc$. Due to holonomy however, even if the transport is carried out along a geodesic, there is, to the best of our knowledge, no way of obtaining such a result. That is, the vector corresponding to the transported tangent tuple $[\ucc,x]$ might be arbitrarily rotated such that $x$ is far away from $\upc$. 

On the other hand, a well-known fact for manifolds is that the parallel transport of a vector that is tangential to a geodesic along the geodesic itself again results in a vector that is tangential to the geodesic. As we will see, this implies an equivalent assertion for both our variants for transporting tangent tuples. That is, the transport of a tangent tuple corresponding to a tangential vector of a geodesic along the geodesic itself again results in a tangent tuple that corresponds to a tangential vector of the geodesic.
In the above-described particular situation (see again Figure \ref{fig:dsym_motivation}, left), this means for example that, since the points $\ucp,\ucc,\ucm$ are on a geodesic and $w^ 2_{\circ,\circ}\simeq [u_{\circ,\circ},u_{\circ,+}]$ corresponds to a tangent vector that is tangential to this geodesic, we know that the transport of $[u_{\circ,\circ},u_{\circ,+}]$ to $\ucm$ will result in the tangent tuple $[\ucm,\ucc]$.
More generally, for the situation that both $(u_{i,j_0})_i$ and $(u_{i_0,j})_j$ are locally geodesic for each $(i_0,j_0)$, it means that we can always transport the tangent tuples corresponding to $w_{\circ,\circ}^1$ and $w_{\circ,-}^1$ in horizontal direction and the tangent tuples corresponding to the $w_{\circ,\circ}^ 2 $ and $w_{-,\circ}^ 2 $ in vertical direction and still know something about the transported tangent tuples.

In view of \textbf{(M-P4')}, a natural approach to define $\dsym$ in the general situation is hence to restrict ourselves to these particular transport directions.
Of course, we also want to define $\dsym$ in an as-simple-as-possible way, meaning that we want to carry out as few transport operations as possible. A quick case study in Figure 
\ref{fig:point_tuples_bivariate}, shows that we have to transport at least three times and that, accounting for the facts that we need to generalize the expression \eqref{eq:reformulations_of_symgrad} and that we can take differences but not sums of tangent tuples in the same location, there is only one possibility to achieve this. As highlighted in Figure \ref{fig:dsym_motivation} on the right for the general situation, our proposed approach is to transport the tangent tuple corresponding to $w^ 2_{\circ,\circ}$ down to $\ucm$ giving $[\ucm, \tilde y_{\circ,\circ}^2]$, the tangent tuple corresponding to $w^1_{\circ,\circ}$ left to $\umc$ giving $[\umc, \tilde y_{\circ,\circ}^1]$ and take the differences, which results in the tangent tuples $[\ycmo,\tilde{y}_{\circ,\circ}^2] \simeq d^2$ and $ [\tilde{y}_{\circ,\circ}^1,\ymct] \simeq d^1$. The distance of the tangent tuples corresponding to $d^1$ and $d^2$ can then be measured with one of our previously defined distance functions, i.e., by evaluating $D([\tilde{y}_{\circ,\circ}^1,\ymct],[\ycmo,\tilde{y}_{\circ,\circ}^2])$, which again requires one transport operation. This corresponds to the reformulation of 
\[\big | w^1_{\circ,\circ}- w^1_{\circ,-} + w^2_{\circ,\circ}- w^2_{-,\circ} \big | = \big | \big( w^2_{\circ,\circ} -w^1_{\circ,-} \big)  - \big(  w^2_{-,\circ} -  w^1_{\circ,\circ}   \big)   \big |. \]
In the particular situation of Figure \ref{fig:dsym_motivation}, left, due to the usage of the particular transport directions, the two tangent tuples resulting from the transport of the tangent tuples corresponding to $w^ 1_{\circ,\circ}$ and $w^ 2_{\circ,\circ}$ and taking the corresponding difference operations will be $[\ucc,\ump]$ and $[\upm,\ucc]$ (corresponding to $d^1 $ and $d^2$), whose distance will then be measured by $D([\ucc,\upm],[\ump,\ucc])$. Using the assumptions as in Proposition \ref{prop:mp1_to_mp4_hold_general}, $D([\ucc,\upm],[\ump,\ucc]) = 0$ then allows us to conclude that $\ump,\ucc,\upm$ are on a distance minimizing geodesic, hence the assumptions for \textbf{(M-P4')} will be fulfilled. 
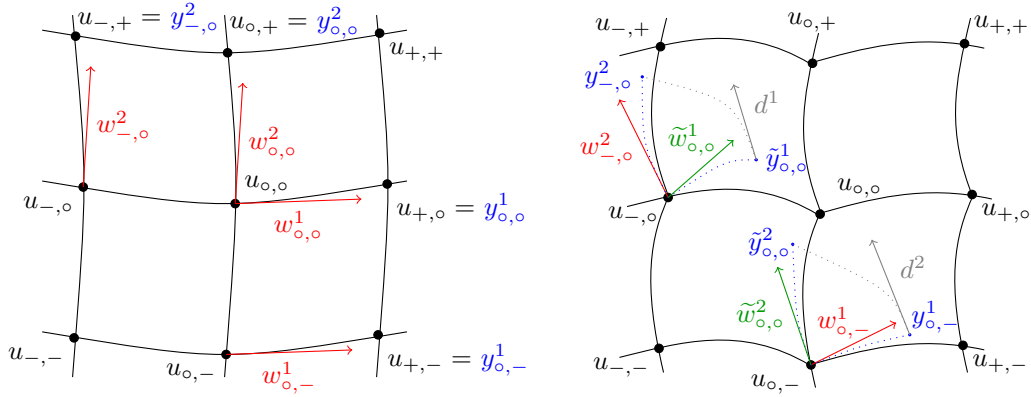
\begin{figure}
\begin{tikzpicture}[scale=0.2]

\foreach \x [count=\xi] in {-10,0,10}
{
	\draw (-12,\x) .. controls (0,-2+\x) .. (12,\x);
	\draw (\x,-13) .. controls (1+\x,0) ..  (\x,11);
}

\def\cw{8pt} 
\filldraw (0.7,-1.5) circle (\cw) node[black,above right]{$u_{\circ,\circ}$};
\filldraw (10.7,-0.2) circle (\cw) node[black,below right]{$u_{+,\circ} =$ \textcolor{blue}{$\ycco$}};
\filldraw (-9.3,-0.4) circle (\cw) node[black,below left]{$u_{-,\circ}$};

\filldraw[yshift=10cm,xshift=-0.5cm] (0.7,-1.5) circle (\cw) node[black,above=0.4cm, right=-0.1cm]{$u_{\circ,+} = $ \textcolor{blue}{$\ycct$}};
\filldraw[yshift=10cm,xshift=-0.5cm] (10.7,-0.2) circle (\cw) node[black,below right]{$u_{+,+}$};
\filldraw[yshift=10cm,xshift=-0.5cm] (-9.3,-0.4) circle (\cw) node[black,above right=-0.1cm]{$u_{-,+}=$  \textcolor{blue}{$\ymct$}};

\filldraw [yshift=-10cm,xshift=-0.6cm] (0.7,-1.5) circle (\cw) node[black,below left]{$u_{\circ,-}$};
\filldraw [yshift=-10cm,xshift=-0.6cm] (10.7,-0.2) circle (\cw) node[black,below right]{$u_{+,-}=$ \textcolor{blue}{$\ycmo$}};
\filldraw [yshift=-10cm,xshift=-0.6cm] (-9.3,-0.4) circle (\cw) node[black,below left]{$u_{-,-}$};

\draw[red,->] (0.7,-1.5) --node [below] {$w_{\circ,\circ}^1$} (9,-1.2) ;
\draw[red,->,yshift=-10cm,xshift=-0.6cm] (0.7,-1.5) --node [below] {$w_{\circ,-}^1$} (9,-1.2);

\draw[red,->] (0.7,-1.5) -- (1.2,6.5) node [midway, right] {$w_{\circ,\circ}^2$};
\draw[red,->,yshift=1.1cm,xshift=-10cm] (0.7,-1.5) -- (1.2,6.5) node [midway, right] {$w_{-,\circ}^2$};

\def\cw{2pt} 

\end{tikzpicture}
\hspace*{0.2cm}
\begin{tikzpicture}[scale=0.2]

\foreach \x [count=\xi] in {-10,0,10}
{
	\pgfmathsetmacro\result{-0.05*abs(\x)}
	\draw (-12+\result,\x-1) ..controls (-8+\result,-3+\x+3) and (-4+\result,-1.8+\x+3) .. (0.5+\result,-1.5+\x);
	\draw (0.5+\result,-1.5+\x) ..controls (4+\result,-3+\x+2.5) and (7.5+\result,-2.5+\x+3) .. (12+\result,\x-0.5);
	\pgfmathsetmacro\result{-0.1*abs(\x)}
	\draw (\x+0.4,-13-\result) ..controls (\x-1,-7-\result) and (\x-0.5,-4-\result) .. (0.8+\x,-1.2-\result) ;
	\draw (\x+0.8,-1.5-\result) ..controls (\x-1,-7-\result+10) and (\x-0.5,-4-\result+10) .. (0.6+\x,10.5-\result);
}

\def\cw{8pt} 
\filldraw (0.7,-1.5) circle (\cw) node[black,above right=0.1cm]{$u_{\circ,\circ}$};
\filldraw (10.7,-0.2) circle (\cw) node[black,below right]{$u_{+,\circ}$};
\filldraw (-9.3,-0.4) circle (\cw) node[black,below left]{$u_{-,\circ}$};

\filldraw[yshift=10cm,xshift=-0.5cm] (0.7,-1.5) circle (\cw) node[black,above=0.3cm]{$u_{\circ,+}$};
\filldraw[yshift=10cm,xshift=-0.5cm] (10.7,-0.2) circle (\cw) node[black,above right]{$u_{+,+}$};
\filldraw[yshift=10cm,xshift=-0.5cm] (-9.3,-0.4) circle (\cw) node[black,above left]{$u_{-,+}$};

\filldraw [yshift=-10cm,xshift=-0.6cm] (0.7,-1.5) circle (\cw) node[black,below left]{$u_{\circ,-}$};
\filldraw [yshift=-10cm,xshift=-0.6cm] (10.7,-0.2) circle (\cw) node[black,below right]{$u_{+,-}$};
\filldraw [yshift=-10cm,xshift=-0.6cm] (-9.3,-0.4) circle (\cw) node[black,below left]{$u_{-,-}$};

\draw[red,->,yshift=-10cm,xshift=-0.6cm] (0.7,-1.5) -- (6.3,1.3) node [left=0.2cm] {$w_{\circ,-}^1$};

\draw[red,->,yshift=1.1cm,xshift=-10cm] (0.7,-1.5) -- (-2.5,5) node [midway, left] {$w_{-,\circ}^2$};

\def\cw{2pt} 
\filldraw[blue,yshift=-10cm,xshift=-0.6cm] (7.2,0.5) circle (\cw) node[blue,above right=-0.1cm]{$y^1_{\circ,-}$};
\draw[blue,dotted,yshift=-10cm,xshift=-0.6cm] (0.7,-1.5) .. controls (4,-0.2) .. (7.2,0.5);
\filldraw[blue,yshift=-10cm,xshift=-0.6cm] (-0.5,6.5) circle (\cw) node[blue,left=-0.1cm]{$\tilde{y}^2_{\circ,\circ}$};
\draw[blue,dotted,yshift=-10cm,xshift=-0.6cm] (0.7,-1.5) .. controls (-0.3,1.5) .. (-0.5,6.5);

\filldraw[blue,yshift=1.1cm,xshift=-10cm] (-1,6.5) circle (\cw) node[blue,left]{$y^2_{-,\circ}$};
\draw[blue,dotted,yshift=1.1cm,xshift=-10cm] (0.7,-1.5) .. controls (-1,2) .. (-1,6.5);
\filldraw[blue,yshift=1.1cm,xshift=-10cm] (6.5,1) circle (\cw) node[blue,right]{$\tilde{y}^1_{\circ,\circ}$};
\draw[blue,dotted,yshift=1.1cm,xshift=-10cm] (0.7,-1.5) .. controls (4.5,1) .. (6.5,1);

\draw[darkgreen,->,yshift=-10cm,xshift=-0.6cm] (0.7,-1.5) -- (-1.5,5) node [midway, left] {$\widetilde{w}_{\circ,\circ}^2$};
\draw[darkgreen,->,yshift=1.1cm,xshift=-10cm] (0.7,-1.5) -- (5,2.3)node [left=0.1cm] {$\widetilde{w}_{\circ,\circ}^1$} ;

\draw[gray,->,yshift=1.1cm,xshift=-10cm] (6.5,1) -- node [above right] {$d^1$} (5,6)  ;
\draw[gray,dotted,yshift=1.1cm,xshift=-10cm] (6.5,1) .. controls (5,5) .. (-1,6.5);
\draw[gray,->,yshift=-10cm,xshift=-0.6cm]  (7.2,0.5)  -- node [above right] {$d^2$} (10.7-6,-0.2+7)  ;
\draw[gray,dotted,yshift=-10cm,xshift=-0.6cm] (7.2,0.5) .. controls (6,3.5) .. (-0.5,6.5);

\end{tikzpicture}

\caption{Two three-by-three point sections of a bivarate signal. Left: Signal points ($u$), tangent vectors (\textcolor{red}{$w$}) and endpoints (\textcolor{blue}{$y$}) in the particular situation that $\MTGV=0$. Right: Signal with tangent vectors (\textcolor{darkgreen}{$\tilde{w}$}) corresponding to transported tangent tuples, transported endpoints (\textcolor{blue}{$\tilde{y}$}) and (\textcolor{gray}{$d$}) corresponding to difference of tangent tuples. 
The dotted lines indicate geodesics that are determinted by tangent vectors and the black lines show a piecewise geodesic interpolation of the signal points. \label{fig:dsym_motivation}}
\end{figure}

The above-described strategy is now used to extend both $\STGV$ and $\PTTGV$ to the bivariate setting. As such, the only difference will be how to carry out the transport operations and which of the two functions $\ds$, $\dpt{}$ is used to measure the distance of tangent tuples.
 
\noindent \textbf{Realization via Schild's-approximation.}
We now realize the axiomatic setting for bivariate TGV for manifold-valued data of
Section~\ref{sec:axomatic_extension_bivariate} using the proposed Schild's approximation.
In particular, we use the Schild's approximation of parallel transport to obtain an instance of the distance-type function $D^{\text{sym}}.$ 

We define the following  bivariate version $\STGV$ of bivariate TGV for manifold-valued data 
using the Schild's approximation by
\begin{equation} \label{eq:tgv_bivariate_schilds}
\begin{aligned}
\text{S-TGV}_\alpha^ 2(u)  
 = \min _{y^ 1_{i,j},y^ 2_{i,j}}  \alpha_1 \sum\nolimits _{i,j} & \Big(   d(u_{i+1,j},y^ 1_{i,j})^p + d(u_{i,j+1},y^ 2_{i,j})^p \Big)^{1/p}\\
  +\alpha _0 \sum\nolimits_{i,j} & \Big( \ds\big([u_{i,j},y^1_{i,j}],[u_{i-1,j},y^1_{i-1,j}]\big)^p + \ds\big([u_{i,j},y^2_{i,j}],[u_{i,j-1},y^2_{i,j-1}]\big)^p  \\
   &+ 2^{1-p}\dss ([u_{i,j},y^1_{i,j}],[u_{i,j},y^2_{i,j}],[u_{i,j-1},y^1_{i,j-1}],[u_{i-1,j},y^2_{i-1,j}])^p \Big)^{1/p}.
\end{aligned}
\end{equation}
Here  $\ds$ is given as in Equation \eqref{eq:schilds_distance} and $\dss$ is defined by
\begin{align}\label{eq:DefSchildSSym}
 \dss ([\ucc,\ycco],[\tilde{u}_{\circ,\circ},\ycct],[\ucm,\ycmo]&,[\umc,\ymct])= \\
 \min\nolimits_{r^ 1,r^2}   \ds([r^ 1,\ymct],[\ycmo,r^ 2]) \quad
& \text{s.t. }  r ^1 \in [\ucc,c^ 1]_2 \text{ with } c^1 \in [\umc,\ycco]_{\frac{1}{2}} \notag \\
& \text{and }r ^2 \in [\tilde{u}_{\circ,\circ},c^ 2]_2 \text{ with } c^2 \in [\ucm,\ycct]_{\frac{1}{2}}. \notag
 \end{align}
We note that, for the sake of a formally correct definition, we have introduced $\tilde{u}_{\circ,\circ}$, but in fact we will always choose $\tilde{u}_{\circ,\circ} = \ucc$ since $\dss$ effectively depends only on seven variables.

Also note that $\dss$ exactly carries out the construction of $\dsym$ as described in Section ~\ref{sec:axomatic_extension_bivariate}, using the Schild's approximation as shown in Figure \ref{fig:schilds_ladder} for transporting tangent tuples and $\ds$ to measure their difference. In the notation of Figure \ref{fig:dsym_motivation}, right, we have that $\tilde{y}_{\circ,\circ}^1 = r^1$ and $\tilde{y}_{\circ,\circ}^2 = r^2$ correspond to the transported endpoints, $[u_{-,\circ},r^1] \simeq \tilde{w}_{\circ,\circ}^1 $ and $[u_{\circ,-},r^2] \simeq \tilde{w}^2_{\circ,\circ}  $ to the transported tangent tuples and $[r^ 1,\ymct] \simeq d^1  $ and $[\ycmo,r^ 2] \simeq d^2 $ to the differences of tangent tuples whose distance is measured with $\ds$.
 Due to our careful design of $\STGV$ we get the following result.

\begin{thm}\label{thm:d_s_properties_multi} The $\STGV$ functional as in \eqref{eq:tgv_bivariate_schilds} satisfies the properties \textbf{(M-P1')} to \textbf{(M-P3')}. If $u = (u_{i,j})_{i,j}$ is locally geodesic, then $\STGV(u) = 0$. If $u$ is such that the geodesic connecting each pair $u_{i,j}$ and $u_{i',j'}$ with $\max\{|i-i'|,|j-j'|\} \leq 2$ is unique and $\STGV(u) = 0$, then $u$ is geodesic.
\proof See Section \ref{sec:proofs_section2} in the Appendix. \qedhere
\end{thm}

\noindent \textbf{Realization via parallel transport.}
Using parallel transport in order to realize a symmetrized gradient as described in Section \ref{sec:axomatic_extension_bivariate} yields a parallel transport version $\PTTGV$ of TGV for bivariate manifold-valued data case as follows.
We let
\begin{equation} \label{eq:tgv_bivariate_pt}
\begin{aligned}
\PTTGV(u)  
 = \min _{y^ 1_{i,j},y^ 2_{i,j}}  \alpha_1 \sum\nolimits _{i,j} & \big( d(u_{i+1,j},y^ 1_{i,j})^p + d(u_{i,j+1},y^ 2_{i,j})^p \Big)^{1/p} \\
  +\alpha _0 \sum\nolimits_{i,j} & \Big( \dpt{}\big([u_{i,j},y^1_{i,j}],[u_{i-1,j},y^1_{i-1,j}]\big)^p + \dpt{}\big([u_{i,j},y^2_{i,j}],[u_{i,j-1},y^2_{i,j-1}]\big)^p  \\
 & + 2^{1-p} \dpts \big([u_{i,j},y^1_{i,j}],[u_{i,j},y^2_{i,j}],[u_{i,j-1},y^1_{i,j-1}],[u_{i-1,j},y^2_{i-1,j}]\big)^p \Big)^{1/p}.
\end{aligned}
\end{equation}
Here $\dpt{}$ is given as in \eqref{eq:definition_transport_distance} and we define
\begin{align}
\dpts ([\ucc,\ycco],[\tilde{u}_{\circ,\circ},\ycct],[\ucm,\ycmo],[\umc,\ymct]) = 
 \min_{r^ 1,r^ 2} \dpt{}([r^1,\ymct],[\ycmo,r^ 2]) \label{eq:DefPtSym} & \\
 \text{s.t. } r^1 \in  \exp (\pt_{\umc}(w^1))  \text{ with } w^ 1 \in \log_{\ucc}(\ycco) \notag & \\
 \text{and }  r^2 \in  \exp (\pt_{\ucm}(w^ 2)) \text{ with } w^ 2 \in \log_{\tilde{u}_{\circ,\circ}}(\ycct) \notag &.
\end{align}
Note that, as for $\dss$, for the sake of a formal correctness, we have introduced $\tilde{u}_{\circ,\circ}$, but in fact we will always choose $\tilde{u}_{\circ,\circ} = \ucc$ since also $\dpts$ effectively depends only on seven variables.

Also, we note again that $\dpts$ realizes exactly the construction of $\dsym$ as discussed in Section~\ref{sec:axomatic_extension_bivariate} (see also Figure \ref{fig:dsym_motivation}), only that now the parallel transport is used to transport tangent tuples and the distance of two tangent-tuple-differences is again measured with $\dpt{}$. Due to our careful construction, the following result follows easily.

\begin{thm}\label{thm:pt_properties_multi} The $\PTTGV$ functional as in \eqref{eq:tgv_bivariate_pt} satisfies the properties \textbf{(M-P1')} to \textbf{(M-P3')}. If $u = (u_{i,j})_{i,j}$ is locally geodesic, then $\PTTGV(u) = 0$. If $u$ is such that the geodesic connecting each pair $u_{i,j}$ and $u_{i',j'}$ with $\max\{|i-i'|,|j-j'|\} \leq 2$ is unique and $\PTTGV(u) = 0$, then $u$ is geodesic.
\proof See Section \ref{sec:proofs_section2} in the Appendix. \qedhere
\end{thm}

\section{Existence results for TGV for manifold-valued data} \label{sec:existence_results}

This section provides existence results for the minimization problem appearing the definition of $\MTGV$ as well as for variational $\MTGV$-regularized denoising. For the sake of brevity, we only consider the bivariate setting and note that the univariate counterpart follows as special case.
We first state the main existence results in a general setting, that builds on lower semi-continuity of the involved distance-type function. The latter is then proven in subsequent lemmata.
\begin{thm}\label{thm:existence_bivariate_general} Assume that $D:\M^2 \times \M^2 \rightarrow [0,\infty)$ and $\dsym:\M^ 2 \times \M^ 2 \times \M^ 2 \times M^ 2 \rightarrow [0,\infty)$ are distance-type functions that are lower semi-continuous. Then, for any $u = (u_{i,j})_{i,j}$ there exist $(y^1_{i,j})_{i,j}$, $(y^2_{i,j})_{i,j}$ in $\M$  that attain the minimum in the definition of $\MTGV$ as in Equation \eqref{eq:tgv_manifold_general_multi}.
Further, $\MTGV$ is lower semi-continuous and for any $(f_{i,j})_{i,j} $ in $\M$ there exists a solution to 
\begin{equation} \label{eq:bivariate_mtgv_denoising}
 \min_{u}\, \MTGV(u) + \sum_{i,j} d(u_{i,j},f_{i,j})^2. 
 \end{equation}
Also, all previously introduced distance-type functions $\ds$, $\dss$, $\dpt{}$, $\dpts{}$ are lower semi-continuous and a solution to \eqref{eq:bivariate_mtgv_denoising} exists if $\MTGV$ is replaced by $\STGV$ or $\PTTGV$. In particular, all the models proposed in this paper have a minimizer.
\begin{proof}
The existence results in the general setting follow by standard arguments and the Hopf-Rinow theorem and we only provide a sketch of proof. Regarding existence for the evaluation of $\MTGV$, we note that all involved terms are positive and, for fixed $u$, any infimizing sequence $((y^1)^n,(y^2)^n)_n$ is bounded due to the first two distance functions involved in $\MTGV$ as in Equation \eqref{eq:tgv_manifold_general_multi}. Hence it admits a convergent subsequence and from lower semi-continuity of all the involved terms, it follows that the limit is a minimizer for the evaluation of $\MTGV$. 

Now take any sequence $(u^n)_n$ converging to some $u$ for which, without loss of generality, we assume that $\liminf_n \MTGV(u^n) = \lim_n \MTGV(u^n)$. We can pick elements $((y^1)^n,(y^2)^n)$ for which the minimum in $\MTGV(u^n)$ is attained. Again by boundedness of $(u^n)_n$ and the first two distance functions in the definition of $\MTGV$, also $((y^1)^n,(y^2)^n)_n$ is bounded and we can extract a subsequence $((y^1)^{n_k},(y^2)^{n_k})_k$ converging to some $(\hat{y}^ 1,\hat{y}^ 2)$. Defining $E(\cdot)$ such that $\MTGV(u) = \inf_{y^1,y^2} E(u,y^1,y^ 2)$, we obtain from lower semi-continuity of all involved terms that
\begin{align*}
\MTGV(u) \leq E(u,\hat{y}^ 1,\hat{y}^ 2) \leq \liminf_k E(u^{n_k},(y^1)^{n_k},(y^2)^{n_k}) 
& = \liminf_k \MTGV(u^{n_k}) 
\end{align*}
and since $\liminf_k \MTGV(u^{n_k}) = \liminf_n \MTGV(u^n)$, $\MTGV$ is lower semi-continuous.

Given lower semi-continuity of $\MTGV$, existence for \eqref{eq:bivariate_mtgv_denoising} follows by similar arguments. Regarding the particular realizations $\STGV$ and $\PTTGV$, it suffices to show lower semi-continuity of the involved distance-type functions, which is done on the sequel.
\end{proof}
\end{thm}
We are now left to show lower semi-continuity of the proposed distance-type functions $\ds$, $\dpt{}$, $\dss$ and $\dpts$. To this aim, we exploit that they are all defined by minimizing a distance-type function subject to a constraint. Using this joint structure, we first provide a standard abstract lower semi-continuity result that covers this setting and reduces lower semi-continuity to a closedness condition on the constraint.

\begin{prop}\label{prop:existence_basis_general} Take $K,N \in \N$ and functions $G:\M^N\times \M^K \rightarrow [0,\infty)$ and $C:\M^N  \rightarrow \mathcal{P}(\M^K)$ such that $G$ is lower semi-continuous, $ C(x)  \neq \emptyset$ for all $x \in \M^N$ and for any bounded sequence $(x^n)_n $ with elements in $\M^N$ and $(y^n)_n $ such that $y^n \in C(x^n)$ for all $n$, there exist subsequences $(x^{n_k})_k$ and $(y^{n_k})_k$  converging to some $x \in \M^N$, $y \in \M^K$ such that $y \in C(x)$.
Define $S: \M^N \rightarrow [0,\infty)$ as
\[ S(x) = \inf_{y \in \M^K} G(x,y) \text{ such that }y \in C(x). \]
Then, for any $x \in \M^N$ there exists $y\in  C(x)$ such that $S(x) = G(x,y)$. Further, $S$ is lower semi-continuous.
\proof See Section \ref{sec:proofs_section3} in the Appendix. \qedhere
\end{prop}
In the application of this result to the proposed distance-type functions, the constraint $y \in C(x)$ will correspond to constraining points to particular positions on distance minimizing geodesics. Consequently, in order to verify the closedness condition, the general result that the set of shortest geodesics connecting two points is closed w.r.t.~perturbation of these points will be necessary and is provided as follows.

\begin{lem} \label{lem:stability_geodesics}Let $\M$ be a (geodesically) complete Riemannian manifold. Let $(p^n)_n$ and $(q^n)_n$ be two sequences in $\M$ converging to $p$ and $q$, respectively. Let each $\gamma^ n :[0,1] \rightarrow \M$ be a distance-minimizing geodesic such that $\gamma^ n(0) = p^n$, $\gamma ^n(1) = q^n$. Then, there exists a distance-minimizing geodesic $\gamma:[0,1] \rightarrow \M$ such that $\gamma(0) = p$, $\gamma(1) = q$ and a subsequence $(\gamma^{n_k})_k$ such that $\gamma^{n_k} \rightarrow \gamma$ uniformly on $[0,1]$. Furthermore, extending the geodesics to any interval $[a,b]$ with $[0,1] \subset [a,b]$ we get that, up to subsequences, $(\gamma^{n_k})_{n_k}$ also converges uniformly to $\gamma$ on $[a,b]$.
\proof See Section \ref{sec:proofs_section3} in the Appendix. \qedhere
\end{lem}

We are now in the position of showing closedness of the constraints appearing in the definitions of $\ds$ and $\dss$.
\begin{lem}\label{lem:shild_transport_stability} Let $((x^n,y^n,u^n,v^n))_n$ be a bounded sequence in $\M^4$ and take $((c^n,\tilde{y}^n))_n$ in $\M^2$ such that
\[ c^n \in [x^n,v^n]_{\frac{1}{2}} \quad \text{and}\quad \tilde{y}^n \in [u^n,c^n]_2. \]
Then there exist subsequences $((x^{n_k},y^{n_k},u^{n_k},v^{n_k}))_k$ and $((c^{n_k},\tilde{y}^{n_k}))_k$ converging to $(x,y,u,v)$ and $(c,\tilde{y})$, respectively, such that
\[ c \in [x,v]_{\frac{1}{2}} \quad \text{and}\quad \tilde{y} \in [u,c]_2. \]
In particular, $C:\M^4\rightarrow \mathcal{P}(\M^2)$ defined as $C(x,y,u,v):= \{ (c',y') \,|\, c' \in [x,v]_{\frac{1}{2}}, y' \in  [u,c']_2 \}$ satisfies the assumption of Proposition \ref{prop:existence_basis_general}.
\begin{proof}
Since each $c^n$ is on a length-minimizing geodesic between $x^n$ and $v^n$, we get that $d(x^n,c^n) \leq d(x^n,v^n)$ and hence $(c^n)_n$ is bounded. Also $d(c^n,\tilde{y}^n) = d(c^n,u^n)$ and hence $(\tilde{y}^n)_n$ is bounded. Consequently we can pick subsequences $((x^{n_k},y^{n_k},u^{n_k},v^{n_k}))_k$ and $((c^{n_k},\tilde{y}^{n_k}))_k$ converging to $(x,y,u,v)$ 
and $(c,\tilde{y})$. Now with $\gamma_c^{n_k}:[0,1]\rightarrow \M$ being a shortest geodesic connecting $x^{n_k}$ and $v^{n_k}$ such that $\gamma_c^{n_k} (1/2) = c^{n_k}$ we get by Lemma \ref{lem:stability_geodesics} that, up to a subsequence, $\gamma_c^{n_k} \rightarrow \gamma_c$ uniformly, where $\gamma_c$ is a again a shortest geodesic connecting $x$ and $v$. Consequently, $c = \gamma(1/2) \in [x,v]_{\frac{1}{2}}$. 

Now pick $\gamma_{\tilde{y}}^{n_k}:[0,1]\rightarrow \M$ to be a length minimizing geodesic between $\gamma_{\tilde{y}}^{n_k}(0) = u^{n_k}$ and $\gamma_{\tilde{y}}^{n_k}(1)= c^{n_k}$ such that $\gamma_{\tilde{y}}^{n_k}(2) = \tilde{y}^{n_k}$. By Lemma \ref{lem:stability_geodesics} we get that, up to subsequences, $\gamma^{n_k}_{\tilde{y}}$ converges uniformly to a geodesic $\gamma_{\tilde{y}}$ on $[0,2]$ such that $\gamma_{\tilde{y}}$ is length-minimizing between $\gamma_{\tilde{y}}(0) = u$ and $\gamma_{\tilde{y}}(1) = c$. By uniform convergence on $[0,2]$ we get that $y = \lim_{k}y^{n_k} = \lim_k \gamma^{n_k}_{\tilde{y}}(2) = \gamma_{\tilde{y}}(2)$ and the assertion follows.
\end{proof}
\end{lem}
Combining this with the general assertion of Proposition \ref{prop:existence_basis_general}, existence and lower semi-continuity results for $\ds$ and $\dss$ follow as direct consequences. 
\begin{lem} \label{lem:shild_ex_lsc}
The minimum in the definition of $\ds$ as in Equation \eqref{eq:schilds_distance} is attained
and $\ds$ is lower semi-continuous. Also, the minimum in the definition of $\dss$ as in Equation \eqref{eq:DefSchildSSym} is attained and $\dss$ is lower semi-continuous.
\proof See Section \ref{sec:proofs_section3} in the Appendix. \qedhere
\end{lem}
Using similar techniques, existence and lower semi-continuity results for the parallel transport variants can be established as follows.
\begin{lem} \label{lem:pt_ex_lsc}The minimum in the definition of $\dpt{}$ as in Equation \eqref{eq:definition_transport_distance} is attained and $\dpt{}$ is lower semi-continuous. Also, the minimum in the definition of $\dpts$ as in Equation  \eqref{eq:DefPtSym} is attained and $\dpts$ is lower semi-continuous.
\proof See Section \ref{sec:proofs_section3} in the Appendix. \qedhere
\end{lem}

\section{Algorithmic approach to TGV for manifold-valued data}\label{sec:algorithm}

In order to algorithmically approach denoising problems using TGV regularization in the manifold setting,
we employ the concept of cyclic proximal point algorithms (CPPAs). 
A reference for cyclic proximal point algorithms in vector spaces is \cite{Bertsekas2011in}.
In the context of Hadamard spaces, the concept of CPPAs was developed by \cite{bavcak2013computing}, where it is used to compute means and medians.
In the context of variational regularization methods for nonlinear, manifold-valued data,
they were first used in \cite{weinmann2014total}. 
More precisely, the reference \cite{weinmann2014total} deals with TV regularization 
as well as classical smooth methods for manifold-valued data.
The CPPA approach was later used for higher-order TV-type methods   
in \cite{bergmann2014second} for circle-valued data and in \cite{bavcak2016second} for manifold-valued data. 

\paragraph{Principle of a CPPA.}
The idea of a CPPAs is to compose a functional $f: \mathcal M \to \mathbb R$ into 
basic atoms $f_i$ and then to compute the proximal mappings of the $f_i$ in a cyclic, iterative way.
More precisely, assume that 
\begin{equation}\label{eq:Decomp4CPPA}
f = \sum\nolimits_{i=1}^n f_i
\end{equation}
and consider the proximal mappings \cite{moreau1962fonctions, ferreira2002proximal, azagra2005proximal} $\prox_{\lambda f_i}: \mathcal{M} \rightarrow \mathcal{M}$ given as
\begin{align} \label{eq:prox_mapping_abstract}
	\prox_{\lambda f_i} x = \argmin_y f_i(y) + \tfrac{1}{2 \lambda} d(x,y)^2.  
\end{align}
One cycle of a CPPA then consists of applying each proximal mapping $\prox_{\lambda f_i}$ once in a prescribed order, e.g., $\prox_{\lambda f_1},$ $\prox_{\lambda f_2},$ $\prox_{\lambda f_3}, \ldots,$ or, generally,
$\prox_{\lambda f_{\sigma(1)}},$ $\prox_{\lambda f_{\sigma(2)}},$ $\prox_{\lambda f_{\sigma(3)}},$ $\ldots,$ 
where the symbol $\sigma$ is employed to denote a permutation.
The cyclic nature is reflected in the fact that the output of  $\prox_{\lambda f_{\sigma(i)}}$ is used as input for $\prox_{\lambda f_{\sigma(i+1)}}.$ 
Since the $i$th update is immediately used for the $(i+1)$st step, it can be seen as a Gauss-Seidel-type scheme.
A CPPA now consists of iterating these cycles, i.e., in the $k$th cycle, we have
\begin{equation}\label{eq:iterationCPPA}
x_{i+1}^{(k)} = \prox_{\lambda_k f_{\sigma(i)}}x_i^{(k)}, 
\end{equation}
and $x_{1}^{(k+1)}$ is obtained by applying $\prox_{\lambda_k f_{\sigma(n)}}$ to $x_{n}^{(k)}.$ Here, $n$ denotes the number of elements in the cycle. 
During the iteration, the parameter $\lambda_k$ of the proximal mappings is successively decreased. 
In this way, the penalty for deviation from the previous iterate is successively increased.
It is chosen in a way such that the sequence $(\lambda^k)_k$ is square-summable but not summable.  
Provided that this condition on the stepsize parameters holds, the cyclic proximal point algorithm can be shown to converge to the optimal solution of the underlying minimization problem, at least in the context of Hadamard manifolds and convex $(f_i)_i$ \cite[Theorem 3.1]{bacak2014convex}.

For general manifolds, the proximal mappings \eqref{eq:prox_mapping_abstract} are not globally defined, and the minimizers are not unique, at least for general possibly far apart points; cf. \cite{ferreira2002proximal, azagra2005proximal}.
This is a general issue in the context of manifolds that are -- in a certain sense -- a local concept
involving objects that are often only locally well defined. In case of ambiguities, we hence consider the above objects as set-valued quantities. Furthermore, we cannot guarantee -- and in fact do not expect --- the second distance-type functions $D$ in the definition of the $\MTGV$ functional to be convex. Hence convergence of the CPPA algorithm to a globally optimal solution cannot be ensured. 
It thus should be seen as an approximative strategy.
Nevertheless, as will be seen in the numerical experiments section, we experience a good convergence behavior of the CPPA algorithm in practice. This was also observed in previous works, where the CPPA algorithm was employed to minimize second-order TV-type functionals \cite{bergmann2014second,bavcak2016second}, which are also non-convex.
In order to be precise, we further point out that we approximately compute the proximal mappings of the distance-type functions $D$; we hence use an inexact proximal point algorithm.

We also point out that a parallel proximal point algorithm has been proposed in 
\cite{weinmann2014total}. Here the proximal mappings of the $f_i$ are computed in parallel and then averaged 
using intrinsic means. In order to reduce the costs for averaging, an 
approximative variant of the parallel proximal point algorithm has been proposed in 
\cite{weinmann2014total} as well.
Principally, the cyclic proximal point algorithm actually applied in this paper might be replaced by this parallel strategy; the derivations in the following provide all information necessary.

\subsection{Algorithms for manifold-valued TGV}\label{sec:AlgBasicUni}

We employ a cyclic proximal point algorithm for the minimization of a TGV-regularized variational approach in the manifold context. For simplicity, we consider the univariate setting first; all aspects discussed in the univariate situation are prototypic and basic for the bivariate situation. Starting with the general version of TGV for manifolds as in \eqref{eq:tgv_manifold_general}, we aim to solve the denoising problem

\begin{equation} \label{eq:tgv_manifoldParallelUnivAlgo} 
 \min_{(u_i)_i,(y_i)_i } \quad \tfrac{1}{2} \sum _{i}  d(u_{i},h_i)^2 + \sum _{i} \alpha _1 d(u_{i+1},y_i)  
 + \sum _{i} \alpha _0 D( [u_i,y_i],[u_{i-1},y_{i-1}]),
\end{equation}
where $(h_i)_i$, being a finite sequence in $\M$, denotes the given data.

We decompose \eqref{eq:tgv_manifoldParallelUnivAlgo} into the atoms 
\begin{align} \label{eq:tgv_manifoldParallelUnivAlgoDecomposition}
\begin{split}
	&g_i(u) := \tfrac{1}{2} d(u_{i},h_i)^2; \qquad g'_i(u,y) := \alpha _1 d(u_{i+1},y_i); 
	\\
	&g''_i(u,y) := \alpha _0 D( [u_i,y_i],[u_{i-1},y_{i-1}]).
\end{split}	
\end{align}
Now we use this decomposition into the atoms $g_i,g'_i, g''_i$ for the decomposition \eqref{eq:Decomp4CPPA} for the CPPA. We apply the iteration \eqref{eq:iterationCPPA} to this decomposition.
We remark that the data terms  $g_i$ are not coupled w.r.t.\ the index $i$. The same applies to the $g'_i$. This allows the parallel computation of the proximal mappings of the $g_i$ for all $i$ and, separately, of the $g'_i$ for all $i$.
The $g''_i$ are actually coupled. However, grouping even and odd indices, we see that the $g''_i,$ $i$ even, are not coupled and that the $g''_i,$ $i$ odd, are not coupled. So we may compute their proximal mappings in parallel. Together, this results in a cycle of length four per iteration, one step for the  $g_i,$ one step for the $g'_i$ and two steps (even, odd) for the $g''_i.$

In the following, the task is to compute the proximal mappings of the atoms $g_i,g'_i, g''_i$ of \eqref{eq:tgv_manifoldParallelUnivAlgoDecomposition}. To this end, we will from now on always assume that the involved points are locally sufficiently close such that there exist unique length minimizing geodesics connecting them. 
We note that this restriction is actually not severe: in the general (non-local) situation we have to remove the cut points -- which are a set of measure zero -- to end up with the corresponding setup. 	

We remark that,
as we will see, an explicit computation of the proximal mapping is possible for $g_i$ and $g_i'$, but not for $g_i''$. We believe that it is an important feature of the proposed definition of manifold-TGV via point-tuples that the proximal mappings of the first term $g'_i$ is still explicit and hence only one part of the overall problem does not allow an explicit proximal mapping. Indeed, also existing generalizations of second-order TV to the manifold setting incorporate one part with non-explicit proximal mappings \cite{bavcak2016second}. Hence, from the algorithmic viewpoint, the step from second-order TV to our proposed version of TGV does not introduce essential additional computational effort.

For the data atoms $g_i,$ the proximal mappings $\prox_{\lambda g_i}$ are given by 
\begin{align}\label{eq:tgv_manifoldParallelUnivAlgoProxData}
(\prox_{\lambda g_i})_{j}(u) =  
\begin{cases}
 [u_{i},h_{i}]_t, & \text{ if } i=j,  \\
 u_{j}, & \text{ else, }  \\
 \end{cases}
 \qquad \text{ where }  t= \tfrac{\lambda}{1+\lambda}.
\end{align}
They have been derived in \cite{ferreira2002proximal}.
We recall that the symbol $[\cdot,\cdot]_t$ denotes the point reached after time $t$
on the (non unit speed) geodesic starting at the first argument reaching the second argument at time $1$.
Further, the proximal mappings of the distance terms  
$g'_i$ have a closed form representation as well (see \cite{weinmann2014total}), which is given by 
\begin{align}\label{eq:tgv_manifoldParallelUnivAlgoProxDist}
\prox_{\lambda g'_i } (u,y)=  
\begin{cases}
 [u_{i+1},y_i]_t,    &\text{at position } (i+1,1),  \\
 [y_i,u_{i+1}]_t,    &\text{at position } (i,2),    \\
 u_j  ,              &\text{at position } (j,1), \ j \neq i+1,  \\   
 y_j                 &\text{at position } (j,2), \  j \neq i,  \\   
\end{cases}
\end{align}
where $t = \lambda \alpha_1 /d(u_{i+1},y_i),$ if $\lambda \alpha_1 < \tfrac{1}{2},$  
and $t = 1/2$ else. Note that for determining the position, we view $u,y$ as column vectors of a matrix with two columns.

The next task is to compute the proximal mappings of the $g''_i.$
Unfortunately, no closed form for the proximal mappings of the $g''_i$ is available.
Instead, we use a subgradient descent scheme to compute these proximal mappings, i.e., to solve the problem 
\begin{equation} \label{eq:ProxMapForSubgr}
 \min_{u,y} \quad \tfrac{1}{2} \sum\nolimits_{j}  d(u_{j},u'_j)^2 + \tfrac{1}{2} \sum\nolimits_{j}  d(y_{j},y'_j)^2 +
\lambda \alpha _0 D( [u_i,y_i],[u_{i-1},y_{i-1}])
\end{equation} 
where the $u'_j,y'_j$ are the input of the proximal mapping. 
A subgradient descent scheme has already been used to compute proximal mappings in \cite{bavcak2016second}.
Looking at \eqref{eq:ProxMapForSubgr}, the optimal $u,y$ fulfill $u_j = u'_j, y_j = y'_j$ whenever $j \notin \{ i,i-1\}$.
Hence we may restrict ourselves to consider the four arguments $u_i,y_i,u_{i-1},y_{i-1}.$
Further the gradient of the mapping $u_{j} \mapsto \tfrac{1}{2} \sum _{j}  d(u_{j},u'_j)^2$ is given as $-\log_{u_{j}}u'_{j}.$
(For background information on (Riemannian) differential geometry, we refer to the books \cite{spivak1975differential,do1992riemannian}.)
So we have to compute the (sub)gradient of the mapping $D( [u_i,y_i],[u_{i-1},y_{i-1}])$ as a function of $u_i,y_i,u_{i-1},y_{i-1}$. To this end, we have to specify the setting to the two versions of $D$ as proposed in Section~\ref{sec:defModel}; the respective derivations are topics of Section~\ref{subsec:AlgSchildUniv}  and Section~\ref{subsec:AlgParUniv}.

We next consider the bivariate denoising problem.
We use $\MTGV$ regularization with the $\MTGV$ functional given in \eqref{eq:tgv_manifold_general_multi} where,
as pointed out, we focus on $p=1$. 
We obtain the 
denoising problem
\begin{align} \label{eq:tgv_manifoldSchildMultAlgo} 
\min_{(u_{i,j})_{i,j},(y_{i,j}^1)_{i,j},(y_{i,j}^2)_{i,j}} \quad & 
\tfrac{1}{2} \sum\nolimits_{i,j}  d(u_{ij},h_{ij})^2 +  \alpha_1 \sum\nolimits_{i,j}  d(u_{i+1,j},y^ 1_{i,j}) + d(u_{i,j+1},y^ 2_{i,j})\\
& +\alpha _0 \sum\nolimits_{i,j}  D\big([u_{i,j},y^1_{i,j}],[u_{i-1,j},y^1_{i-1,j}]\big) + D\big([u_{i,j},y^2_{i,j}],[u_{i,j-1},y^2_{i,j-1}]\big)  \notag\\
&  +\alpha _0 \sum\nolimits_{i,j} D^{\text{sym}} ([u_{i,j},y^1_{i,j}],[u_{i,j},y^2_{i,j}],[u_{i,j-1},y^1_{i,j-1}],[u_{i-1,j},y^2_{i-1,j}])\notag
\end{align}
where the $(h_{ij})_{ij}$ denotes the bivariate data.

We decompose the objective functional in \eqref{eq:tgv_manifoldSchildMultAlgo} into the atoms 
\begin{align} \label{eq:tgv_manifoldSchildMultAlgoDecomposition}
& g^{(1)}_{ij}(u,y^1,y^2) := \tfrac{1}{2} d(u_{i},h_i)^2; \notag \\
& g^{(2)}_{ij}(u,y^1,y^2) := \alpha _1 d(u_{i+1,j},y^ 1_{i,j});
\qquad g^{(3)}_i(u,y^1,y^2) := \alpha _1 d(u_{i,j+1},y^ 1_{i,j}); \notag\\
& g^{(4)}_{ij}(u,y^1,y^2) :=  \alpha _0 D\big([u_{i,j},y^1_{i,j}],[u_{i-1,j},y^1_{i-1,j}]\big) ;   \\
& g^{(5)}_{ij}(u,y^1,y^2) := \alpha _0  D\big([u_{i,j},y^2_{i,j}],[u_{i,j-1},y^2_{i,j-1}]\big) ;      \notag \\
&g^{(6)}_{ij}(u,y^1,y^2) := \alpha _0  D^{\text{sym}} ([u_{i,j},y^1_{i,j}],[u_{i,j},y^2_{i,j}],[u_{i,j-1},y^1_{i,j-1}],[u_{i-1,j},y^2_{i-1,j}]).\notag
\end{align}
We use the decomposition w.r.t.\ the atoms $g^{(k)}_{ij}$ as the decomposition \eqref{eq:Decomp4CPPA} for the CPPA,
and apply the iteration \eqref{eq:iterationCPPA} w.r.t. the derived decomposition.
The proximal mappings $\prox_{\lambda g^{(1)}_{ij}}$ are given by \eqref{eq:tgv_manifoldParallelUnivAlgoProxData},
and the proximal mappings $\prox_{\lambda g^{(2)}_{ij}},\prox_{\lambda g^{(3)}_{ij}}$ are given by \eqref{eq:tgv_manifoldParallelUnivAlgoProxDist}.
The proximal mappings $\prox_{\lambda g^{(4)}_{ij}},\prox_{\lambda g^{(5)}_{ij}}$ 
are computed as in the univariate case explained above.
It remains to compute the proximal mappings of the atoms $g^{(6)}_{ij}.$
As before, there is no explicit formula available and we use a subgradient descent for their computation as well.
The computation of the corresponding derivatives 
are topics of Section~\ref{subsec:AlgSchildUniv}  and Section~\ref{subsec:AlgParUniv}.

\subsection{Riemannian gradients for the Schild version}
\label{subsec:AlgSchildUniv}

We here derive the derivatives needed to compute the proximal mappings of the mappings $D$ and $D^{\text{sym}}$ 
in Section~\ref{sec:AlgBasicUni}
when specifying to the Schild's variant.
We start out to compute the derivative of the mapping $\ds$ given by
\begin{equation} \label{eq:DSAgain}
\ds(u_{i-1},y_{i-1},u_{i},y_{i}) = \ds([u_i,y_i],[u_{i-1},y_{i-1}]) = d([u_{i-1},[u_{i},y_{i-1}]_{\tfrac{1}{2}} ]_2,y_{i}).
\end{equation}
Recall that, for our computations, we assume that geodesics are unique, which is, as pointed out above, the case up to a set of measure zero. 

We directly see that $\ds$ is symmetric w.r.t. interchanging $y_{i-1},u_{i}$.
Hence we only have to compute the differential with respect to one of these variables $y_{i-1},u_{i}$.

Furthermore, the gradient $\nabla_{y_{i}} \ds$ w.r.t.\ the fourth argument $y_{i}$ of 
$\ds$ for $y_{i}$ with $y_{i}\neq [u_{i-1},[u_{i},y_{i-1}]_{\tfrac{1}{2}} ]_2$  is just the well-known gradient of the Riemannian distance function which is known to be given by \cite{do1992riemannian,afsari2011riemannian}
\begin{equation}
\nabla_{y_{i}} \ds = 
- \log_{y_{i}}S(u_{i-1},y_{i-1},u_{i})  /
{\big |\log_{y_{i}}S(u_{i-1},y_{i-1},u_{i})   \big |}
\end{equation}
where $|\cdot|$ is the norm in the tangent space associated with the point $y_{i}$, i.e., the square root of the Riemannian scalar product.
Here,
\begin{equation}\label{eq:SasFunctionOfUi}
 S(u_{i-1},y_{i-1},u_{i}) = [u_{i-1},[u_{i},y_{i-1}]_{\tfrac{1}{2}} ]_2
\end{equation}
denotes the result of applying the Schild's construction to the respective arguments.

In order to determine the gradients w.r.t.\ the other variables we have to apply the adjoint of the 
respective differentials to the gradient of the distance function w.r.t. the first argument. More precisely, we have to calculate 
\begin{align}\label{eq:InitDiscussion1}
\nabla_{u_{i-1}} \ds &=  
-T_1
\left(
{\log_{S(u_{i-1},y_{i-1},u_{i})} y_{i}}    
/
{\big | \log_{S(u_{i-1},y_{i-1},u_{i})} y_{i} \big |}
\right), \\
\nabla_{y_{i-1}} \ds &=  
-T_2
\left(
{\log_{S(u_{i-1},y_{i-1},u_{i})} y_{i}}    
/
{\big | \log_{S(u_{i-1},y_{i-1},u_{i})} y_{i} \big |}
\right),\label{eq:InitDiscussion2}
\end{align}
where $T_1$ is the adjoint of the differential of the mapping 
$u_{i-1} \mapsto 
[u_{i-1},[u_{i},y_{i-1}]_{\tfrac{1}{2}} ]_2$,
and where 
$T_2$ is the adjoint of the differential of the mapping 
$y_{i-1} \mapsto 
[u_{i-1},[u_{i},y_{i-1}]_{\tfrac{1}{2}} ]_2$.
The differential w.r.t. $u_{i}$ is obtained by symmetry as pointed out above.
In the following we derive these adjoint mappings in terms of  Jacobi fields.

To this aim, we first explain the notion of a Jacobi field and then point out the connection with geodesic variations. 
In a Riemannian manifold $\mathcal M,$ the Riemannian curvature tensor $R$ is given by $R(X,Y)Z = \nabla_X \nabla_Y Z - \nabla_Y \nabla_X Z - \nabla_{[X,Y]} Z$ 
where $X,Y,Z$ are vector fields and $\nabla$ denotes the Levi-Civita connection,
and where $[X,Y]=XY-YX$ denotes the Lie bracket of the vector fields $X,Y.$
A Jacobi field $Y$ along a geodesic $\gamma$ is the solution of the differential equation 
$\tfrac{D}{ds}\tfrac{D}{ds}Y + R(\tfrac{d}{ds}\gamma,Y)\tfrac{d}{ds}\gamma =0.$
The space of Jacobi fields is a $2N$-dimensional linear space where $N$ denotes the dimension of $\M.$
The connection to geodesic variations is as follows: the ``derivative'' vector field  
$Y(s) = \tfrac{d}{dt}|_{t=0}V(s,t)$ of a geodesic variation $V$ is a Jacobi field, and conversely, any Jacobi field is obtained from a geodesic variation. Further details may be found in the books \cite{cheeger1975comparison,spivak1975differential,do1992riemannian}. 

We also need the notion of  (Riemannian) symmetric spaces for the formulation of the next lemma.
A symmetric space is a Riemannian manifold for which its curvature tensor $R$ is invariant under parallel transport,
 $R(\pt_{x,y} X, \pt_{x,y} Y) \pt_{x,y} Z = \pt_{x,y} R(X,Y)Z$
 where $X,Y,Z$ are tangent vectors at $x$ and $\pt_{x,y}$ denotes the parallel transport from $x$ to $y$ along a curve $\gamma$
(here not reflected in the notation).
Then, in particular, the covariant derivative of the curvature tensor $R$ along any curve equals zero in a symmetric space.
A reference for symmetric spaces is \cite{cheeger1975comparison}.

\begin{lem} \label{lem:GradDSwrtui}
	The mapping $T_1$ of Equation \eqref{eq:InitDiscussion1} 
	can be computed using Jacobi fields. 
	In particular, we explicitly have for the class of symmetric spaces and points with $\ds \neq 0$ that 
	\begin{equation}\label{eq:NablaUISym}
	\nabla_{u_{i-1}} \ds =  \pt_{S(u_{i-1},y_{i-1},u_{i}),u_{i-1}}
	\left(
	{\log_{S(u_{i-1},y_{i-1},u_{i})} y_{i}}    
	/{\big | \log_{S(u_{i-1},y_{i-1},u_{i})} y_{i} \big |}
	\right),
	\end{equation}
	which means that the gradient of the distance function w.r.t.\ the second argument is 
	reflected at $[u_{i},y_{i-1}]_{\tfrac{1}{2}},$ or in other words, 
	parallel transported from the Schild point	$[u_{i-1},[u_{i},y_{i-1}]_{\tfrac{1}{2}} ]_2$ to $u_{i-1}$ and multiplied by $-1.$	
\end{lem}
\proof The proof is given in  Section \ref{sec:proofs_section4}. \endproof
Regarding the fact that we consider points with $\ds \neq 0,$ 
	we remark that the points $u_i$, $u_{i-1}$, $y_i$, $y_{i-1}$ such that $\ds(u_i,u_{i-1},y_i,y_{i-1})=0$ 
	which corresponds to $y_i = [u_{i-1},[u_i,y_{i-1}]_{\frac{1}{2}}]_2$
	form a set of measure zero. On this zero set,
	for instance, the four-tuple consisting of the four zero-tangent vectors sitting in  $u_i,u_{i-1},[u_{i-1},[u_i,y_{i-1}]_{\frac{1}{2}}]_2,y_{i-1}$
	belong to the subgradient of $\ds.$
\begin{lem} \label{lem:GradDSwrtyi}
	The gradient of the function  $\ds$ for points with $\ds \neq 0$ given by \eqref{eq:DSAgain} w.r.t.\ the variable $y_{i-1}$
	is given by 	
	\begin{equation}\label{eq:diffDswrtyi}
	\nabla_{y_{i-1}} \ds =  
	-T_4 
	\left( T_3
	\left(
	{\log_{S(u_{i-1},y_{i-1},u_{i})} y_{i}}    
	/{\big | \log_{S(u_{i-1},y_{i-1},u_{i})} y_{i} \big |}
	\right)\right), \\			
	\end{equation}
	where $T_3$ is the adjoint of the derivative of the mapping 
	$
	m \mapsto  [u_{i-1},m]_2, 
	$
	and	where $T_4$ is the adjoint of the derivative of the mapping 
	$
	y_{i-1} \mapsto  [u_{i},y_{i-1}]_{\tfrac{1}{2}},
	$	
	that is, $T_2$ as in \eqref{eq:InitDiscussion2} is given as $T_2 = T_4 \circ T_3$. Further $T_3$ and $T_4$ can be computed using Jacobi fields.
\end{lem}

\proof The proof is given in  Section \ref{sec:proofs_section4}. \endproof

For the class of symmetric spaces, we make the mappings $T_3,T_4$ more explicit. 

\begin{lem} \label{lem:T3T4explicit}
Let $\M$ be a symmetric space.	
Let $\gamma:[0,1]\rightarrow \M$ denote the geodesic connecting $\gamma(0)=u_{i-1}$ and $\gamma(1)=m=[u_{i},y_{i-1}]_{\tfrac{1}{2}}.$ 	
We consider an orthonormal basis $v_n$ of eigenvectors of the self-adjoint Jacobi operator 
$J \mapsto R(\tfrac{\gamma'(1)}{ |\gamma'(1) |} ,J) \tfrac{\gamma'(1)}{ |\gamma'(1) |}$  with corresponding eigenvalues $\lambda_n,$ 
and $v_1$ tangent to $\gamma.$ W.r.t.\ this basis, the mapping $T_3$ is given by
\begin{equation}\label{eq:JacobiSymZentrStr}
T_3: \ \sum\nolimits_n \alpha_n \pt_{[u_{i},y_{i-1}]_{\frac{1}{2}},S(u_{i-1},y_{i-1},u_{i})} v_n \mapsto 
\sum\nolimits_n \alpha_n f(\lambda_n) v_n,
\end{equation}
where, using the symbol $d$ for the distance between $u_{i-1}$ and $[u_{i},y_{i-1}]_{\frac{1}{2}},$
\begin{equation}\label{eq:f4T3}
f(\lambda_n) = 
\begin{cases}
2,   & \quad \text{if} \ \lambda_n = 0, \\
{\sin (2 \sqrt{\lambda_n} d)}/{\sin (\sqrt{\lambda_n} d)},    & \quad \text{if} \ \lambda_n > 0,   \quad  d < \pi/\sqrt{\lambda_n},\\
{\sinh (2 \sqrt{-\lambda_n} d)}/{\sinh (\sqrt{-\lambda_n} d)},  &  \quad \text{if} \ \lambda_n < 0.
\end{cases}
\end{equation}
Here, the $\alpha_n$ are the coefficients of the corresponding basis representation. 

Further, let $\xi:[0,1]\rightarrow \M$ be the geodesic connecting
$y_{i-1}=\xi(0)$ and $u_{i}=\xi(1)$ and 
$w_n$ be an orthonormal basis of eigenvectors of  
$J \mapsto R(\frac{\xi'(0)}{ |\xi'(0)|} ,J) \frac{\xi'(0)}{ |\xi'(0) |}$  with eigenvalues  $\mu_n,$ 
and $w_1$ tangent to $\xi.$
Then, w.r.t.\ this basis, the mapping $T_4$ is given by
\begin{equation}\label{eq:JacobiSymZentrStr2}
T_4: \ \sum\nolimits_n \beta_n \pt_{y_{i-1},[u_{i},y_{i-1}]_{\frac{1}{2}}}  w_n \mapsto 
\sum\nolimits_n \beta_n g(\lambda_n) w_n,
\end{equation}
where, using the symbol $d'$ for the distance between $y_{i-1}$ and $u_{i},$
\begin{equation}
g(\lambda_n) = 
\begin{cases}
1/2,   & \quad \text{if} \ \lambda_n = 0, \\
{\sin (\frac{1}{2} \sqrt{\lambda_n} d')}/{\sin (\sqrt{\lambda_n} d')},    & \quad \text{if} \ \lambda_n > 0,
\quad  d < \pi/\sqrt{\lambda_n},\\
{\sinh (\frac{1}{2} \sqrt{-\lambda_n} d')}/{\sinh (\sqrt{-\lambda_n} d')},  &  \quad \text{if} \ \lambda_n < 0.
\end{cases}
\end{equation}
Here, the $\beta_n$ are the coefficients of the corresponding basis representation. 
\end{lem}	

\proof The proof is given in  Section \ref{sec:proofs_section4}. \endproof

Next, we consider the Riemannian gradients for the Schild version $D^{\text{sym}}=\dss$ as given in \eqref{eq:DefSchildSSym}
which are needed for the CPPA for the bivariate TGV functional.
{
To this end, we that note that $\dss$ can be expressed in terms of the function $S$ of \eqref{eq:SasFunctionOfUi} by
\begin{align*}
\dss ([u_{i,j},y_{i,j}^ 1],[u_{i,j},y^ 2_{i,j}],[u_{i,j-1},y^ 1_{i,j-1}],[u_{i-1,j},y^2_{i-1,j}])  
 = d( S(y_{i,j-1}^1,r^2,r^1),  y_{i-1,j}^2) \\
 \text{s.t. }  r ^1 = S(u_{i,j},y^1_{i,j},u_{i-1,j}), \, r ^2 = S(u_{i,j},y^2_{i,j},u_{i,j-1}). \notag
\end{align*}
}
Hence, the mapping $\dss$ (which depends on seven arguments) is an expression of
the Riemannian distance function and three realizations of $S.$ We have seen how to differentiate $S$
in Lemma~\ref{lem:GradDSwrtui} and Lemma~\ref{lem:GradDSwrtyi}. 
We derive the gradient of 
$\dss$ by iterated application of these mappings and the concepts of these lemmata. 
We point out the symmetry of $\dss$ with respect to $u_{i-1,j},y^1_{i,j},$
and  with respect to $u_{i,j-1}, y^2_{i,j}$.
This reduces the task to actually considering five different arguments. 
The derivative of $\dss$ is provided in the following proposition. Its proof is a direct consequence of the previous results and the decomposition of $\dss$ into decompositions of the function $S$ and the distance function $d$ as above.

\begin{prop} \label{lem:GradDSym}
	We consider constellations of points with $\dss \neq 0.$ 
	Then, the gradient of $\dss$  w.r.t.\ the variable $y^1_{i,j}$
	is given by 	
	\begin{equation}\label{eq:dSymDer3}
	\nabla_{y^1_{i,j}} \dss =  
	- T_4 \circ T_3 \circ \tilde{T}_4 \circ \tilde{T}_3 
	\left(
	{\log_{S(y_{i,j-1}^1,r^2,r^1)} y_{i-1,j}^2}/   
	{\big |\log_{S(y_{i,j-1}^1,r^2,r^1)} y_{i-1,j}^2 \big |}
	\right), \\			
	\end{equation}
	with the adjoint operators $T_3$, $T_4$ given as in Lemma~\ref{lem:GradDSwrtyi} formed w.r.t.\ the points $u_{i,j},u_{i-1,j},y^1_{i,j}$ and the adjoint operators $\tilde{T}_3$, $\tilde{T}_4$ given as in Lemma~\ref{lem:GradDSwrtyi} formed w.r.t.\ the points $y_{i,j-1}^1,r^2,r^1.$
	The gradient of $\dss$ w.r.t.\ $u_{i-1,j}$
	has the same form by symmetry.
	
	Similarly, the gradient of $\dss$  w.r.t.\ the variable $y^2_{i,j}$
	is given by 	
	\begin{equation}\label{eq:dSymDer1}
	\nabla_{y^2_{i,j}} \dss =  
	- T_4 \circ T_3 \circ \tilde{T}_4 \circ \tilde{T}_3 
	\left(
	{\log_{S(y_{i,j-1}^1,r^2,r^1)} y_{i-1,j}^2}/    
	{\big |\log_{S(y_{i,j-1}^1,r^2,r^1)} y_{i-1,j}^2 \big |}
	\right), \\			
	\end{equation}   
	with the adjoint operators $T_3$, $T_4$ as in Lemma~\ref{lem:GradDSwrtyi} formed w.r.t.\ the points $u_{i,j},u_{i,j-1},y^2_{i,j}$ and the adjoint operators $\tilde{T}_3$, $\tilde{T}_4$ as in Lemma~\ref{lem:GradDSwrtyi} formed w.r.t.\ the points $y_{i,j-1}^1,r^2,r^1$.
	The gradient of $\dss$ w.r.t.\ $u_{i,j-1}$ 
	has the same form by symmetry.
	
	The gradient of $\dss$  w.r.t.\ the variable $y_{i,j-1}^ 1$
	is given by 	
	\begin{equation}\label{eq:dSymDer2}
	\nabla_{y_{i,j-1}^ 1} \dss =  
	- T_1 
	\left(
	{\log_{S(y_{i,j-1}^1,r^2,r^1)} y_{i-1,j}^2}/    
	{\big |\log_{S(y_{i,j-1}^1,r^2,r^1)} y_{i-1,j}^2 \big |}
	\right), \\			
	\end{equation}   
	with the adjoint operator $T_1$ given as in Lemma~\ref{lem:GradDSwrtui} formed w.r.t.\ the points $y_{i,j-1}^1,r^2,r^1$.
	The gradient with respect to the variable $y_{i-1,j}^ 2$ is simply given by
	\begin{equation}
	\nabla_{y^2_{i-1,j}} \dss =  
	-	{\log_{y_{i-1,j}^2} S(y_{i,j-1}^1,r^2,r^1)}/    
	{\big |\log_{y_{i-1,j}^2} S(y_{i,j-1}^1,r^2,r^1) \big |}. \\			
	\end{equation}
	Finally, the gradient of $\dss$  w.r.t.\ the variable $u_{i,j}$
	is given by 	
	\begin{equation}\label{eq:dSymDer4}
	\nabla_{u_{i,j}} \dss =  
	-T_1 \circ T_4 \circ T_3 
	\left(
	\tfrac{\log_{S(y_{i,j-1}^1,r^2,r^1)} y_{i-1,j}^2}    
	{\big |\log_{S(y_{i,j-1}^1,r^2,r^1)} y_{i-1,j}^2 \big |}
	\right) \\			
	-
	\tilde{T}_1\circ T_4 \circ T_3 
	\left(
	\tfrac{\log_{S(y_{i,j-1}^1,r^2,r^1)} y_{i-1,j}^2}    
	{\big |\log_{S(y_{i,j-1}^1,r^2,r^1)} y_{i-1,j}^2 \big |}
	\right).
	\end{equation}	
	Here the adjoint operators $T_3$ and $T_4$ are given as in Lemma~\ref{lem:GradDSwrtyi} formed w.r.t.\ the points $y_{i,j-1}^1,r^2,r^1$, $T_1$  is given as in Lemma~\ref{lem:GradDSwrtui} formed w.r.t.\ the points $u_{i,j},y_{i,j}^1,u_{i-1,j}$ and $\tilde{T}_1$ is given as in Lemma~\ref{lem:GradDSwrtui} formed w.r.t.\ the points $u_{i,j},y_{i,j}^2,u_{i,j-1}$.
	
\end{prop}

\subsection{Riemannian gradients for the parallel transport version}
\label{subsec:AlgParUniv}

We here derive the derivatives needed to compute the proximal mappings 
in Section~\ref{sec:AlgBasicUni}
when specifying to the parallel transport variant.
We use the short-hand notation
\begin{align}\label{eq:DefF}
F(u_i,u_{i-1},y_i,y_{i-1}) = \dpt([u_i,y_i],[u_{i-1},y_{i-1}]).
\end{align}
For the following computations recall that, by our assumption of uniqueness of length-minimizing geodesics, the $\log$ mapping, initially defined as set-valued, always maps to single points in the tangent space.
\begin{lem} \label{FSymAndImplications4Der}
	The function $F$ given by \eqref{eq:DefF} is symmetric with respect to interchanging $(u_i,y_i)$ with $(u_{i-1},y_{i-1}).$ 
	In particular, for points with $F \neq 0,$ the gradient of  $F$ w.r.t.\ the third variable $y_{i},$
	is given by 
	$
	\nabla_{y_{i}} F(u_i,u_{i-1},y_i,y_{i-1})  = 
	\nabla_{y_{i-1}} F(u_{i-1}, u_i,y_{i-1},  y_i).
	$
	Further, again for points with $F \neq 0$, the gradient of the function $F$ w.r.t.\ the first component variable $u_{i},$
	is given by 
	$
	\nabla_{u_{i}} F(u_i,u_{i-1},y_i,y_{i-1})  = 
	\nabla_{u_{i-1}} F(u_{i-1}, u_i,y_{i-1},  y_i).
	$	
\end{lem}

\proof The proof is given in Section \ref{sec:proofs_section4} by specifying $F=F_0$ in Lemma~\ref{FtSymAndImplications4Der}. 
\endproof
Regarding the fact that we consider points with $F \neq 0$ in the above statement,
we remark that the points $u_i,u_{i-1},y_i,y_{i-1}$ such that $F(u_i,u_{i-1},y_i,y_{i-1})=0$ form a set of measure zero.
Only on this zero set, $F$ is not differentiable.
Further, in this case, the four-tuple consisting of the four zero-tangent vectors sitting in  $u_i,u_{i-1},y_i,y_{i-1}$
belong to the subgradient of $F.$
We note that Lemma~\ref{FSymAndImplications4Der} tells us that, in order to compute the gradient of the second order type difference $F$, we only need to compute the respective gradient of $F$ w.r.t. $u_{i-1}$ and $y_{i-1}.$ This is done in the following. 

\begin{lem} \label{lem:GradF1wrtyip1}
	The gradient of the function $F$ for points with $F \neq 0$ w.r.t.\ the variable $y_{i-1}$
	is given by 
	\begin{equation}
		\nabla_{y_{i-1}} F = T  
		\left(\left({\log_{u_{i-1}} y_{i-1} - \pt_{u_{i},u_{i-1}} \log_{u_i} y_i}\right)/{\big | \log_{u_{i-1}} y_{i-1} - \pt_{u_{i},u_{i-1}} \log_{u_i} y_i  \big |}	\right),	
	\end{equation}
	where $T = (d_{y_{i-1}} \log_{u_{i-1}})^\ast$ is the adjoint of the (Fr\'echet) derivative (denoted by the symbol $d_{y_{i-1}}$) of the $\log$ mapping w.r.t. the variable
	$y_{i-1}.$ $T$ can be computed using Jacobi fields. 
\end{lem}

\proof The proof is given in  Section \ref{sec:proofs_section4}. 
\endproof

In Riemannian symmetric spaces the above mapping $T$ can be made more explicit.

\begin{lem} \label{lem:Texplicit}
	Let $\M$ be a symmetric space.
	Consider the geodesic $\gamma:[0,1]\rightarrow \M$ connecting $\gamma(0)=u_{i-1}$ and $\gamma(1)=y_{i-1},$ 
	and let $d$ denote the distance between $u_{i-1}$ and $y_{i-1}.$
	Let $(v_n)_n$ be an orthonormal basis of eigenvectors of the self-adjoint Jacobi operator 
	$J \mapsto R(\frac{\gamma'(0)}{ |\gamma'(0) |},J)\frac{\gamma'(0)}{ |\gamma'(0) |}$
	with $v_1$ tangent to $\gamma,$ and denote, 
	for each $n$, the eigenvalue associated with $v_n$ by $\lambda_n.$ 
	W.r.t.\ this basis, the operator $T$ of Lemma~\ref{lem:GradF1wrtyip1} 
	can be represented by 
	\begin{equation}\label{eq:DiffOfLogInSym}
	T: \  \sum_n \alpha_n  v_n \mapsto \sum_n \alpha_n f(\lambda_n) \pt_{u_{i-1},y_{i-1}} v_n,
	\end{equation}
	where the $\alpha_n$ are the coefficients of the corresponding basis representation and 
	the function $f,$ depending on the sign of $\lambda_n$ is given by 
	\begin{equation}\label{eq:functionalCalculusDiffOfLogInSym}
	f(\lambda_n) = 
	\begin{cases}
	1,   & \quad \text{if} \ \lambda_n = 0, \\
	{\sqrt{\lambda_n} d}/{\sin (\sqrt{\lambda_n} d)},    & \quad \text{if} \ \lambda_n > 0, \quad  d < \pi/\sqrt{\lambda_n},\\
	{\sqrt{-\lambda_n} d}/{\sinh (\sqrt{-\lambda_n} d)},  &  \quad \text{if} \ \lambda_n < 0.
	\end{cases}
	\end{equation}
\end{lem}	
\proof The proof is given in  Section \ref{sec:proofs_section4}. 
\endproof

Finally, we consider the gradient of $F$  w.r.t.\ the variable $u_{i-1}.$
\begin{lem} \label{lem:GradF1wrtuip1}
	The gradient of the function $F$ for points with $F \neq 0$  w.r.t.\ the variable $u_{i-1}$
	is given by 
	\begin{equation}\label{eq:ExpandGradinONB}
	\nabla_{u_{i-1}} F_1 = \sum \alpha_n  v_n,	
	\end{equation}
	where the $(v_n)_n$ form an orthonormal basis of the tangent space at $u_{i-1},$ 
	and the coefficients $\alpha_n$ are given by 
	\begin{align}\label{eqref:defAlphaN}
	\alpha_n = 
	\frac{d}{dt}|_{t=0} \big | L^n_t - B^n_t \big | 
	& =  \left\langle   \tfrac{L(u_{i-1})- B(u_{i-1})}{\left| L(u_{i-1})- B(u_{i-1}) \right |} ,  
	\tfrac{D}{dt}|_{t=0} L_t^n - \tfrac{D}{dt}|_{t=0} B_t^n  \right\rangle. 
	\end{align}
	Here $L_t^n,B_t^n$ denote the vector fields $L,B$ 
	(defined by  
		$
		L: u_{i-1} \mapsto \log_{u_{i-1}}y_{i-1} 
		$
		and 
		$
		B: u_{i-1} \mapsto \pt_{u_{i},u_{i-1}} z,$ where 
		$z = \log_{u_{i}}y_{i},$)	
	along the (specific) geodesic $t \mapsto \exp_{u_{i-1}}tv_n,$ $t \in [0,1],$
	determined by $v_n,$ i.e.,
	\begin{align}\label{eqref:defAlphaNVecF}
	L^n_t = \log_{\exp_{u_{i-1}}tv_n}y_{i-1} \quad\text{and}\quad
	B^n_t = \pt_{u_{i},\exp_{u_{i-1}}tv_n} z,
	\end{align} 
	with $z = \log_{u_{i}}y_{i}.$
\end{lem}

\proof The proof is given in  Section \ref{sec:proofs_section4}. 
\endproof

The precise computation of $\frac{D}{dt}|_{t=0} L_t^n$ in symmetric spaces 
is topic of Lemma~\ref{lem:GradLn}.
Further, the computation of $\frac{D}{dt}|_{t=0} B_t^n$ is carried out for 
	the manifolds explicitly considered in this paper: this is done for 
	the sphere in Lemma~\ref{lem:GradBnSphere},
and for the space of positive matrices in Lemma~\ref{lem:GradBnPos}.    
Lemma~\ref{lem:GradBnPos} and its proof may serve as a prototypic guide for deriving similar expressions for other symmetric spaces such as the rotation groups or the Grassmannians for instance.
We note that the approach is by no means restricted to the two considered classes of spaces and might serve as a guide for other manifolds; we only did not derive a more explicit representation on the general level of symmetric spaces. We further note that numerical differentiation of the particular term is a second option as well. 	

\begin{lem} \label{lem:GradLn}	
	Assume that the manifold $\mathcal M$ is a symmetric space.  
	Let $(v_n)_n$ be an orthonormal basis of eigenvectors of the self-adjoint Jacobi operator 
	$J \mapsto R(\frac{\gamma'(0)}{ |\gamma'(0) |} ,J) \frac{\gamma'(0)}{ |\gamma'(0) |}$ where the (constant speed) geodesic $\gamma:[0,1]\rightarrow \M$ connects $u_{i-1} = \gamma (0)$
	and $y_{i-1} = \gamma (1),$ and 
	where $R$ denotes the Riemannian curvature tensor.
	For each $n$, we denote by $\lambda_n$ the eigenvalue associated with $v_n.$
	 
	The covariate derivatives $\frac{D}{dt}|_{t=0} L_t^n,$ 
	of the vector fields
	$L^n_t = \log_{\exp_{u_{i-1}}tv_n}y_{i-1}$ at $t=0$  	
	can be computed jointly for all $n $ using Jacobi fields as follows:
	\begin{equation}\label{eq:ExplicitDiffOfL}
	     \frac{D}{dt}|_{t=0} L_t^n = 
	     \begin{cases}
	             -v_n,   &\text{if} \quad \lambda_n = 0,  \\
	             - d \sqrt{\lambda_n} \ \frac{  \cos( \sqrt{\lambda_n} d)}{\sin( \sqrt{\lambda_n} d)} \ v_n,  
	             & \text{if} \quad \lambda_n > 0, \quad  d < \pi/\sqrt{\lambda_n}, \\
	             - d \sqrt{-\lambda_n} \ \frac{  \cosh( \sqrt{-\lambda_n} d)}{\sinh( \sqrt{-\lambda_n} d)} \ v_n, 
	             & \text{if} \quad \lambda_n < 0.
	     \end{cases}	     
	\end{equation} 
	Here $d = d(u_{i-1},y_{i-1})$ denotes the length of the geodesic 
	connecting $u_{i-1},y_{i-1}.$
	(If the term $\sqrt{\lambda_n} d = 0$ in the denominators of the second line in \eqref {eq:ExplicitDiffOfL},
	then $u_{i-1} = y_{i-1},$ and the formula is still valid since we are facing a removable singularity then.)
\end{lem}

\proof The proof is given in  Section \ref{sec:proofs_section4}. 
\endproof

\begin{lem} \label{lem:GradBnSphere}	
	Consider the unit sphere $S^2$ embedded into euclidean space $\mathbb R^3$. 
	For $u_i, u_{i-1}$ with $u_i \neq u_{i-1}$, the differential $\frac{D}{dt}|_{t=0} B_t^n$ is given by 
	\begin{align}\label{eq:StatementParTransS2}
	\tfrac{D}{dt}|_{t=0} B_t^n = 	\tfrac{D}{dt}|_{t=0} \pt_{u_{i},\exp_{u_{i-1}}tv_n} z  = 
	\begin{cases}
	 \ 0 \quad  & \text{ for } v_n  \ \|  \log_{u_{i-1}}u_i, \\
	 \  {\mathrm L}_\omega \pt_{u_{i},u_{i-1}} z,
	 \qquad  & \text{ for } v_n \perp \log_{u_{i-1}}u_i, \  |v_n |=1,	 
	 \end{cases}
	 \end{align}
	 and $v_n$ to the left of $\log_{u_{i-1}}u_i$ (otherwise multiplied by $-1$ accounting for the change of orientation).
	 Here the skew-symmetric matrix $
	 {\mathrm L}_\omega = 
	 \begin{pmatrix}
	 0 & \omega  \\ -\omega  & 0
	 \end{pmatrix}
	 $ 
	 is taken w.r.t. the basis 
	 $\{\log_{u_{i-1}}u_i,$ $(\log_{u_{i-1}}u_i)^\perp\}$  of the tangent space $T_{u_{i-1}},$ 
	 and $\omega$ is given by	 
	 $
	  \omega     =\tfrac{1}{\sin d} - \tfrac{1}{\tan d},  
	  \quad\text{where } d = d(u_{i},u_{i-1}).
	 $
For general $v_n,$ 
    \begin{align}\label{eq:StatementParTransS2GeneralVn}
    \tfrac{D}{dt}|_{t=0} B_t^n = \tfrac{D}{dt}|_{t=0} \pt_{u_{i},\exp_{u_{i-1}}tv_n} z  = 
     \big\langle  v_n, w    \big\rangle  \  {\mathrm L}_\omega \pt_{u_{i},u_{i-1}} z,
    \end{align}
where $w$ is the vector determined by  $w \perp \log_{u_{i-1}}u_i, \  |w|=1,$ and $w$ is to the left of $\log_{u_{i-1}}u_i.$
In other words, 
we have to multiply the second line of \eqref{eq:StatementParTransS2} with the signed length of the projection of $v_n$ to the normalized vector $(\log_{u_{i-1}}u_i)^\perp$.

If  $u_i = u_{i-1},$ then the differential $\frac{D}{dt}|_{t=0} B_t^n = 0$ (which is consistent with letting $d \to 0$ in the above formulae.)	 
\end{lem}

\proof The proof is given in  Section \ref{sec:proofs_section4}. 
\endproof

\begin{lem} \label{lem:GradBnPos} 	
	Let $\M$ be the space of symmetric positive definite matrices. 
	Then, the covariate derivative $\frac{D}{dt}|_{t=0} B_t^n$ (which is a tangent vector sitting in $u_{i-1}$) is given by the following sum of matrices
	\begin{equation}
	\tfrac{D}{dt}|_{t=0} B_t^n = (T-\tfrac{1}{2}S) + (T-\tfrac{1}{2}S)^\top,
	\end{equation}
	where $(T-\tfrac{1}{2}S)^\top$ denotes the transpose of the matrix $T-\tfrac{1}{2}S.$ 
    The matrix $S$ is determined in terms of elementary matrix operations 
	of the data (by \eqref{eq:defMatSpos} in the proof of the statement). 
	The matrix $T$ is determined in terms of elementary matrix operations and the solution of a Sylvester equation  
	(by \eqref{eq:defTandX} in the proof of the statement 
	with the Sylvester equation given by \eqref{eq:SyvEqToSolve} there).	 
\end{lem}

\proof The proof is given in  Section \ref{sec:proofs_section4}. 
\endproof
Summing up, we have computed the derivatives of all building blocks necessary to compute the derivative of 
$F(u_i,u_{i-1},y_i,y_{i-1}) = \dpt([u_i,y_i],[u_{i-1},y_{i-1}])$ for the non-degenerate case $F \neq 0.$ 	
	
\begin{rem} For the bivariate version of the parallel transport based $\MTGV$ realization of Section \ref{sec:axomatic_extension_bivariate},
 we can use the analogue of the decomposition \eqref{eq:tgv_manifoldSchildMultAlgoDecomposition} with $\ds, \dss$ replaced 
 by the corresponding parallel transport versions \eqref{eq:DefDpt} and \eqref{eq:DefPtSym}.
 Then we can use this analogue of the decomposition \eqref{eq:tgv_manifoldSchildMultAlgoDecomposition}
 and apply the CPPA iteration \eqref{eq:iterationCPPA} to this decomposition.
 The proximal mappings $\prox_{\lambda g^{(1)}_{ij}}$ are given by \eqref{eq:tgv_manifoldParallelUnivAlgoProxData}, $\prox_{\lambda g^{(2)}_{ij}},\prox_{\lambda g^{(3)}_{ij}}$ are given by \eqref{eq:tgv_manifoldParallelUnivAlgoProxDist} as for the Schild variant above. The proximal mappings $\prox_{\lambda g^{(4)}_{ij}},\prox_{\lambda g^{(5)}_{ij}}$ 
 are computed as in the univariate case considered in Section~\ref{subsec:AlgParUniv}.
 In order to compute the proximal mappings of the atoms $g^{(6)}_{ij}$ using a subgradient descent,
 it is necessary to differentiate the mapping $\dpts$ of \eqref{eq:DefPtSym} (which is the analogue of $\dss$ for the Schild case) w.r.t. its seven arguments. As in the Schild case, it is possible to decompose the mapping into simpler functions
 which we have already considered in Section~\ref{subsec:AlgParUniv}. We do not carry out the rather space consuming derivation here. 
\end{rem}

\section{Numerical results}\label{sec:numericalResults}

This section provides numerical results for synthetic and real
signals and images. We first describe the experimental setup.

We carry out experiments for $S^1$, $S^2$ and $\Pos_3$ (the manifold of symmetric positive definite $3 \times 3$ matrices equipped with the Fisher-Rao metric) valued data. $S^1$ data is visualized by the phase angle, and color-coded as hue value in the HSV color space when displaying image data. We visualize $S^2$ data
either by a scatter plot on the sphere as in Figure~\ref{fig:PTvsSchild}, or 
by a color coding as in Figure~\ref{fig:S2_synth}.
Data on the $\Pos_3$ manifold is visualized
by the isosurfaces of the corresponding quadratic forms. More precisely, the ellipse visualizing
the point $f_p$ at voxel $p$ are the points $x$ fulfilling $(x-p)^\top f_p (x-p) = c,$ for some $c>0.$

To quantitatively measure the  quality of a reconstruction,
we use the manifold variant of the \emph{signal-to-noise ratio improvement} 
\[
	\deltaSNR = 10 \log_{10} \left(  \frac{\sum_{ij} d(g_{ij}, f_{ij})^2   }{\sum_{ij} d(g_{ij}, u_{ij})^2}\right) \dB,
\]
see \cite{unser2014introduction, weinmann2014total}.
Here $f$ is the noisy data, $g$ the ground truth, and $u$ a regularized  restoration.
A higher $\deltaSNR$ value  means better reconstruction quality.

For adjusting 
the model parameters  $\alpha_0,\alpha_1$ of $\MTGV$,
it is convenient to parametrize them by
\begin{equation} \label{eq:model_parameter_definition}
	\alpha_0 =  r \frac{(1 - s)}{\min(s, 1-s)}, \quad \text{and}\quad \alpha_1 =  r \frac{s}{\min(s, 1-s)},
\end{equation}
so that  $r \in (0, \infty)$ controls the overall regularization strength 
and $s \in (0,1)$ the balance of the two TGV penalties.
For $s \to 0$ we get $\alpha_0 \rightarrow \infty$ and $\alpha_1 = r,$
so that $\TGV$ minimization approximates the minimization of $\TV$ modulo a linear term which can be added at no cost.  For $s \to 1$ we have $\alpha_0 = r$ and $\alpha_1 \rightarrow \infty$ which corresponds to pure second-order TV regularization. 
One may think of $r$ as the parameter mainly depending  on the noise level,
and of $s$ as the parameter mainly depending on the geometry of the underlying image.
Figure~\ref{fig:S1_synth_parameters} illustrates the influence on $r$ and $s$ 
for a synthetic $S^1$-valued image corrupted by von Mises noise with concentration parameter $\kappa = 5.$
There and in most of the following experiments, we observed  satisfactory results for the fixed value $s = 0.3.$
Based on these observations, we suggest to use $s = 0.3$ as a starting point,
to decrease or increase it if the image is dominated by edges or smooth parts, respectively.

We have implemented the presented methods in Matlab 2016b.
We use the toolbox MVIRT \cite{bavcak2016second} for the basic manifold operations such as log, exp and parallel transport, and for parts of the visualization\footnote{Implementation available at \url{https://github.com/kellertuer/MVIRT}}.
We used 100000 iterations for all experiments with univariate data and 1000 iterations for the image data.
The cooling sequence $(\lambda^k)_{k \in\N}$ used as stepsize in the gradient descent for computing the non-explicit proximal mappings was chosen as $\lambda^k = \lambda^0 k^{-\tau}$
with $\tau = 0.55.$
For the univariate spherical data we observed faster convergence
when using a stagewise cooling, i.e., letting the sequence fixed to $\lambda_0$ for $500$ iterations in the first stage, use the moderate cooling $\tau = 0.35$ in the second stage until iteration $1000$ and then the cooling $\tau = 0.55$ afterwards.

 \begin{figure}[!p]
\def\figfolder{experiments/parameters_syntheticImage/}
\begin{subfigure}[c]{3ex}
\rotatebox[origin=c]{90}{Original and noisy}
\end{subfigure}
\def\figurewidth{0.22\textwidth}
\def\hs{\hspace{1ex}}
\begin{subfigure}[c]{\figurewidth}
\includegraphics[width=1\textwidth]{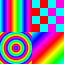}
\end{subfigure}
\hs
\begin{subfigure}[c]{\figurewidth}
\includegraphics[width=1\textwidth]{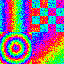}
\end{subfigure}
\\[3ex]
\foreach \j in {1,...,4}{
	\begin{subfigure}[c]{3ex}
	\rotatebox[origin=c]{90}{$s = \protect\input{\figfolder TGV_s-i1-j\j.txt}$}
	\end{subfigure}
	\foreach \i in {1,...,4}{
	\hs
	\begin{subfigure}[c]{\figurewidth}
	\centering
	\includegraphics[width=1\textwidth]{\figfolder exp_synthImage_TGVM-i\i-j\j.png}
	\end{subfigure}
	}
	\\[1ex]
}
\begin{subfigure}[c]{3ex}
\phantom{x}
\end{subfigure}
\foreach \i in {1,...,4}{
		\hs
		\begin{subfigure}[c]{\figurewidth}
		\centering
		$r = \protect\input{\figfolder TGV_r-i\i-j1.txt}$ 
	\end{subfigure}
}
\caption{
Effect of the model parameters of S-TGV using the parametrization by $r >0 $ and $s \in (0,1)$ according to \eqref{eq:model_parameter_definition} for an $S^1$-valued image.
A higher value of $r$ results in a stronger smoothing.
For small values of $s,$ the edges are well-preserved but some staircasing effects appear. For high values of $s,$ the linear trends are recovered but the edges are smoothed out. When using an intermediate value such as $s =0.3,$ we observe a combination of  positive effects of $\TV$ and $\TV^2$ regularization: rather sharp edges and good recovery of linear trends.
}
\label{fig:S1_synth_parameters}
\end{figure}

\subsection{Numerical evaluation of the proposed models}

\begin{figure}[!t]
\centering
\def\figfolderA{experiments/compare_Shild_PT_S1/}
\def\figfolderB{experiments/compare_Shild_PT_S2_2/}
\def\figfolderC{experiments/compare_Shild_PT_Pos3_4/}
\def\hs{\hfill}
\def\vs{\vspace{0.03\textwidth}}
\def\figurewidth{0.32\textwidth}
\def\figurewidthB{0.25\textwidth}
\def\figurewidthC{0.3\textwidth}
\begin{subfigure}[t]{\textwidth}
\centering
\includegraphics[width=\figurewidth]{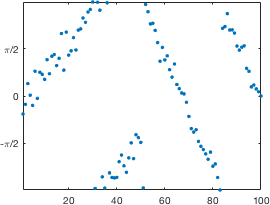}
\hs
\includegraphics[width=\figurewidth]{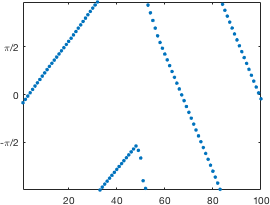} 
\hs
\includegraphics[width=\figurewidth]{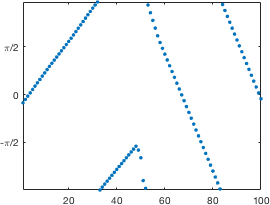}
\caption{$S^1$-valued signal, visualized by phase angle.}
\end{subfigure}
\\[0ex]
\begin{subfigure}[t]{\textwidth}
\centering
\includegraphics[width=\figurewidthB, trim={5cm 2cm 5cm 1cm}, clip]{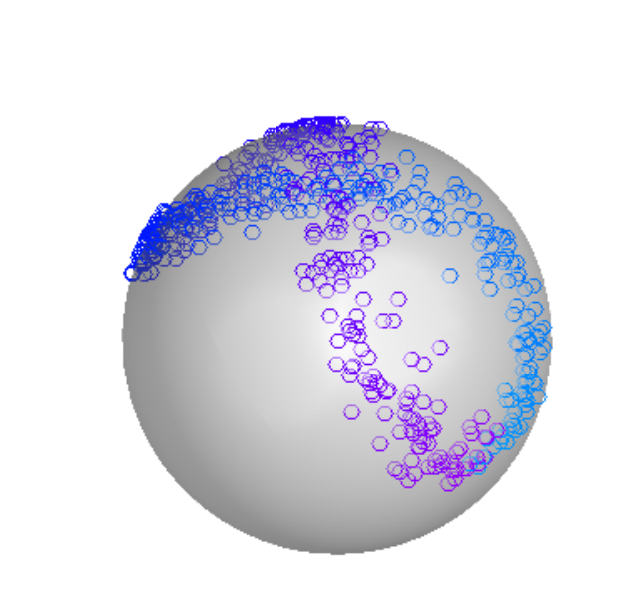}
\hs
\includegraphics[width=\figurewidthB, trim={5cm 2cm 5cm 1cm}, clip]{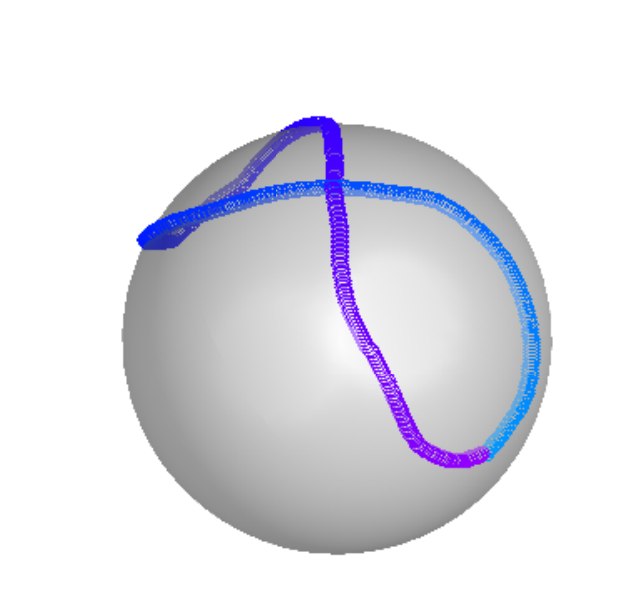} 
\hs
\includegraphics[width=\figurewidthB, trim={5cm 2cm 5cm 2cm}, clip]{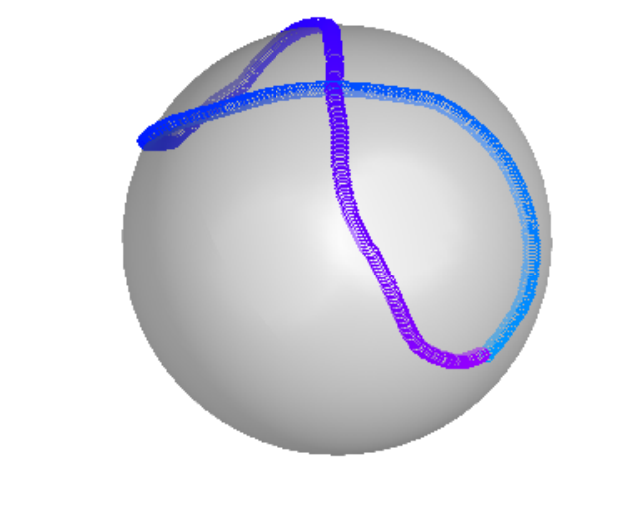}
\caption{$S^2$-valued signal, visualized as scatter plot on the sphere.} 
\end{subfigure}\\[8ex]
\begin{subfigure}[t]{\textwidth}
\centering
\includegraphics[width=\figurewidthC]{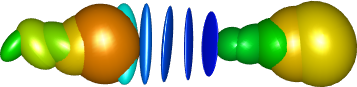}
\hs
\includegraphics[width=\figurewidthC]{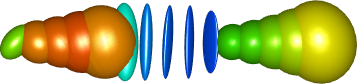}
\hs
\includegraphics[width=\figurewidthC]{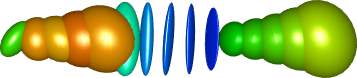}\\[2ex]
\caption{$\Pos_3$-valued signal, visualized as ellipsoids. 
}
\end{subfigure}
\\[2ex]
\begin{subfigure}[t]{\textwidth}
\end{subfigure}
	\caption{
	Comparison of parallel transport and Schild variant of $\MTGV$ for various data spaces.
	The subfigures show the noisy data  \emph{(left)}, 
	the result of the Schild variant S-TGV \emph{(center)},
	and the result of the parallel transport variant PT-TGV \emph{(right)}.}
\label{fig:PTvsSchild}
\end{figure}

\begin{figure}
\def\figurewidth{1\textwidth}
\tikzstyle{myspy}=[spy using outlines={gray,lens={scale=2.5},width=0.15\textwidth, height=0.35\textwidth, connect spies, every spy on node/.append style={thick}}]

\centering
\def\subfigwidth{0.48\textwidth}
\def\nodeSpy{(2.8, -2.3)}
\def\nodeWindow{(0.5,-1.3)}
\def\nodeSpyB{(-0.8, 2.1)}
\def\nodeWindowB{(-1.7,1.7)}

\begin{subfigure}[t]{\subfigwidth}
	\begin{tikzpicture}[myspy]
	\node{\includegraphics[width=\figurewidth]{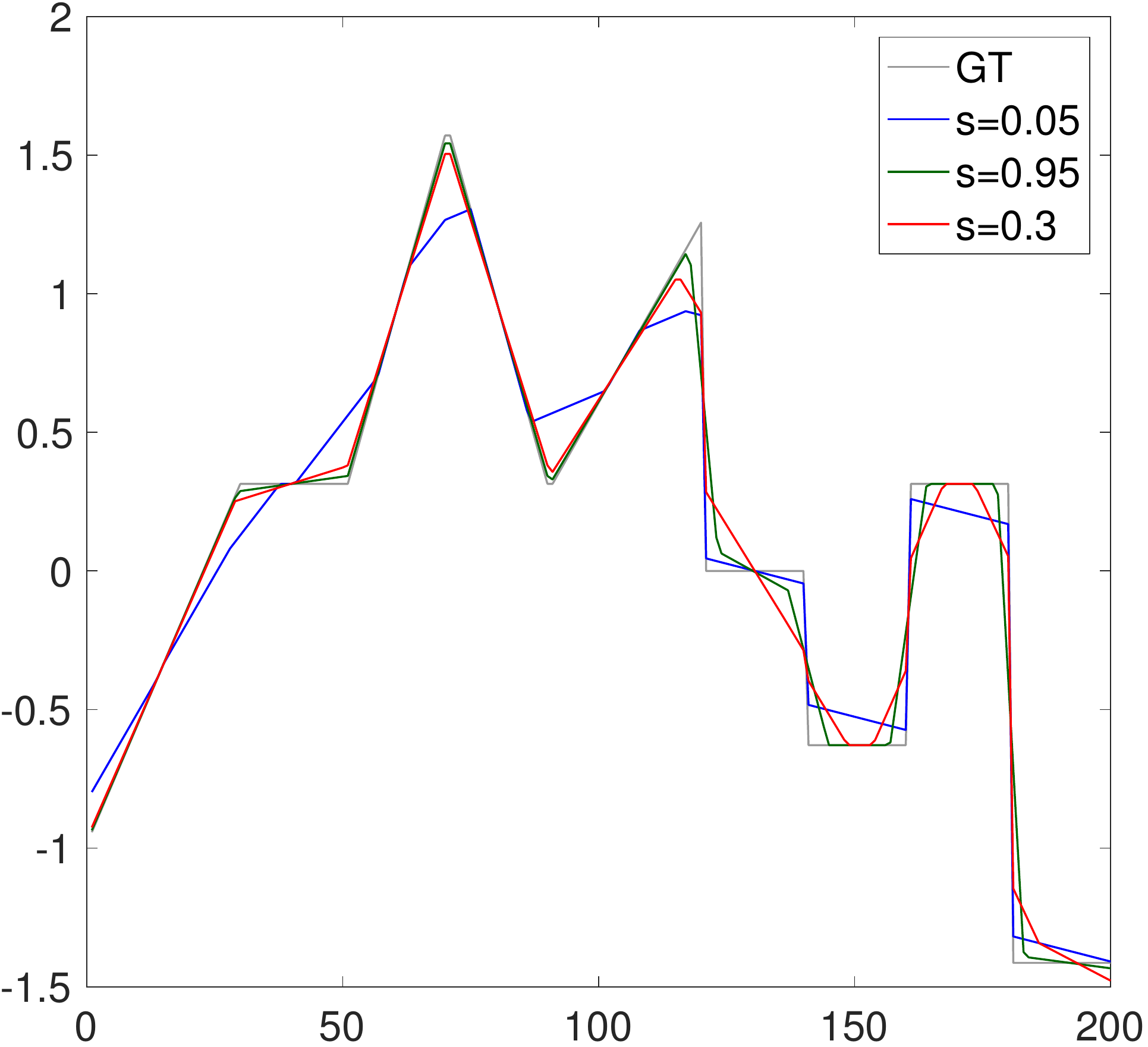}};
	\spy on \nodeSpy in node [left] at \nodeWindow;
	\spy on \nodeSpyB in node [left] at \nodeWindowB;
\end{tikzpicture}
\end{subfigure}
\hfill
\begin{subfigure}[t]{\subfigwidth}
	\begin{tikzpicture}[myspy]
	\node{\includegraphics[width=\figurewidth]{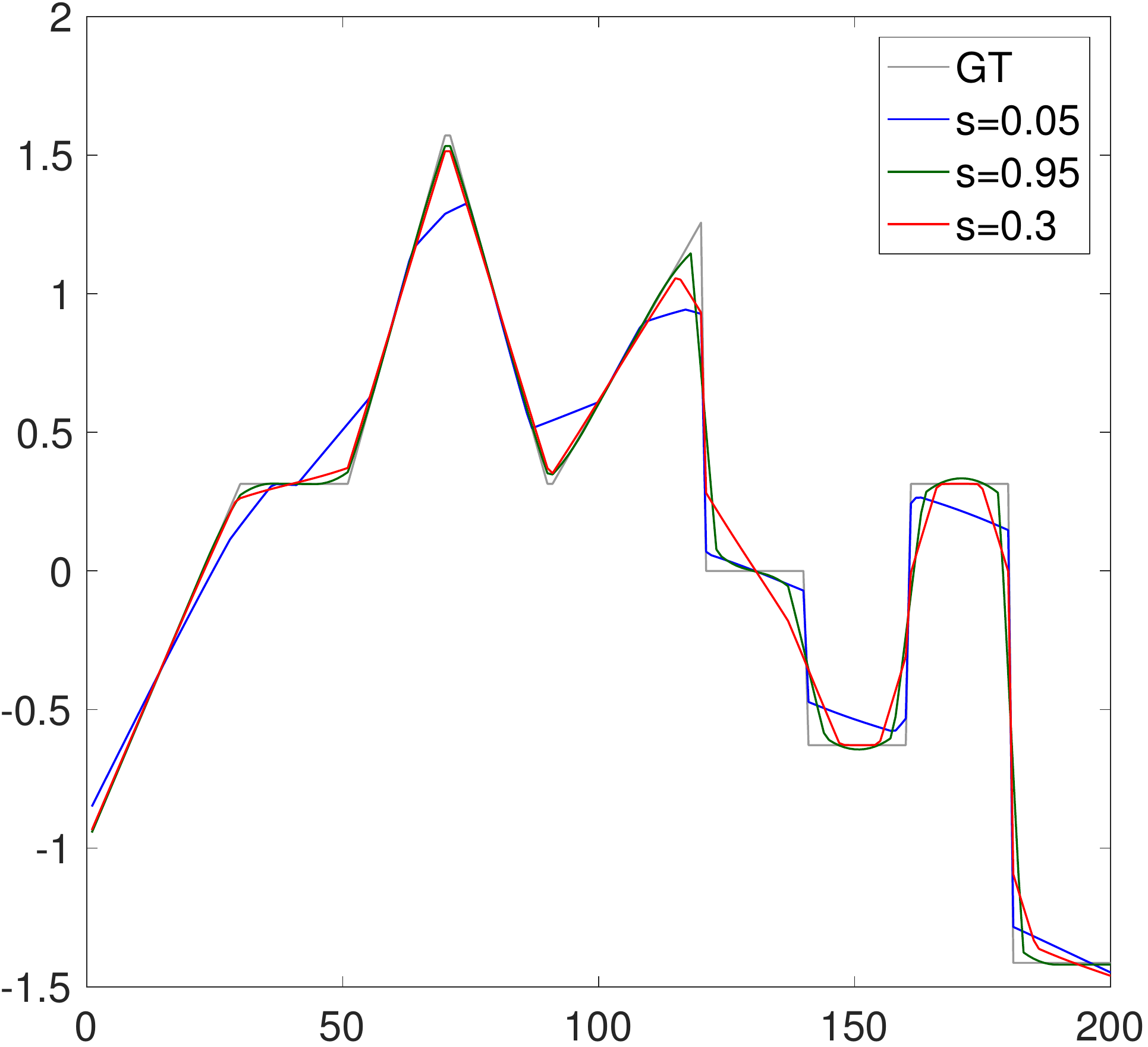}};
	\spy on \nodeSpy in node [left] at \nodeWindow;
	\spy on \nodeSpyB in node [left] at \nodeWindowB;
\end{tikzpicture}
\end{subfigure}
\caption{\label{fig:man_vec_comparison} Comparison of the approximate vector-space ground truth and the result obtained with the manifold TGV implementation for univariate $S^ 1$ data and different ratios of $s$. \emph{Left:} Approximate ground truth obtained with reference vector space implementation. \emph{Right:} Result with TGV manifold code. 
}
\end{figure}

\paragraph{Comparison between Schild variant and parallel transport variant.} First we compare the two proposed realizations of manifold TGV: the parallel transport variant
and the Schild variant.
Figure~\ref{fig:PTvsSchild} shows the results with both variants
for some typical univariate signals for the $S^1,$ the $S^2$ and the $\Pos_3$ manifold.
(The $S^1$-valued signal was
corrupted by von Mises noise with concentration parameter $\kappa = 10.$
The $S^2$-valued signal was taken from \cite{bavcak2016second}
and corrupted by applying the exponential mapping of Gaussian distributed tangent vectors as in \cite{bavcak2016second}
with $\sigma = 0.1.$ The $\Pos_3$-valued signal was corrupted
by applying the exponential
mapping of Gaussian noise distributed on the tangent vectors with $\sigma = 0.2.$
The $r$-parameters were chosen as $r= 1$ for the spherical signals and as $r = 0.2$
for the $\Pos_3$ signal. In all cases $s = 0.3$.)
The results for the spherical data appear very similar.
For the $\Pos_3$ manifolds, 
there are slight visible differences at some points, e.g. near the discontinuity of the second to fourth position of the  signal and at the last position.
In summary, we observe a qualitatively similar result
in the sense that the geodesic parts are well reconstructed and that the edges are preserved.
In the following, we focus on the Schild variant.

\paragraph{Comparison of manifold TGV with vector-space TGV.}
In order to validate both our generalizations of TGV to the manifold setting as well as numerical feasibility of our optimization algorithms, we carry out a comparison to the vector-space case. This is done for $S^1$, the unit sphere in $\R^2$. By generating a signal with values that are strictly contained in one hemisphere, we can unroll the signal and compare to vector-space TGV-denoising in a way that the same results can be expected. We carried out this comparison for synthetic data without noise and different values of the balancing parameter $s$. 
We tested the setting of $s=0.3$, which comprises a good balance between first and second order terms, as well as the rather extreme settings of $s=0.05$ and $s = 0.95$. The regularization parameter $r$ was fixed to $r=1$.
In order to approximate the ground-truth solution of second-order TGV denoising in the vector space setting, we implemented the Chambolle-Pock algorithm \cite{Chambolle11} for this situation. To ensure a close approximation of the ground truth for all parameter settings, we carried out a dedicate stepsize tuning to accelerate convergence of the algorithm, computed $2\times 10^{5}$ iterations and ensured optimality by measuring the duality gap.

The result of this evaluation can be found in Figure \ref{fig:man_vec_comparison}. It can be observed that the qualitative properties of the numerical solution obtained with the manifold-TGV code are similar to the ones of the approximate ground-truth of the vector space setting, confirming overall feasibility of our model and implementation.

For the case $s = 0.05$, one can see in particular in the right part of the bottom plot (starting with the first plateau after 180 on the x-axis), that the solution is piecewise constant up to a linear part, which is approximately the same for all four plateaus. This is what one would expect for the extreme case $s\rightarrow 0$, as in this case, $\TGV$ minimization mimics TV minimization up to a globally linear term. On the other hand, for the case $s = 0.95$, one can see that the solution is piecewise linear with no jumps. This can be expected from second-order TV minimization, which coincides with second order $\TGV$ minimization for $s \rightarrow 1$. In particular, the piecewise linear part on the left is approximated well, while the plateaus on the right are not well captured. The case $s=0.3$ provides a good compromise here: The linear parts on the left are still well captured, but the solution still admits jumps on the plateau parts, as can be seen in particular outermost right jump after point. 

We also note that there are still some differences of the solution obtained with the manifold code and the Chambolle-Pock result interpreted as approximate ground truth of the vector space case, in particular for the cases $s=0.05$ and $s=0.95$. We believe that this is mostly an issue of the algorithmic realization rather than the model itself caused by the numerical solutions obtained for the more extreme cases $s=0.05$ and $s=0.95$.

\paragraph{Basic situations.} The aim of this experiment is to investigate the performance of the proposed manifold-TGV model for different basic situations and noise levels. The experiment is carried out on a univariate signal on the two-sphere $S^2$, that comprises jumps in the signal as well as its derivative and is composed of piecewise constant and geodesic signals.

The results of the experiment for different noise levels (no noise, intermediate noise, strong noise), a fixed value of $r=1$ and different values of $s$, namely $s \in \{ 0.05,0.95,0.3\}$, can be found in Figure \ref{fig:s2_basic_situations}. As can be seen there, the manifold-TGV functional regularizes the data quite well and is able to achieve good results, even in the case with relatively strong noise where it is hard to see any structure in the data. As in the previous experiment, the choice $s=0.05$ promotes piecewise constant solutions, which is naturally best for regions where the signal is piecewise constant, but leads to ``staircasing'' in smooth regions, as can be particularly seen in the geodesic parts of the strong-noise case. The choice $s=0.95$ approximates geodesics well, but does not allow for jumps and also produces some oscillations in the smooth parts of the case with strong noise. Again the choice $s=0.3$ is a good compromise. Even though it does not reconstruct jumps as well as the TV-like version, it still allows for jumps and reconstructs the geodesic parts rather accurately. 

\begin{figure} 
\centering
\includegraphics[width=0.24\linewidth]{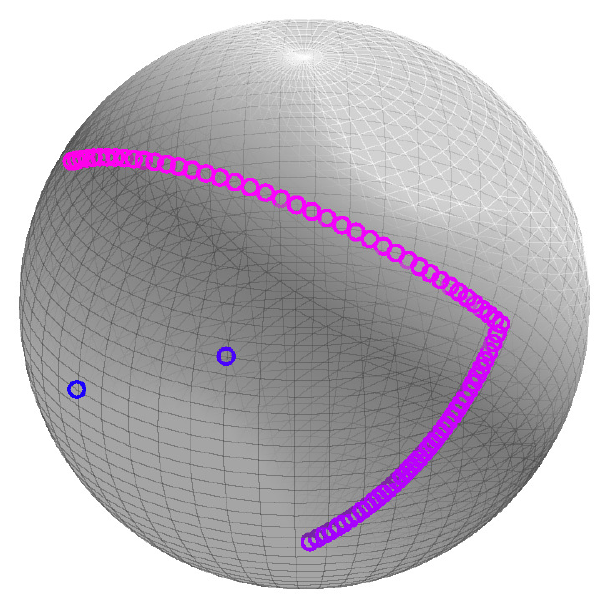}
\includegraphics[width=0.24\linewidth]{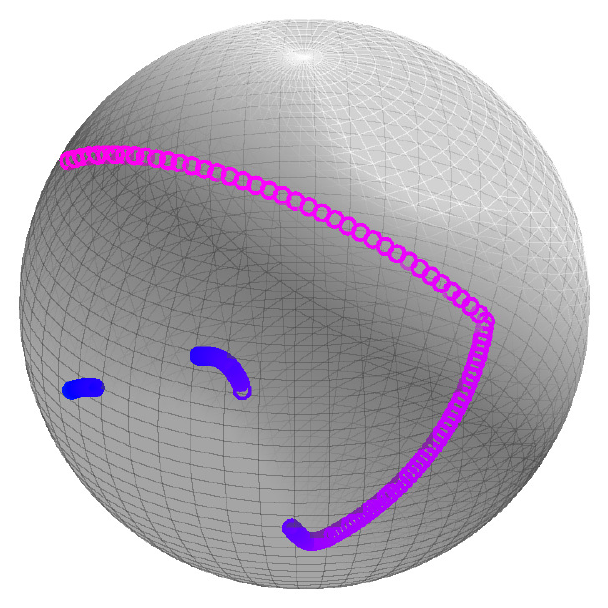}
\includegraphics[width=0.24\linewidth]{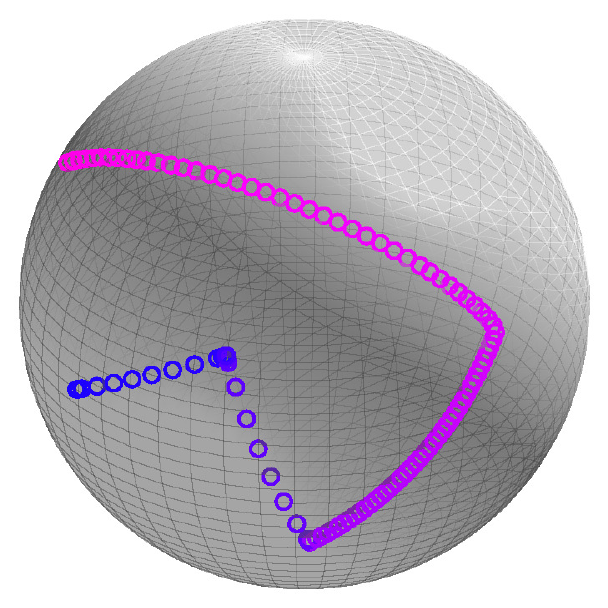}
\includegraphics[width=0.24\linewidth]{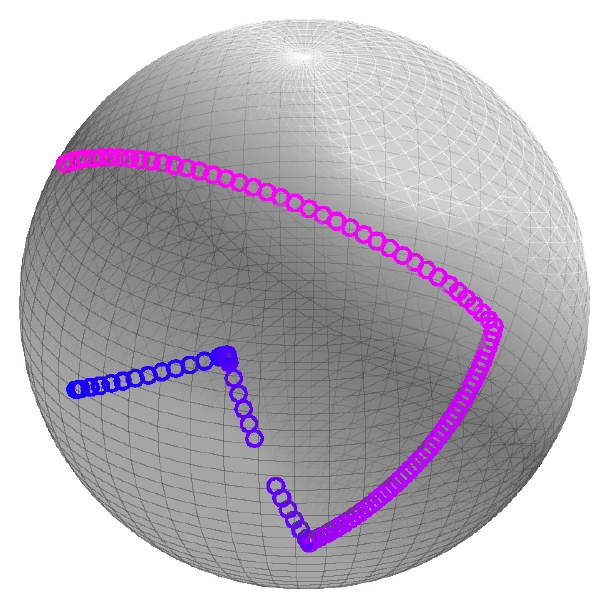}
 
\includegraphics[width=0.24\linewidth]{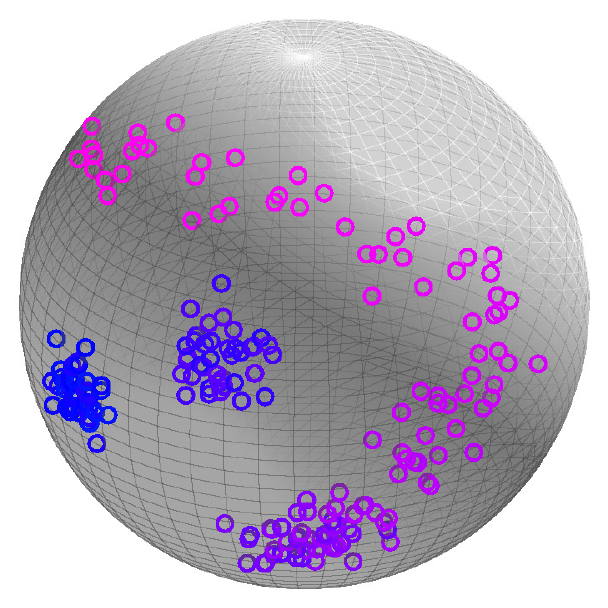}
\includegraphics[width=0.24\linewidth]{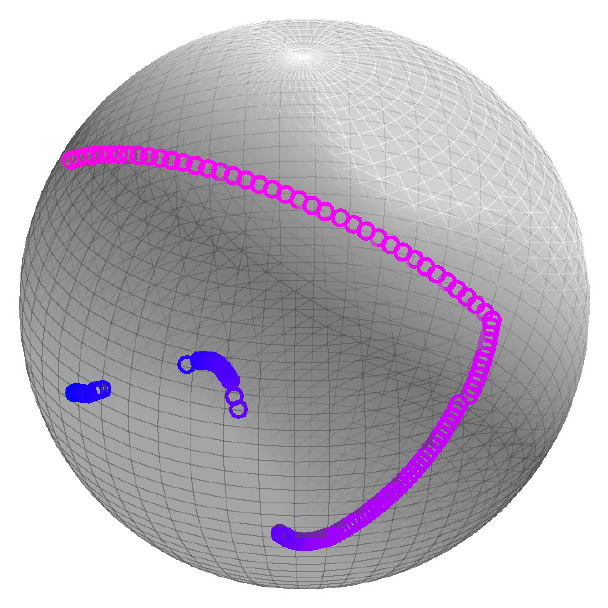}
\includegraphics[width=0.24\linewidth]{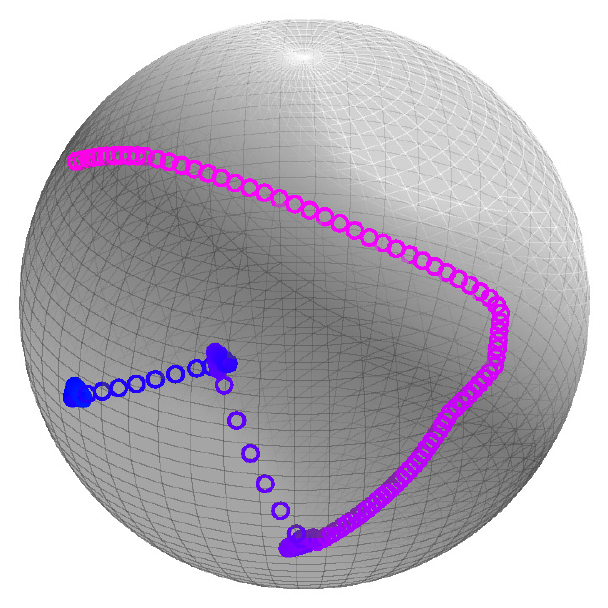}
\includegraphics[width=0.24\linewidth]{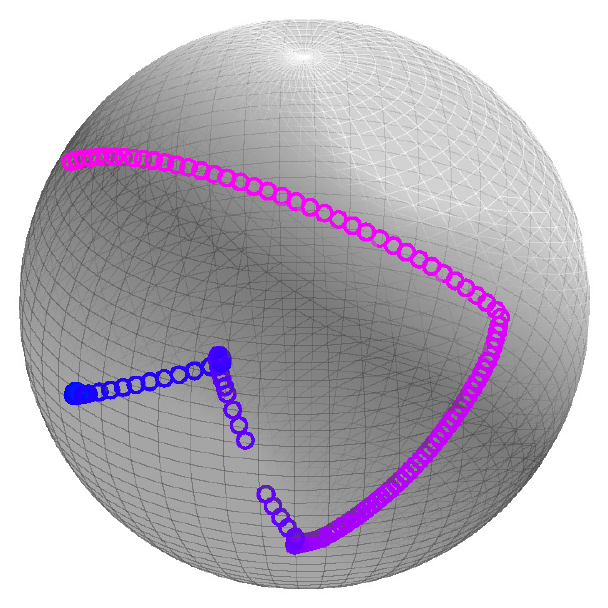}
 
\includegraphics[width=0.24\linewidth]{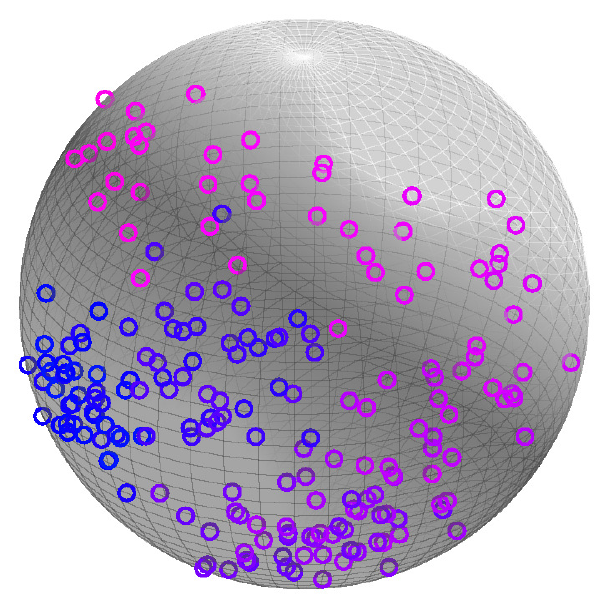}
\includegraphics[width=0.24\linewidth]{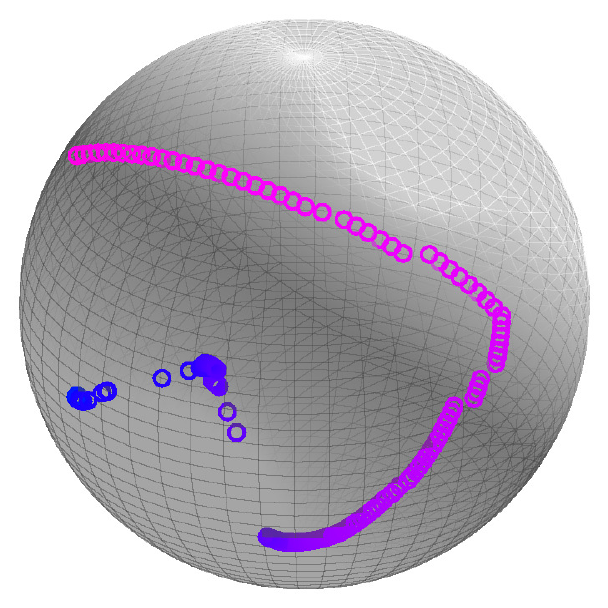}
\includegraphics[width=0.24\linewidth]{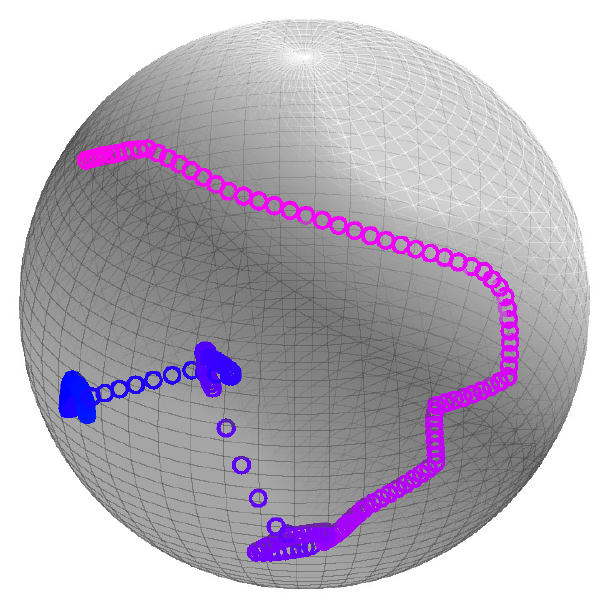}
\includegraphics[width=0.24\linewidth]{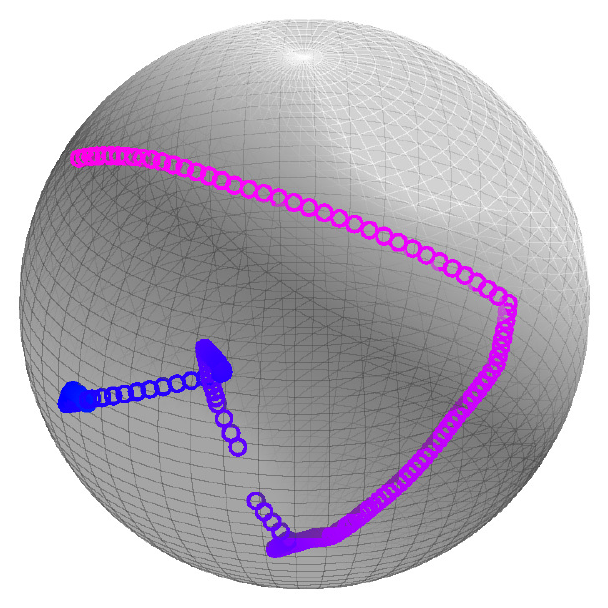}
\caption{\label{fig:s2_basic_situations} Results for univariate $S^2$ data. \emph{Top}: Ground truth data and reconstruction. \emph{Middle}: Data with intermediate noise and reconstruction. \emph{Bottom}: Data with strong noise and reconstruction. \emph{From left to right}: Data, $\STGV$ regularized reconstruction for $s = 0.05$, $s = 0.95$, $s=0.3$, respectively. The color gradation indicates the ordering of the signal.}
\end{figure}

\subsection{Results for synthetic  data}

Next we investigate the denoising performance of $\MTGV$
on synthetic images.
We compare the results of the proposed algorithm with
the results of manifold TV regularization \cite{weinmann2014total}, 
and with the result of second-order manifold TV, denoted by $\TV^2$ \cite{bavcak2016second}.
The model parameter of first-order TV is denoted by $\alpha$
and that of $\TV^2$ by $\beta.$
For all algorithms, $1000$ iterations were used.

First we look at $S^2$-valued images.
As in~\cite{bavcak2016second}, the noisy data $f$ were created from the original image $g$ by
$f_{ij} = \exp_{g_{ij}} \eta_{ij}$ 
where  $\eta_{ij}$ is a tangent vector at $g_{ij}$
and both its  components are Gaussian distributed
with standard deviation $\sigma = \frac{4}{45}\pi.$
For comparison with first order TV 
we scanned the parameter range $\alpha = 0.1, 0.2, \ldots, 1$
and the special parameters $\alpha = 0.22$ and $\alpha = 3.5 \cdot 10^{-2}$
given in~\cite{bavcak2016second} and the corresponding implementation.
Similarly, for $\TV^2$ we scanned the parameter range $\beta = 1, 5, 10, 15 \ldots, 30,$
and the  special parameters $\beta = 8.6$ and $\beta = 29.5$ suggested in~\cite{bavcak2016second}. 
As no beneficial effect of combining first and second-order TV was observed
in \cite{bavcak2016second}, we used pure $\TV$ and $\TV^2$ regularization.
For the proposed method, we fixed $s=0.3$ and report 
the best result among the six parameters $r = 0.05, 0.1, 0.15, \ldots, 0.3.$ 
The results in  Figure~\ref{fig:S2_synth} show that the second-order methods
 give a significantly smoother result than first order TV
 and that they do not suffer from staircasing effects.
On the flipside, $\TV^2$ regularization results in an undesired smoothing effect of the edges between the tiles,
best seen at the bottom left tile.
The proposed TGV regularization  preserves these edges
which results in an improved reconstruction quality.

 \begin{figure}
\centering
\def\figfolder{experiments/compare_with_TV_2D_S2_5/}
\def\figurewidth{0.3\textwidth}
\def\hs{\hspace{0.03\textwidth}}
\begin{subfigure}[t]{\figurewidth}
\includegraphics[width=1\textwidth]{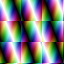}
\caption{Original.}
\end{subfigure}
\hs 
\begin{subfigure}[t]{\figurewidth}
\includegraphics[width=1\textwidth]{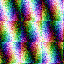}
\caption{Noisy data.}
\end{subfigure}
\hs 
\begin{subfigure}[t]{\figurewidth}
\includegraphics[width=1\textwidth]{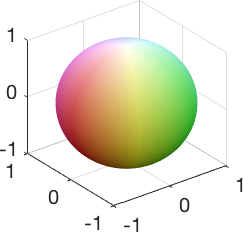}
\caption{Color code for $S^2$.}
\end{subfigure}
\\[2ex]
\begin{subfigure}[t]{\figurewidth}
\includegraphics[width=1\textwidth]{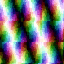}
\caption{TV with $\alpha = \protect\input{\figfolder TV_alpha.txt},$ \\
$\deltaSNR = \protect\input{\figfolder deltaSNR_TV.txt} \dB$ }
\end{subfigure} 
\hs
\begin{subfigure}[t]{\figurewidth}
\includegraphics[width=1\textwidth]{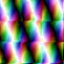}
\caption{$\TV^2$ with 
$\beta = \protect\input{\figfolder TV2_beta.txt},$ \\
$\deltaSNR = \protect\input{\figfolder deltaSNR_TV2.txt} \dB$ }
\end{subfigure} 
\hs
\begin{subfigure}[t]{\figurewidth}
\includegraphics[width=1\textwidth]{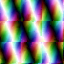}
\caption{S-TGV with 
$r = \protect\input{\figfolder TGV_r.txt},$ \\
$\deltaSNR = \protect\input{\figfolder deltaSNR_TGV.txt} \dB$}
\end{subfigure}
	\caption{Comparison of S-TGV with first and second-order total variation 
	on an $S^2$-valued image from~\cite{bavcak2016second}. The spherical values are visualized according 
	to the scheme (c) which means that the north pole is white, the south pole is black 
	and the equator gets hue values according to its longitude.
	}
\label{fig:S2_synth}
\end{figure}

Next we look at $\Pos_3$-valued images.
Such images appear naturally in diffusion tensor imaging (DTI),
 so we briefly describe the setup.
 DTI is a medical imaging modality
 based on diffusion weighted magnetic resonance images (DWI)
which for example allows to reconstruct fiber tract orientations \cite{basser1994mr}.
A DWI captures the diffusivity of water molecules with respect to a direction $v \in \R^3.$
The relation between a diffusion tensor $f(p)$ and the DWIs $D_v(p)$ at the voxel $p$ is given by the Stejskal-Tanner equation
\begin{equation} \label{eq:StejTan}
   D_v(p) =  A_0 e^{- b \ v^\top f(p) v} 
\end{equation}
with  constants $b,A_0>0.$
A standard noise model for the DWIs is the Rician model
which assumes a complex-valued Gaussian noise in the original frequency domain measurements \cite{basu2006rician}.
This means that assuming the model \eqref{eq:StejTan} the actual measurement in direction $v$ at pixel $p$ is given by 
$
D'_v(p) = ((X+ D_v(p))^2 + Y^2)^{1/2},$ with the Gaussian random variables $ X,Y \sim N(0,\sigma^2).$ 
Building on this model, we impose the noise as follows.
From the synthetic tensor image, we  compute  DWIs according to \eqref{eq:StejTan} with respect to $21$ different directions.
Then we impose Rician noise to all derived DWIs,
and we estimate the corresponding diffusion tensor image $f$
using a least squares approach on \eqref{eq:StejTan}.
Due to the noise, some of the fitted tensors might not be positive definite
and thus have no meaningful interpretation.
To handle such invalid tensor
we exclude them from the data fidelity term.
On the other hand, we keep the corresponding pixels 
in the regularizing term so that we achieve a reasonable inpainting.
In Figure~\ref{fig:Pos3_synth}, we illustrate the  effects of the regularization for a synthetic $\Pos_3$-valued image. 
For comparison with (combined) first and second-order TV,
we scanned all combinations of the parameters $\alpha = 0,   0.01, 0.035,    0.1, 0.3125,  0.375,$
and $\beta =   0.02, 0.05, 0.125,  0.625.$
As before, this comprises the parameters used in \cite{bavcak2016second} and the corresponding implementation for that manifold.
For the proposed method, we report the best result among the  possible  combinations of the parameters $r = 0.06, 0.08,0.09, 0.1, 0.12$ and $s = 0.3, 0.35, 0.4, 0.45.$
In the example we observe that combined $\TV$/$\TV^2$ regularization yields a similar reconstruction quality as pure first order $\TV,$
whereas S-TGV gives a significantly higher reconstruction quality.

\begin{figure}
\centering
\def\figfolder{experiments/compare_with_TV_2D_Pos3_BZ4_2/}
\def\figurewidth{0.3\textwidth}
\def\hs{\hspace{0.03\textwidth}}
\begin{subfigure}[t]{\figurewidth}
\includegraphics[width=1\textwidth]{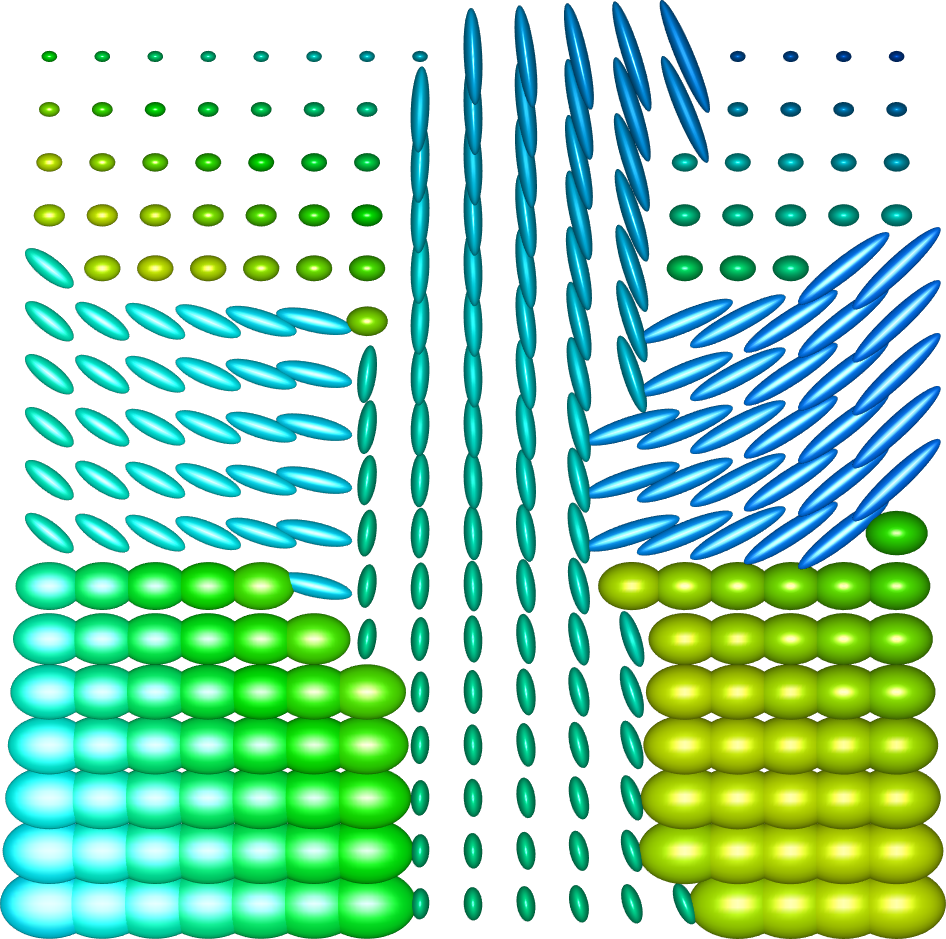}
\caption{Original data}
\end{subfigure}
\hs 
\begin{subfigure}[t]{\figurewidth}
\includegraphics[width=1\textwidth]{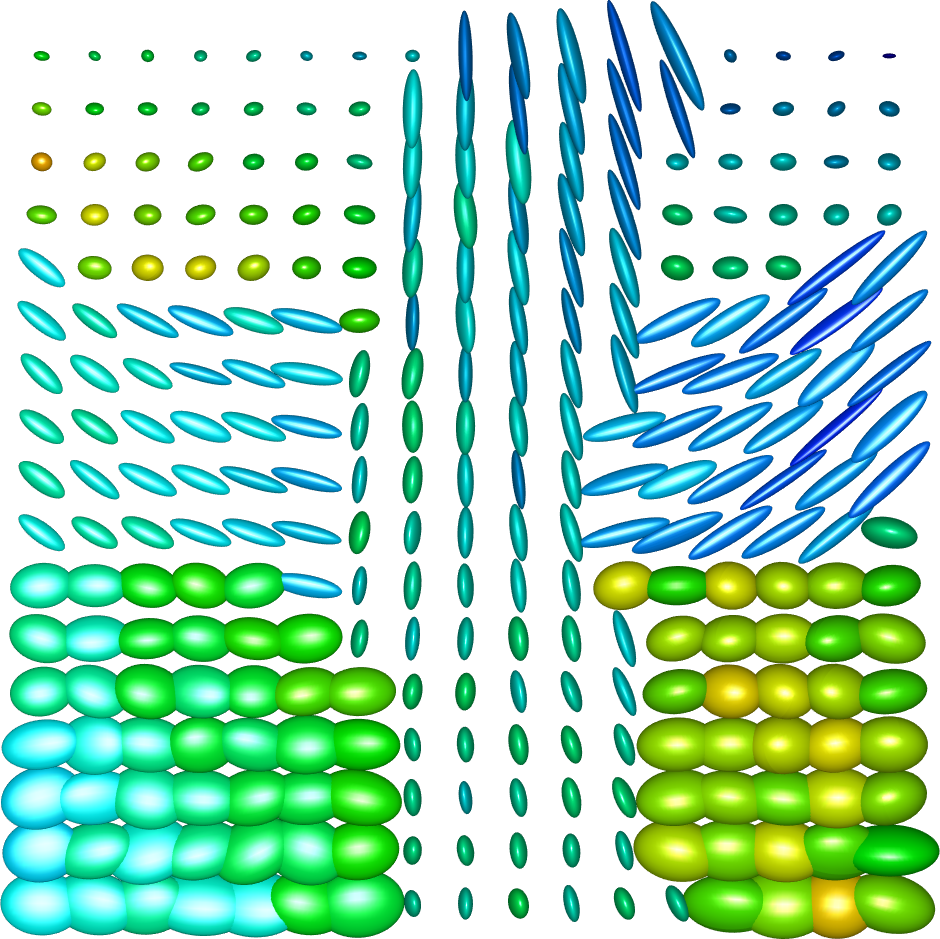}
\caption{Data with Rician noise}
\end{subfigure}
\\[2ex]
\begin{subfigure}[t]{\figurewidth}
\includegraphics[width=1\textwidth]{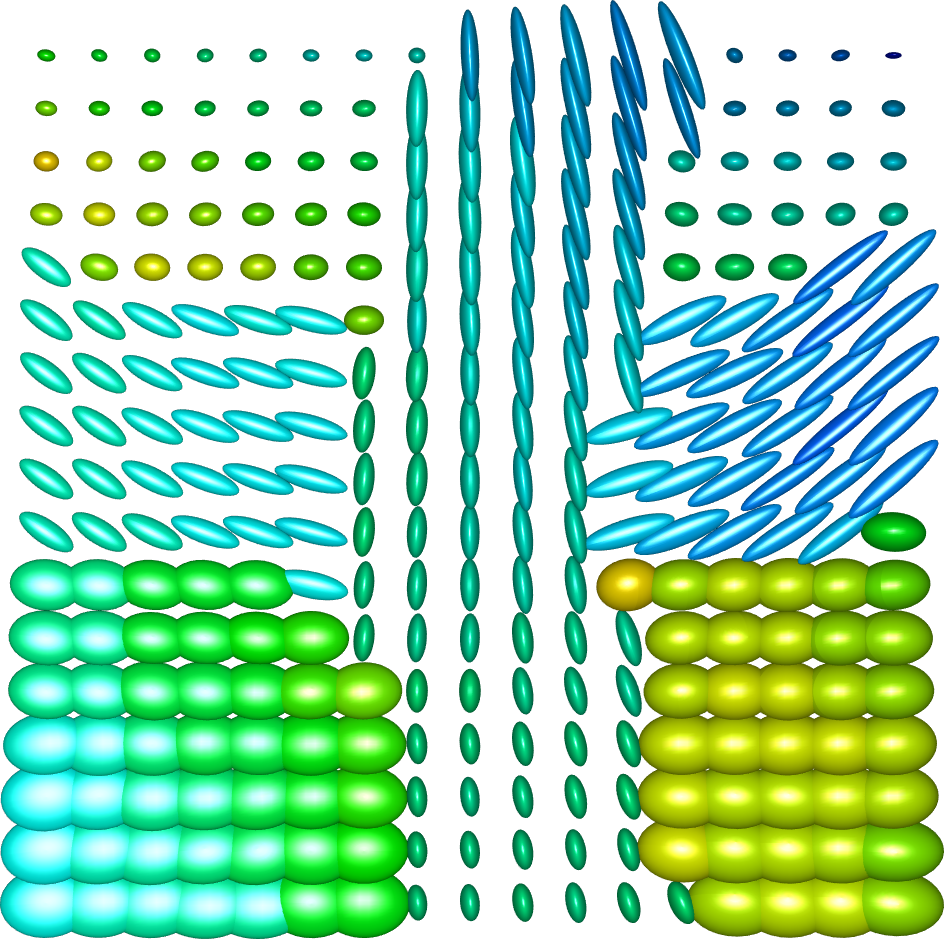}
\caption{$\TV$ 
$\alpha = \protect\input{\figfolder TV_alpha.txt},$ \\
$\deltaSNR = \protect\input{\figfolder deltaSNR_TV.txt} \dB$ }
\end{subfigure} 
\hs
\begin{subfigure}[t]{\figurewidth}
\includegraphics[width=1\textwidth]{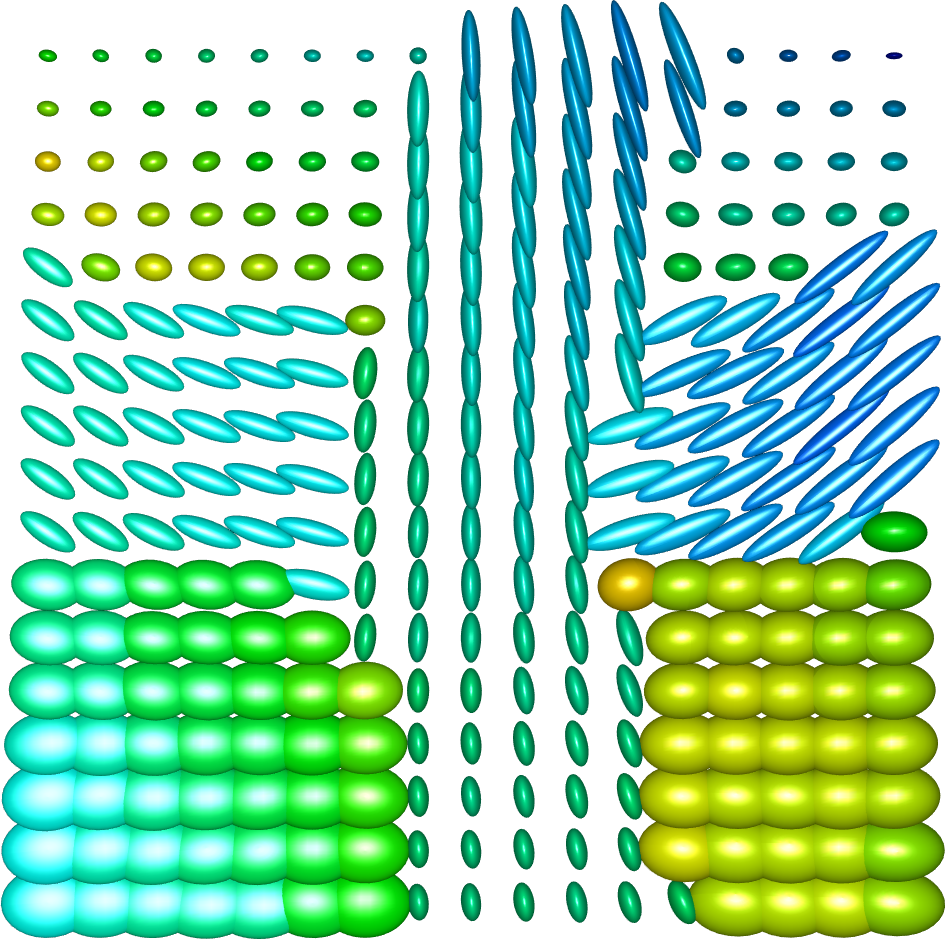}
\caption{$\TV, \TV^2,$  
$\alpha = \protect\input{\figfolder TV_TV2_alpha.txt},$ 
$\beta = \protect\input{\figfolder TV_TV2_beta.txt},$ \\
$\deltaSNR = \protect\input{\figfolder deltaSNR_TV_TV2.txt} \dB$ }
\end{subfigure} 
\hs
\begin{subfigure}[t]{\figurewidth}
\includegraphics[width=1\textwidth]{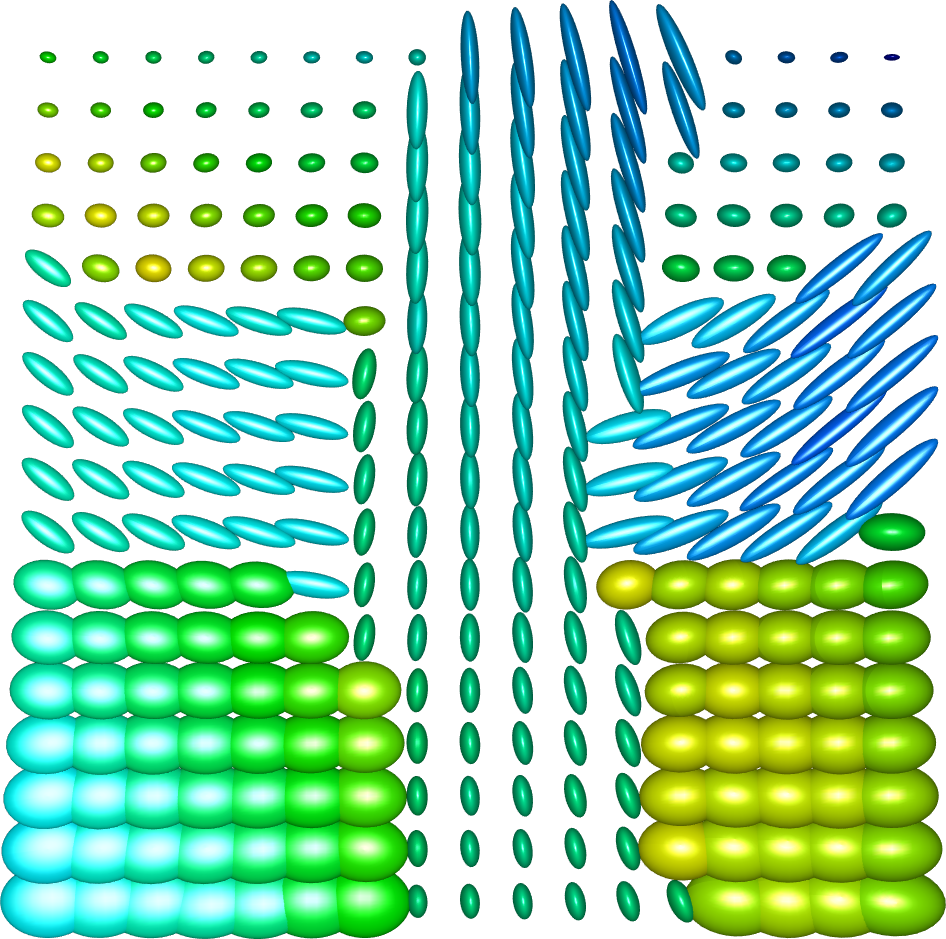}
\caption{$\operatorname{S-TGV},$ $s = \protect\input{\figfolder TGV_s.txt},$ $r = \protect\input{\figfolder TGV_r.txt},$
$\deltaSNR = \protect\input{\figfolder deltaSNR_TGV.txt} \dB$}
\end{subfigure}
	\caption{%
	Denoising results for a synthetic $\Pos_3$-valued image corrupted by Rician noise.
	Combined $\TV/\TV^2$-regularization yields a similar reconstruction quality as first order $\TV,$
	whereas S-TGV gives a significantly higher reconstruction quality.
	}
\label{fig:Pos3_synth}
\end{figure}

\subsection{Results for real data}

We illustrate the effects of TGV regularization on real manifold-valued signals and images.

First we look at smoothing a time series of wind directions.
The natural data space for a signal of orientations is the unit circle $S^1.$
The present data shows wind directions 
recorded every hour in 2016 by the weather station SAUF1, St.~Augustine, FL.; see Figure~\ref{fig:wind}\footnote{Data available at \url{http://www.ndbc.noaa.gov/historical_data.shtml}.}.
We observe that the proposed method smoothes the orientations 
while respecting the phase jump from $-\pi$ to $\pi$
and preserving linear trends in the data.

\begin{figure}
\centering
	\def\figfolder{experiments/S1_wind_directions_rev/}
		\def\figurewidth{1\textwidth}
\def\plottitle{}
\includegraphics[width=\figurewidth]{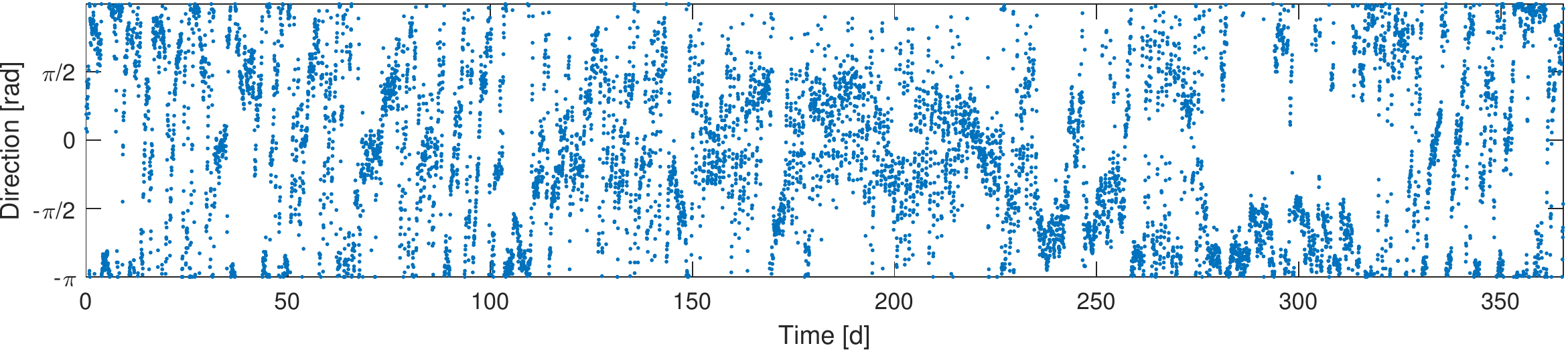} \\[1ex]
\includegraphics[width=\figurewidth]{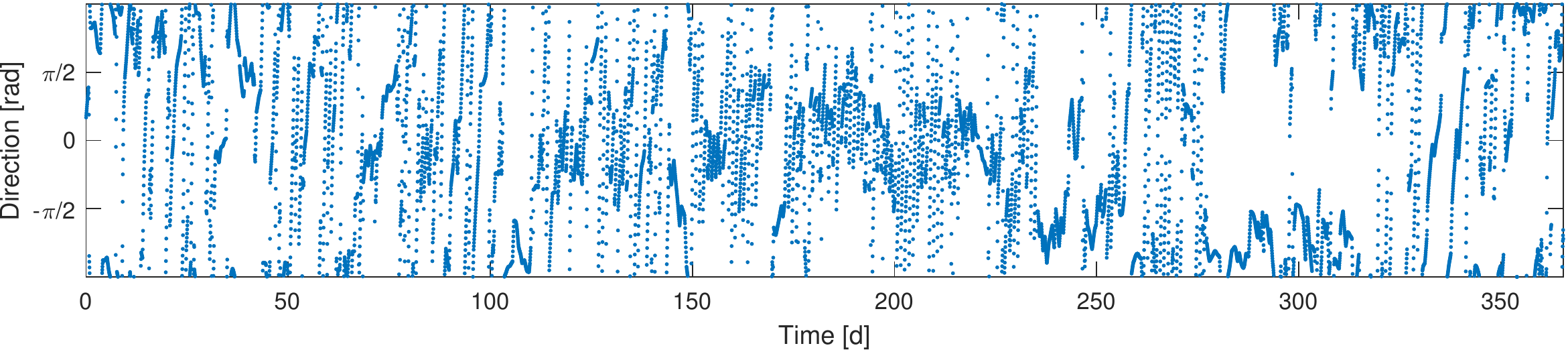}\\[-1ex] 
	\caption{\emph{Top:} Hourly wind directions 
	at weather station SAUF1 (St.~Augustine, FL) in the year  2016.
	\emph{Bottom:} S-TGV ($s = \protect\input{\figfolder TGV_s.txt},$ $r = \protect\input{\figfolder TGV_r.txt}$)
	smoothes the signal while not over-smoothing jump discontinuities and properly dealing with the phase jump from $-\pi$ to $\pi$
 and	preserving linear trends in the data.
	}
\label{fig:wind}
\end{figure}

Next we look at smoothing of interferometric synthetic aperture radar (InSAR) images.
Synthetic aperture radar (SAR) is a radar technique for
sensing the earth's surface from measurements taken by aircrafts or satellites. 
InSAR images consist of the phase difference between two SAR images,
recording a region of interest either from two different angles of view or at two different points in time.
Important  applications of InSAR are the creation of accurate digital elevation models and the detection of terrain changes; cf.~\cite{massonnet1998radar,rocca1997overview}.
As InSAR data consists of phase values 
that are defined modulo $2\pi,$ their natural data space is the unit circle.
In Figure  \ref{fig:InSAR}, we illustrate the effect of S-TGV to an InSAR image taken from 
\cite{thiel1997ers}\footnote{Data available at \url{https://earth.esa.int/workshops/ers97/papers/thiel/index-2.html}.}.
We observe a clear smoothing effect while sudden phase changes are preserved.

\begin{figure}
\centering
\def\figfolder{experiments/InSAR_3/}

\centering
\def\figurewidth{0.7\textwidth}

\tikzstyle{myspy}=[spy using outlines={black,lens={scale=5},width=0.35\textwidth, height=0.31\textwidth, connect spies, every spy on node/.append style={thick}}]
\centering
\def\subfigwidth{0.45\textwidth}
\def\nodeSpy{(0.4, 0.75)}
\def\nodeSpyB{(0.5, -0.7)}
\def\nodeWindow{(4.8,1.15)}
\def\nodeWindowB{(4.8,-1.15)}

\begin{subfigure}[t]{\subfigwidth}
	\begin{tikzpicture}[myspy]
	\node {\includegraphics[interpolate=false,width=\figurewidth]{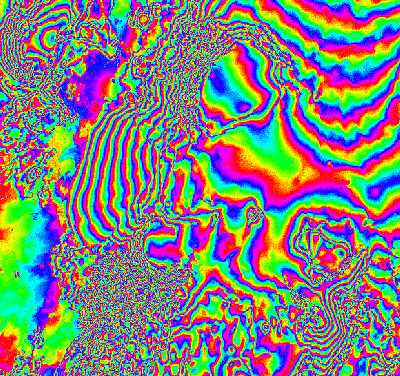}};
	\spy on \nodeSpy in node [left] at \nodeWindow;
	\spy on \nodeSpyB in node [left] at \nodeWindowB;
\end{tikzpicture}
\end{subfigure}
\hfill
\begin{subfigure}[t]{\subfigwidth}
	\begin{tikzpicture}[myspy]
	\node {\includegraphics[interpolate=false,width=\figurewidth]{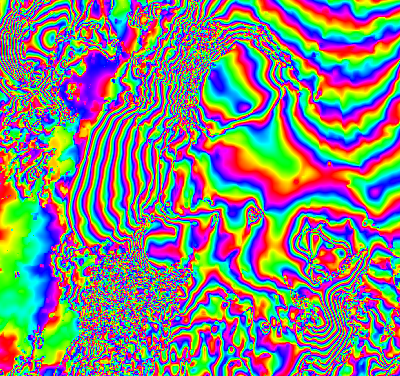}};
	\spy  on \nodeSpy in node [left] at \nodeWindow;
	\spy on \nodeSpyB in node [left] at \nodeWindowB;
	\end{tikzpicture}
\end{subfigure}
	\caption{
	\emph{Left:} InSAR image from \cite{thiel1997ers}. 
	\emph{Right:} Result of S-TGV with $r = \protect\input{\figfolder TGV_r.txt},$ $s = \protect\input{\figfolder TGV_s.txt}.$
	The image is smoothed and sharp edges are preserved.
	}
\label{fig:InSAR}
\end{figure}

At last we consider real DTI data
which was taken from the Camino project \cite{cook2006camino}\footnote{Data available at \url{http://camino.cs.ucl.ac.uk/}}.
In Figure~\ref{fig:camino}, we see an axial slice  of a human brain
{\color{blue} (slice 20 of the Camino data).}
We also display a zoom to the corpus callosum
which connects the right and the left hemisphere.
The original tensors were computed from the diffusion weighted images by a least squares method
based on the Stejskal-Tanner equation.
We observe that the  proposed method 
smoothes the tensors in the corpus callosum 
while it preserves the sharp edges to the 
adjacent tissue.
Also observe the inpainting of the invalid tensors.

\begin{figure}
\centering
\def\figfolder{experiments/pos3_camino_data_5/}
\def\figurewidth{0.45\textwidth}

\tikzstyle{myspy}=[spy using outlines={gray,lens={scale=2.8},width=1\textwidth, height=0.5\textwidth, connect spies, every spy on node/.append style={thick}}]
\centering
\def\figwidth{0.45\columnwidth}
\def\nodeSpy{(-0.1, -1.2)}
\def\nodeWindow{(3.2,-5.5)}
\begin{subfigure}[t]{\figwidth}
	\begin{tikzpicture}[myspy]
	\node {\includegraphics[width=\textwidth]{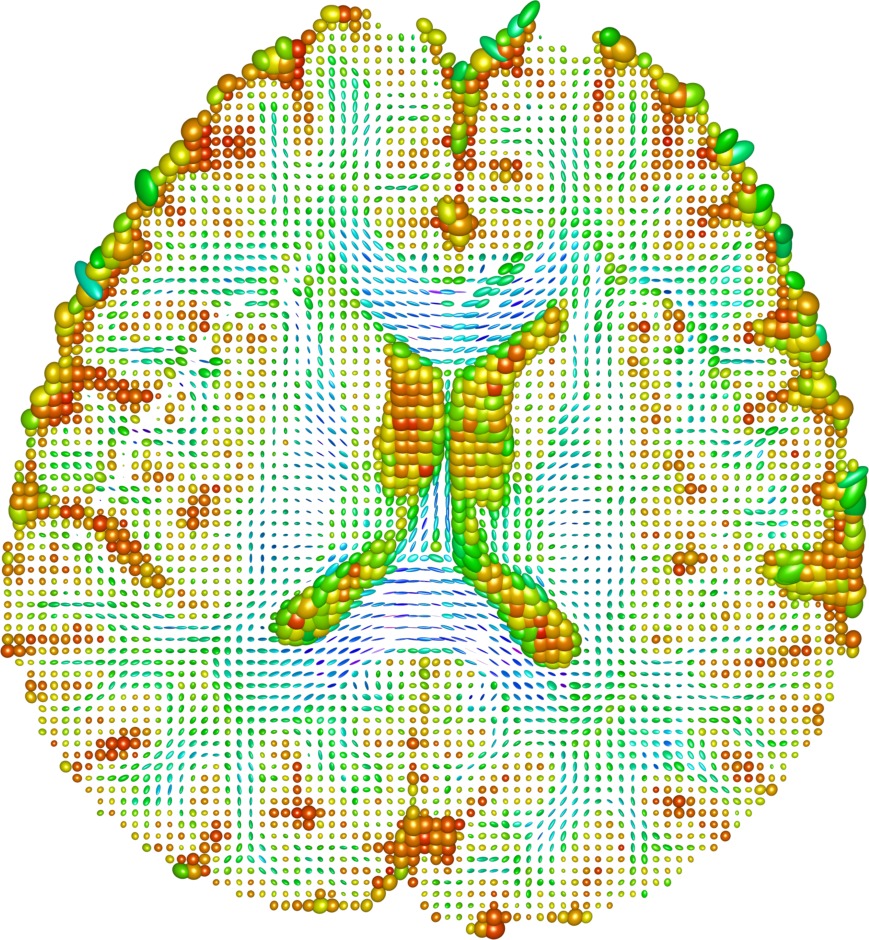}};
	\spy on \nodeSpy in node [left] at \nodeWindow;
\end{tikzpicture}
\end{subfigure}
\hfill
\begin{subfigure}[t]{\figwidth}
	\begin{tikzpicture}[myspy]
	\node {\includegraphics[interpolate=true,width=\textwidth]{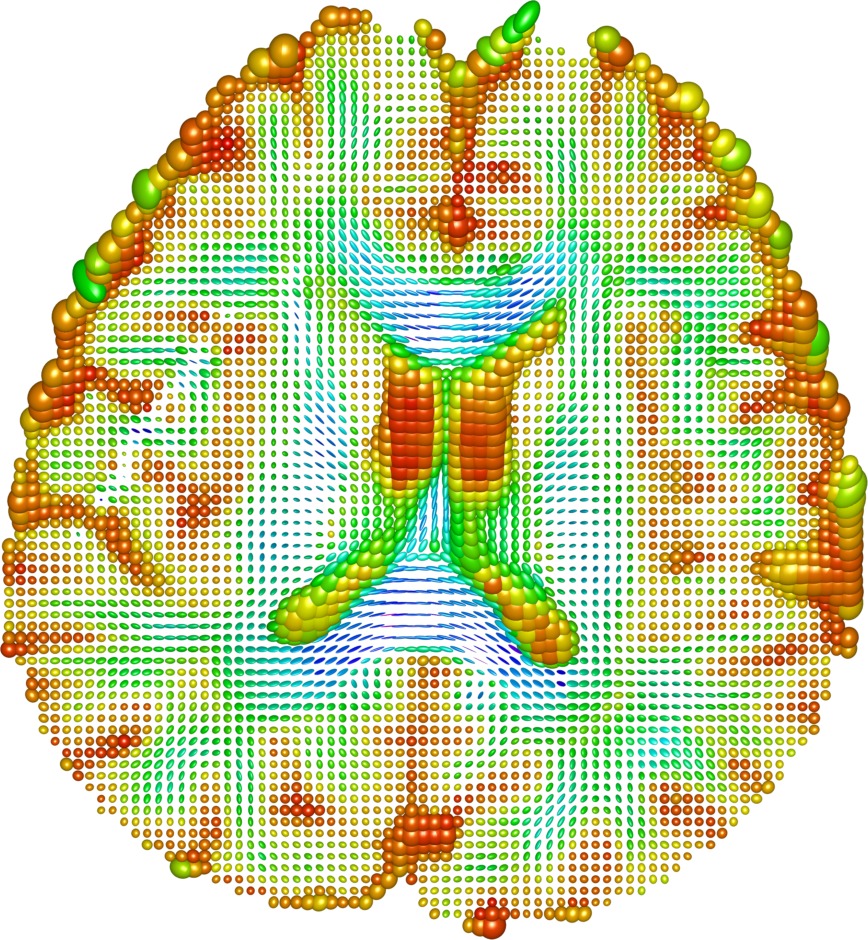}};
	\spy  on \nodeSpy in node [left] at \nodeWindow;
	\end{tikzpicture}
\end{subfigure}

	\caption{
		\emph{Left:} Diffusion tensor image of a human brain (axial slice). 
		The magnified image shows the corpus callosum. Note the missing (invalid) tensors at several voxels.  \emph{Right:} Result of S-TGV regularization with $r = \protect\input{\figfolder TGV_r.txt},$
		$s = \protect\input{\figfolder TGV_s.txt}.$
		The streamlines are smoothed whereas sharp transitions are preserved.
		Invalid voxels are reasonably inpainted.
	}
\label{fig:camino}
\end{figure}

\section{Conclusion}\label{sec:Conclusion}

In this work, we have  introduced and explored a notion of second-order total generalized variation (TGV) regularization for manifold-valued data. 
First, we have derived a variational model for total generalized variation for manifold-valued data.
For this purpose, we have used an axiomatic approach. 
We have first formalized reasonable fundamental properties of vector-valued TGV that should be conserved in the manifold setting. 
Then we have proposed two realizations which we have shown to fulfill the required axioms. The realization based on parallel transport  is rather natural -- although not straight forward -- from the point of view of differential geometry. The realization based on the Schild's ladder may be seen as an approximation of the parallel transport variant. It requires less differential geometric concepts and it is easier to realize numerically while yielding comparable numerical results as shown in the experimental section. Existence results for $\MTGV$-based denoising have been obtained for the proposed variants.
Next, we have derived an algorithm for the proposed model. To this end we
built on the well-established concept of a cyclic proximal algorithm. As main contribution, we have performed  the challenging task to derive all quantities necessary to compute the proximal mappings of the involved atoms. 
Finally, we have conducted a numerical study of the proposed scheme. 
The experiments revealed that M-TGV regularization reliably removes noise
while it preserves edges and smooth trends.
A quantitative comparison on images with groundtruth indicates that TGV
regularization improves upon first and second-order TV for manifold-valued data.

An interesting topic of future research are improved numerics
 such as earlier stopping criteria.

\section*{Acknowledgment}
Kristian Bredies and Martin Holler acknowledge support by the Austrian Science Fund (FWF) (Grant P 29192).
Martin Storath acknowledges support by the German Research Foundation DFG (Grant STO1126/2-1).
Andreas Weinmann acknowledges support by the German Research Foundation DFG (Grants WE5886/4-1, WE5886/3-1).

{\small
	\bibliography{lit_dat}
	\bibliographystyle{plain}
}

\appendix
\section{Appendix}
\subsection{Proofs for Section 2} \label{sec:proofs_section2}

\begin{proof}[Proof of Proposition \ref{prop:mp1_to_mp4_hold_general_bivariate}]
Given that $D$ satisfies the assumptions of Proposition \ref{prop:mp1_to_mp4_hold_general} and the particular form of $\dsym$ in the vector-space case, it is immediate that $\MTGV$ reduces to vector-space TGV, hence \textbf{(M-P1')} holds. 

 Now in case the minimum in \eqref{eq:tgv_manifold_general_multi} is attained at $y_{i,j}^1 = y_{i,j}^2  = u_{i,j}$ we get, 
since 
\[\dsym([u_{i,j},u_{i,j}],[u_{i,j},u_{i,j}],[u_{i,j-1},u_{i,j-1}],[u_{i-1,j},u_{i-1,j}]) = 0 \]
 for all $i,j$, that
 $\MTGV(u) = \alpha_1 \sum_{i,j} \Big( d(u_{i+1,j},u_{i,j})^p + d(u_{i,j+1},u_{i,j})^p \Big)^{1/p} = \alpha_1 \TV(u). $

On the other hand, in case the minimum is attained at $y_{i,j}^1 =  u_{i+1,j}$ and $y_{i,j}^2   = u_{i,j+1}$ we get, in the vector space situation, that
\begin{multline*}
 \MTGV(u)  =\alpha_0 \sum _{i,j}   \Big( |  u_{i+1,j}- 2u_{i,j} + u_{i-1,j}|^p+ |  u_{i,j+1}- 2u_{i,j} + u_{i,j-1}|^p \\
 + 2^{1-p} | (u_{i+1,j}-  u_{i,j}) - (u_{i+1,j-1} - u_{i,j-1})  + (u_{i,j+1}-  u_{i,j}) - (u_{i-1,j+1} - u_{i-1,j}) |^p \Big)^{1/p}
\end{multline*}
which coincides with $\alpha_0 \TV^2$ and $\TV^2$ as in Definition \ref{def:discreete_tv_tv2}.

Now suppose that $(u_{i,j})_{i,j}$ is locally geodesic. Choosing $y_{i,j}^1 = u_{i+1,j}$ and $y_{i,j}^2 = u_{i,j+1}$ for all $i,j$ we get, as in the proof of Theorem \ref{thm:d_s_properties}, that $\MTGV$ is bounded above by
\[ \alpha_0 \sum_{i,j}\dsym ([u_{i,j},u_{i+1,j}],[u_{i,j},u_{i,j+1}],[u_{i,j-1},u_{i+1,j-1}],[u_{i-1,j},u_{i-1,j+1}]) ,\]
which is zero by assumption. 

Now conversely, assume that $\MTGV(u) = 0$ and any two points in $u_{i,j}$, $u_{i',j'}$ with $\max\{|i-i'|,|j-j'|\} \leq 2$ are connected by unique geodesics. As in the proof of Proposition \ref{prop:mp1_to_mp4_hold_general} this implies that both $(u_{i,j})_i$ and $(u_{i,j})_j$ are locally geodesic as univariate signals. Also we get that
\[ 0 = \sum_{i,j} \dsym ([u_{i,j},u_{i+1,j}],[u_{i,j},u_{i,j+1}],[u_{i,j-1},u_{i+1,j-1}],[u_{i-1,j},u_{i-1,j+1}]), \]
which implies by assumption, that $(u_{i+k,j-k})_k$ is locally geodesic. Hence, by definition, $(u_{i,j})_{i,j}$ is locally geodesic as bivariate signal. Finally, by Lemma \ref{lem:locally_globally_geodesic_bivariate}, $u$ is geodesic.
\end{proof}

\begin{proof}[Proof of Theorem \ref{thm:d_s_properties_multi}]
It suffices to verify the assumptions of Proposition \ref{prop:mp1_to_mp4_hold_general_bivariate}. 
It is easy to see that, since $\ds([x,x],[u,u])= 0$, also $\dss([x,x],[x,x],[u,u],[u,u]) = 0$. 
Further, in the vector space setting, we get that $c^1 = \frac{\umc + \ycco}{2}$ and $c^2 = \frac{\ucm + \ycct}{2}$ as well as $r^1 = \umc + \ycco - \ucc$ and $r^2 =  \ucm + \ycct - \ucc$. Consequently,
\begin{align*}
 \ds([r^ 1,\ymct],[\ycmo,r^ 2]) & = | \ymct- \big( \ucm + \ycco - \ucc\big)   -\big( \ucm + \ycct - \ucc - \ycmo \big) | \\
 & = \big | \ycco - \ucc - (\ycmo - \ucm) + \ycct - \ucc - (\ymct - \umc)\big | 
 \end{align*}
and, using Proposition \ref{prop:mp1_to_mp4_hold_general_bivariate}, \textbf{(M-P1')} to \textbf{(M-P3')} follow.

Now suppose that $(u_{i,j})_{i,j}$ is a geodesic three-by-three signal. Hence, with the notation as in the definition of $\dss$, $c^1 = u_{i,j} \in [u_{i-1,j},u_{i+1,j}]_{\frac{1}{2}}$ and we can choose $r^1 = u_{i,j} \in [u_{i,j},u_{i,j}]_2$. Similar, $c^2 = u_{i,j} \in [u_{i,j-1},u_{i,j+1}]_{\frac{1}{2}}$ and we can again choose $r^2 = u_{i,j} \in [u_{i,j},u_{i,j}]_2$. Consequently, 
\begin{multline*}
\dss ([u_{i,j},u_{i+1,j}],[u_{i,j},u_{i,j+1}],[u_{i,j-1},u_{i+1,j-1}],[u_{i-1,j},u_{i-1,j+1}]) 
\\ \leq \ds([u_{i,j},u_{i-1,j+1}],[u_{i+1,j-1},u_{i,j}]).
\end{multline*} But, as shown in the proof of Theorem \ref{thm:d_s_properties}, the right-hand-side vanishes since $(u_{i-k,j+k})_k$ is geodesic as univariate signal.

Now conversely, assume that $u=(u_{i,j})_{i,j}$ is a three-by-three signal such that both $(u_{i,})_i$ and $(u_{i,j})_j$ are locally geodesic and that
\[ 0 = \dss ([u_{i,j},u_{i+1,j}],[u_{i,j},u_{i,j+1}],[u_{i,j-1},u_{i+1,j-1}],[u_{i-1,j},u_{i-1,j+1}]). \]
By the assumption on uniqueness of geodesics we get, again with the notation as in the definition of $\dss$, that $c^1 = u_{i,j}$ and hence $r^1 = u_{i,j}$ and also that $c^2 = u_{i,j}$ and hence $r^2 = u_{i,j}$. Consequently, $\dss(\ldots)$ reduces to $\ds([u_{i,j},u_{i-1,j+1}],[u_{i+1,j-1},u_{i,j}])$. But the latter being zero implies that $u_{i-1,j+1} \in [u_{i+1,j-1},u_{i,j}]_2$ and, consequently, $u$ is geodesic.

Hence also the remaining assumptions of Proposition \ref{prop:mp1_to_mp4_hold_general_bivariate} are satisfied and the assertion follows.
\end{proof}

The following well-known fact on the parallel transport in manifolds will be required in the proof of Theorem \ref{thm:pt_properties_multi}. For the sake of completeness, we provide a short proof.
\begin{lem}\label{lem:pt_of_log} Let $\um$, $\uc$ and $\up$ be three points in a manifold such that there exists a pairwise distance minimizing geodesic $\gamma:[0,1]\rightarrow \M $ such that $\gamma(0) = \um$, $\gamma (1/2) = \uc$, $\gamma (1) = \up$. Then $\log_\um^\gamma (\uc) = \pt^\gamma_\um (\log^\gamma_\uc (\up))$.
\proof For $t \in [0,1]$, set $w(t) = (1/2)\frac{d}{dt}\gamma(\frac{t}{2})$. Then $w$ satisfies $\frac{D}{dt}w(t) = 0$ on $[0,1]$, $w(0) = \log_\um^\gamma (\uc)$. By definition of the parallel transport, $\pt_{\uc}^ \gamma (\log_\um^\gamma (\uc)) = w(1) = (1/2)\frac{d}{dt}\gamma(1/2) =  \log^\gamma_\uc (\up).$ Applying $\pt^\gamma_{\um}$ on both sides, the result follows.
\end{lem}

\begin{proof}[Proof of Theorem \ref{thm:pt_properties_multi}]
It suffices to verify the assumptions of Proposition \ref{prop:mp1_to_mp4_hold_general_bivariate}.
Since the parallel transport is isometric, reduces to the identity if starting- and endpoint coincide, and since $\big|\log_x(x)\big|_x = 0$ for all $x \in \M$, it is easy to see that $\dpts([x,x],[x,x],[u,u],[u,u]) = 0$. 
In the vector space setting, with the notation as in the definition of $\dpts$, we get that $w^ 1 = \ycco - \ucc$ and $r^ 1 = \umc +\ycco - \ucc $ as well as $w^ 2 = \ycct - \ucc $ and $r^ 2 = \ucm +\ycct - \ucc $. Consequently, from the properties of $\dpt{}$ it follows that
\begin{align*}
\dpts ([\ucc,\ycco],[\ucc,\ycct],[\ucm,& \ycmo],[\umc ,\ymct]) = \\
&= \dpt{}([\umc +\ycco - \ucc,\ymct],[\ycmo,\ucm +\ycct - \ucc])  \\ 
&= \big |\ymct - \big( \umc +\ycco - \ucc) - \big( \ucm +\ycct - \ucc  - \ycmo  \big) \big|  \\
&= \big | \ycco - \ucc - (\ycmo - \ucm) + \ycct - \ucc - (\ymct - \umc)\big |
\end{align*}
and from Proposition \ref{prop:mp1_to_mp4_hold_general_bivariate} it follows that \textbf{(M-P1')} to \textbf{(M-P3')} holds.

Now let $u = (u_{i,j})_{i,j}$ be a three-by-three geodesic signal and choose $y^1_{i,j} = u_{i+1,j}$ and $y^2_{i,j} = u_{i,j+1}$ for each $(i,j)$. Denoting by $\gamma^ 1$ and $\gamma^ 2$ distance minimizing geodesics such that $\gamma^ 1(0) = u_{i-1,j}$, $\gamma^ 1(1/2) = u_{i,j}$, $\gamma^ 1(1) = u_{i+1,j}$ and $\gamma^ 2(0) = u_{i,j-1}$, $\gamma^ 2(1/2) = u_{i,j}$, $\gamma^ 2(1) = u_{i,j-1}$, respectively, Lemma \ref{lem:pt_of_log} implies that $\log_{u_{i-1,j}}^ {\gamma^ 1}(u_{i,j}) = \pt_{u_{i-1,j}}(\log^ {\gamma^ 1}_{u_{i,j}}(u_{i+1,j}))$  as well as $\log_{u_{i,j-1}}^ {\gamma^ 2}(u_{i,j}) = \pt_{u_{i,j-1}}(\log^ {\gamma^ 2}_{u_{i,j}}(u_{i,j+1}))$. Hence we can choose $r^ 1 = r^ 2 = u_{i,j}$ and
\begin{multline*}
\dpts ([u_{i,j},u_{i+1,j}],[u_{i,j},u_{i,j+1}],[u_{i,j-1},u_{i+1,j-1}],[u_{i-1,j},u_{i-1,j+1}]) 
\\ \leq \dpt{}([u_{i,j},u_{i-1,j+1}],[u_{i+1,j-1},u_{i,j}]) = 0,
\end{multline*}
where the last term is zero since $(u_{i-k,j+k})_k$ is locally geodesic as univariate signal.

Now conversely, assume that $u = (u_{i,j})_{i,j}$ is a three-by-three signal such that the geodesics connecting each pair of points are unique and such that $(u_{i,j})_i$ and $(u_{i,j})_j$ are locally geodesic.
By Lemma \ref{lem:pt_of_log} and uniqueness of geodesics we get $\log_{u_{i-1,j}}(u_{i,j}) = \pt_{u_{i-1,j}}(\log_{u_{i,j}}(u_{i+1,j}))$  as well as $\log_{u_{i,j-1}}(u_{i,j}) = \pt_{u_{i,j-1}}(\log_{u_{i,j}}(u_{i,j+1}))$. Again by uniqueness of geodesics, with the notation as in the definition of $\dpts$, we conclude that $r^ 1 = r^ 2 = u_{i,j}$ hence,
\begin{multline*}
0 = \dpts ([u_{i,j},u_{i+1,j}],[u_{i,j},u_{i,j+1}],[u_{i,j-1},u_{i+1,j-1}],[u_{i-1,j},u_{i-1,j+1}]) 
\\ = \dpt{}([u_{i,j},u_{i-1,j+1}],[u_{i+1,j-1},u_{i,j}]).
\end{multline*}
From the properties of $\dpt{}$ it hence follows that the points $u_{i+1,j-1}$, $u_{i,j}$, $u_{i-1,j+1}$ are on a joint, length minimizing geodesic at equal distance. Hence $u$ is geodesic. This ensures that all assumptions of Proposition \ref{prop:mp1_to_mp4_hold_general_bivariate} are fulfilled and hence the assertion follows.
\end{proof}

\subsection{Proofs for Section 3} \label{sec:proofs_section3}

\begin{proof}[Proof of Proposition \ref{prop:existence_basis_general}]
Take $(y^n)_n$ such that $y^n \in C(x) $ and $S(x) = \lim_{n} G(x,y^n)$. By assumption $(y^n)_n$ admits a subsequence $(y^{n_k})_k$ converging to some $y \in C(x)$. By lower semi-continuity of $G$ it follows that $S(x) = G(x,y)$.

For lower semi-continuity of $S$ now take $(x^n)_n$ in $\M^N$ converging to some $x \in \M^N$ for which, without loss of generality, we assume that $\liminf_n S(x^n) = \lim _n S(x^n)$.
Pick $y^n \in C(x^n)$ such that $S(x^n) = G(x^n,y^n)$. By assumption and uniqueness of limits we can obtain a subsequences $(x^{n_k})_k$ and $(y^{n_k})_k$ converging to $x$ and $y$, respectively, such that $y \in C(x)$. We conclude that
\[ S(x) \leq G(x,y) \leq \liminf_k G(x^{n_k},y^{n_k}) = \liminf_k S(x^{n_k}) = \liminf _n S(x^n) \] and the assertion follows.
\end{proof}

\begin{proof}[Proof of Lemma \ref{lem:stability_geodesics}]
The proof relies on the following fact: for any bounded subset $\mathcal{N}$ of $\M$, there is a constant $D$ such that, if the length of any geodesic $\psi:[r,s] \rightarrow \M$ with $\psi([r,s]) \subset \mathcal{N}$ is smaller than $D$, then $\psi $ is a unique distance minimizing geodesic between $\psi(r)$ and $\psi(s),$ and the Jacobi fields have no zero along this geodesic $\psi;$
see for instance \cite[Chapter 1.6/1.7]{cheeger1975comparison}.
In addition, there is a constant $L$ which is uniform for all such geodesics such that, with $f:[r,s] \times [0,1] \rightarrow \M$ a geodesic variation of $\psi$ with $f(\cdot,0) = \psi$, we have
\[ d(\psi(t),f(t,\tau)) \leq L \max\{ d(\psi(r),f(r,\tau)), d(\psi(s),f(s,\tau))\} \]
for all $\tau \in [0,1]$ and all $t \in [r,s]$.

In order to obtain uniform convergence of $(\gamma^n)_n$, our approach is now to use compactness arguments and subdivide the curves into small segments where these assertions hold.

At first we define the bounded set $\mathcal{N}$. To this aim, define $B= \{(p^n,q^n)\,|\, n\in \N\} \cup \{(p,q)\}$, $K = \sup \{ d(p',q') \,|\, (p',q') \in B\}< \infty$ and $\mathcal{N}$ to be the union of all images of shortest geodesics connecting two points $p'$ and $q'$ with $(p',q') \in B$. Then, for any $z \in \mathcal{N}$ there is a shortest geodesic $\psi:[r,s] \rightarrow \M$ such that $(\psi(r),\psi(s)) \in  B$ and $z \in \psi([r,s])$, and we get that $d(p,z) \leq d(p,\psi(r)) + d(\psi(r),z) \leq d(p,\psi(r))  + d(\psi(r),\psi(s)) \leq  \sup_{n \in \N}d(p,p^n) + K$, hence $\mathcal{N}$ is bounded.

Further, by construction, $\gamma^n([0,1]) \subset \mathcal{N}$ for any $n$. With this choice of $\mathcal{N}$, we now choose the constant $D$ as above and we choose $l \in \N$ large enough such that $l \geq 2\frac{d(p,q)}{D}$ and subdivide the interval $[0,1]$ into the $l+1$ points $0,\frac{1}{l},\frac{2}{l},\ldots,1$. Then, by compactness, we can find a subsequence $(\gamma^{n_k})_k$ such that $\gamma^{n_k} (t)$ converges for all $t \in \{ 0,\frac{1}{l},\ldots,1\}$. We set $z^{t}:= \lim_{k\rightarrow \infty} \gamma^{n_k}(t)$. Then
\[ d(z^{t},z^{t+1/l}) = \lim_{k \rightarrow \infty} d(\gamma^{n_k}(t),\gamma^{n_k}(t+1/l)) \leq  \frac{1}{l}\lim_{k \rightarrow \infty}d(p^{n_k},q^{n_k}) = \frac{d(p,q)}{l} < D. \]
Hence, for each $l$, there is a unique shortest geodesic connecting $z^{t}$ and $z^{t+1/l}$ and we define the curve $\gamma:[0,1] \rightarrow \M$ as the concatenation of these geodesics, parametrized proportionally. Then $\gamma(0) = p$, $\gamma(1) = q$ and the length of $\gamma$ is given as
\[ \sum_{t\in \{ 0,\frac{1}{l},\ldots,1-\frac{1}{l}\}} d(z^{t},z^{t+1/l}) \leq \frac{d(p,q)}{l}l = d(p,q). \]
Hence the length of $\gamma$ is less or equal to the distance of its two endpoints which implies that $\gamma$ is a shortest geodesic connecting $p$ and $q$; see for instance \cite{do1992riemannian}. Defining $\psi^t = \gamma|_{[z_{t},z_{t+1/l}]}$ we get that the image of $\psi^t$ is in $\mathcal{N}$ and its length is less than $D$. Defining the geodesic variation $f^{t,k}:[t,t+1/l] \times [0,1]$ as 
\[f^{t,k}(\eta,\tau) =  \big[ [z^{t},\gamma^{n_k}(t)]_\tau,[z^{t+1/l},\gamma^{n_k}(t+\frac{1}{l})]_\tau \big ] _\theta,
\]
with $\theta $ chosen such that $(1-\theta) t + \theta (t+\frac{1}{l}) = \eta$,
we get, for sufficiently large $k$ (such that the brackets are single-valued), that
\[ d(\gamma(\eta),\gamma^{n_k}(\eta)) = d(\psi^{t}(\eta),f^{t,k}(\eta,1)) \leq L \max \{ d(z^{t},\gamma^{n_k}(t)) , d(z^{t+1/l},\gamma^{n_k}(t+\frac{1}{l})) \} .\]
Hence $\gamma ^{n_k} \rightarrow \gamma$ uniformly on $[0,1]$. 

Now consider an arbitrary interval $[a,b]\supset [0,1]$. First we uniquely extend the geodesics $(\gamma^{n_k})_{k}$, $\gamma$ to this interval. Then, $\tilde{\mathcal{N}}:=\{\gamma^{n_k}([a,b])\, |\, k \in \N\} \cup \gamma([a,b])$ is a bounded set and we can again pick a constant $\tilde{D}$ such that any geodesic in $\tilde{\mathcal{N}}$ with length smaller than $\tilde{D}$ is a unique length minimizing geodesic between its start- and endpoint.
Now since geodesics have constant speed, we note that, for any $a',b'$ such that $[a',b'] \subset [a,b]$, the length of $\gamma^{n_k}|_{[a',b']}$ equals $|b'-a'|$ times the length of $\gamma^{n_k}|_{[0,1]}$ which in turn is equal to $|b'-a'|d(p^{n_k},q^{n_k})$. But the latter is uniformly bounded by convergence of $(p^{n_k})_k$, $(q^{n_k})_k$, and hence we can pick a uniform $\epsilon>0$ such that the length of $\gamma^{n_k}|_{[a',b']}$
(and of $\gamma|_{[a',b']}$) is less than $\tilde{D}$ whenever $|b'-a'| \leq 2\epsilon$.

Our approach is now to show that, whenever $\gamma^{n_k}$ converges uniformly to $\gamma$ on an interval $I$ with $[0,1] \subset I \subset [a,b]$, uniform convergence (up to subsequences) still holds if we extend the interval by $\epsilon$ on both sides (up to the boundary of $[a,b]$). By induction, the claimed assertion then follows.
We show this extension result by considering the extension from $[0,1]$ to $[0,1+\epsilon]$, the general case follows similarly and by symmetry.

For this purpose, we observe that for each $k$, $\gamma^{n_k}|_{[1-\epsilon,1+\epsilon]}$ is a unique length minimizing geodesic between $\gamma^{n_k}(1-\epsilon)$ and $\gamma^{n_k}(1+\epsilon)$, both of which, up to subsequences, converge to some limit points $q_{-\epsilon}$ and $q_{\epsilon}$, respectively. Hence, employing the first result of this lemma, we obtain that, again up to subsequences, $\gamma^{n_k}|_{[1-\epsilon,1+\epsilon]}$ converges uniformly to some $\psi:[1-\epsilon,1+\epsilon]\rightarrow \M$ with $\psi(1-\epsilon) = q_{-\epsilon}$, $\psi(1+\epsilon) = q_\epsilon$ being a length minimizing geodesics between those points. But by uniform convergence of $\gamma^{n_k}$ on $[0,1]$ and uniqueness of geodesics, we get that $\psi|_{[1-\epsilon,1]} = \gamma|_{[1-\epsilon,1]}$. Hence they coincide also on $[1-\epsilon,1+\epsilon]$ and the assertion follows.
\end{proof}

\begin{proof}[Proof of Lemma \ref{lem:shild_ex_lsc}]
For the assertion on $\ds$, apply Proposition \ref{prop:existence_basis_general} with $C:\M^4 \rightarrow \mathcal{P}(\M^2)$, $C(x,y,u,v):= \{ (c',y') \,|\, c' \in [x,v]_{\frac{1}{2}}, y' \in  [u,c']_2 \}$ and $G((x,y,u,v),(c,\tilde{y})) = d(\tilde{y},y)$. For the assertion on $\dss$, apply Proposition \ref{prop:existence_basis_general} with $C:\M^8 \rightarrow \mathcal{P}(\M^4)$,
$C(u^1,v^1,u^2,v^2,x^1,y^1,x^2,y^2):=$
$\big\{(\tilde{c}^1,\tilde{r}^1,\tilde{c}^2,\tilde{r}^2) \,|$ $(\tilde{c}^1,\tilde{r}^1,\tilde{c}^2,\tilde{r}^2) \in [x^2,v^1]_{\frac{1}{2}} \times[u^1,\tilde{c}^1]_2$ $\times [x^1,v^2]_{\frac{1}{2}} \times  [u^2,\tilde{c}^2]_2 \big \}$ 
and $G:\M^ 8 \times \M^ 4 \rightarrow \R$ according to
$G((u^1,v^1,u^2,v^2,x^1,y^1,x^2,y^2),(c^1,r^1,c^2,r^2)):=\ds([r^1,y^2],[y^1,r^2]).$ \qedhere 
\end{proof}

Our next goal is to provide the proof of Lemma \ref{lem:pt_ex_lsc}, stating existence and lower semi-continuity results for the parallel-transport-based distance-type functions. Regarding $\dpt{}$, we note that a proof
based on a statement similar to Proposition~\ref{prop:existence_basis_general} (as done for the Schild's variant) 
seems possible. However, since this would require a generalization of Proposition~\ref{prop:existence_basis_general} that involves metrics on the tangent bundle we chose to work out the proof for $\dpt{}$ directly in order to avoid introducing additional technicalities and notation. 
\begin{lem}\label{lem:PtVersionLscUniv}
  The parallel-transport-based distance-type functional $\dpt$
  is lower semi-continuous.  In particular, the minimum in
  \eqref{eq:definition_transport_distance} is attained for any
  $x,y,u,v \in \M.$
\end{lem}
\begin{proof}
We consider sequences $x^n \to x,$  $y^n \to y, $ $u^n \to u,$ and $v^n \to v,$ and show that
\begin{equation}\displaystyle \label{eq:toShowExPar}
	\dpt([x,y],[u,v]) \leq \liminf\nolimits_n \dpt([x^n,y^n],[u^n,v^n]).
\end{equation}
Here, by \eqref{eq:definition_transport_distance}	
\begin{equation}\displaystyle \label{eq:RecModDef4Existence}
\dpt([x,y],[u,v])=  \inf\nolimits_{z,w,\gamma} \big|   z -  \pt^\gamma_{\gamma(1),\gamma(0)}  w  \big| = \inf\nolimits_{z,w,\gamma} \big| \pt^\gamma_{\gamma(0),\gamma(1)}   z -  w  \big|, 
\end{equation}
where 
$ z \in \log_x(y),$ 
$ w \in \log_u(v),$
and $\gamma:[0,1]\rightarrow \M$ is a shortest geodesic connecting $\gamma(0)=x$ and $\gamma(1)=u.$
We note that the second equation in \eqref{eq:RecModDef4Existence} is true since the parallel transport is an isometry in a Riemannian manifold.
We start out choosing subsequences $x^{n_k} \to x,$  $y^{n_k} \to y,$ $u^{n_k} \to u,$ $v^{n_k} \to v,$
such that 
\begin{equation} \label{eq:ChoseFirstSubsequ}\displaystyle 
\lim\nolimits_{k \to \infty} \dpt([x^{n_k},y^{n_k}],[u^{n_k},v^{n_k}]) = \liminf\nolimits_n \dpt([x^n,y^n],[u^n,v^n]).
\end{equation}
Then, for each $k \in \mathbb N,$ we choose tangent vectors $z^{n_k},w^{n_k},$ and shortest geodesics $\gamma^{n_k}$ 
such that 
\begin{equation}\displaystyle  \label{eq:OurChoiceOfzwgamma}
     \left| \pt^{\gamma^{n_k}}_{\gamma^{n_k}(0),\gamma^{n_k}(1)}   z^{n_k} -  w^{n_k}  \right|  \leq \dpt([x^{n_k},y^{n_k}],[u^{n_k},v^{n_k}]) + \tfrac{1}{n_k}.
\end{equation}
Here, $z^{n_k}$ is sitting in the tangent space at the point $x^{n_k},$ 
$w^{n_k}$  is sitting in the tangent space at the point $u^{n_k},$ 
and $\gamma^{n_k}$ is one of the (potentially non-unique) shortest geodesics 
connecting the points $\gamma^{n_k}(0)= x^{n_k}$ and $\gamma^{n_k}(1)= u^{n_k}.$ 
The parallel transport is understood along $\gamma^{n_k}.$
By Lemma \ref{lem:stability_geodesics} there is a subsequence $\gamma^{n_l}$ of $\gamma^{n_k}$
(for convenience, we suppress iterated indices and write  $\gamma^{n_l}$ instead of $\gamma^{n_{k_l}}$ in the following)
such that $\gamma^{n_l} \to \gamma$ uniformly on $[0,1],$
for some shortest geodesic $\gamma:[0,1]\rightarrow \M$ connecting $\gamma(0)=x$ and $\gamma(1)=u.$

Since $x^{n_l}$ converges to $x,$ the geodesic connecting $x^{n_l}$ and $x$ is unique for sufficiently large $l$.
The same is true for $u^{n_l}$ and $u.$ So we may identify the tangent vectors  $z^{n_l}$ sitting at  $x^{n_l},$
and the tangent vectors $w^{n_l}$ sitting at  $u^{n_l},$
with their parallel transported versions 
\begin{equation}\displaystyle \label{eq:z_bar_pt_z}
\bar{z}^{n_l} = \pt_{x_{n_l},x} z_{n_l}, \qquad 
\bar{w}^{n_l} = \pt_{u_{n_l},u} w_{n_l},
\end{equation}  
along the corresponding unique geodesic. 
Note that the $\bar{z}^{n_l}$ are sitting in the common point $x,$  
and that the $\bar{w}^{n_l}$ are sitting in the common point $u.$
The sequences  $\bar{z}^{n_l}$  is bounded since parallel transport is an isometry and 
$ z^{n_l} \in \log_{x^{n_l}}(y^{n_l}),$ 
where both $x^{n_l}$ and  $y^{n_l}$ converge.  
The analogous statements hold for $\bar{w}^{n_l}$ with the same argument.
So, by going to subsequences $\bar{z}^{n_r}$ of $\bar{z}^{n_l},$
and $\bar{w}^{n_r}$ of $\bar{w}^{n_l},$ we get that
\begin{equation}\displaystyle \label{eq:BarsConverge}
\bar{z}^{n_r} \to z, \qquad 
\bar{w}^{n_r} \to w,
\end{equation}  
for a tangent vector $z$ sitting in $x$ and a tangent vector $w$ sitting in $u.$
We have that the  limit $ z \in \log_x(y),$ 
and that $ w \in \log_u(v),$ since the exponential map from $T\M \to \M $ is differentiable in a Riemannian manifold
which implies $\exp_x z = \lim_{r \to \infty} \exp_{x^{n_r}} z^{n_r} = \lim_{r \to \infty} y^{n_r} = y, $
and $\exp_u w = v.$

We are now prepared to estimate $| \pt^\gamma_{\gamma(0),\gamma(1)}   z -  w |$ which allows us to bound
$\dpt([x,y],[u,v])$ from above; see \eqref{eq:RecModDef4Existence}.
Since the parallel transport is an isometry, we have 
\begin{equation}\displaystyle \label{eq:MoveToUnr}
\big| \pt^ \gamma_{\gamma(0),\gamma(1)}   z -  w  \big| = \big| \pt_{u,u^{n_r}} \pt^\gamma_{\gamma(0),\gamma(1)}   z -  \pt_{u,u^{n_r}} w  \big|.
\end{equation}
We further have that
\begin{equation}\displaystyle \label{eq:ParCommu}
\varepsilon_r := \big| \pt_{u,u_{n_r}} \pt^\gamma_{\gamma(0),\gamma(1)} z - \pt^{\gamma^{n_r}}_{\gamma^{n_r}(0),\gamma_{n_r}(1)} \pt_{x,x^{n_r}} z \big| 
\to 0 \qquad \text{ as } \quad r \to \infty 
\end{equation}
which is a consequence of the parallel transport along nearby geodesics being differentiably dependent on the variation of the geodesics.
Using \eqref{eq:ParCommu} together with \eqref{eq:MoveToUnr}, we have 
\begin{align} \label{eq:EstExPar1}
\big| \pt^\gamma_{\gamma(0),\gamma(1)}   z -  w  \big| 
& \leq \big| \pt^{\gamma^{n_r}}_{\gamma^{n_r}(0),\gamma^{n_r}(1)} \pt_{x,x^{n_r}}  z -  \pt_{u,u^{n_r}} w  \big| + \varepsilon_r \\
& \leq \big| \pt^{\gamma^{n_r}}_{\gamma^{n_r}(0),\gamma^{n_r}(1)} \pt_{x,x^{n_r}}  \bar{z}^{n_r}  -  \pt_{u,u^{n_r}} \bar{w}_{n_r}  \big| 
+ |\bar{z}^{n_r} - z | + |\bar{w}^{n_r} - w |+ \varepsilon_r. \notag 
\end{align}
The second inequality is a consequence of the triangle inequality together with the fact that parallel transport is an isometry.
We take the limit w.r.t.\ $r$ on the right hand side of \eqref{eq:EstExPar1}: by \eqref{eq:BarsConverge},
we have that $\bar{z}^{n_r} \to z,$ and that $\bar{w}^{n_r} \to w $ and by \eqref{eq:ParCommu} that $\varepsilon_r \to 0$
which implies 
\begin{align}\displaystyle \label{eq:EstExPar2}
\big| \pt^\gamma_{\gamma(0),\gamma(1)}   z -  w  \big| 
& \leq \lim_{r \to \infty} \big| \pt^{\gamma^{n_r}}_{\gamma^{n_r}(0),\gamma^{n_r}(1)} \pt_{x,x^{n_r}}  \bar{z}^{n_r}  -  \pt_{u,u^{n_r}} \bar{w} ^{n_r} \big|\\
& \leq \lim_{r \to \infty} \left( \dpt([x^{n_r},y^{n_r}],[u^{n_r},v^{n_r}]) + \tfrac{1}{n_r} \right) 
 =  \liminf_n \dpt([x^n,y^n],[u^n,v^n]).   \notag     
\end{align}
The second inequality in \eqref{eq:EstExPar2} is a consequence of our choice of $z^{n_k},w^{n_k},$ and $\gamma^{n_k}$ in \eqref{eq:OurChoiceOfzwgamma} and the definition of $\bar{z}^{n_l},\bar{w}^{n_l}$ as in \eqref{eq:z_bar_pt_z}, and the equality in the last line follows by our choice of subsequences in \eqref{eq:ChoseFirstSubsequ}. Then passing to the infimum according to \eqref{eq:RecModDef4Existence}
shows \eqref{eq:toShowExPar} which shows the first assertion.

To see that the infimum in \eqref{eq:definition_transport_distance} is attained for any $x,y,u,v \in \M$,
we choose the constant sequences $x^n:= x,$  $y^n := y, $ $u^n := u,$ and $v^n := v,$ and apply the previous result to these sequences. This shows the second assertion and completes the proof.
\end{proof}
Having shown Lemma~\ref{lem:PtVersionLscUniv} we can now employ  
Proposition~\ref{prop:existence_basis_general} to show existence and lower semi-continuity results for $\dpts$.
As a preparation we need the following lemma.

\begin{lem} \label{lem:PtVersionLscBivAux}
Let $u^n \to u,$  $\tilde u^n \to \tilde u,$ and $y^n \to y$ in the complete manifold $\mathcal M.$ 
We consider a sequence $(r^n)_n$ with
\begin{equation}\label{eq:TheFlesh4BivPt1}
   r^n \in  \exp (\pt_{\tilde u^n}(w^n)) \quad  \text{ with } \quad  w^n \in \log_{u^n}y^n.	
\end{equation}
Then there is a subsequence $(r^{n_k})_k$ which converges and a limit $r = \lim_{k \to \infty} r^{n_k}$ such that $r$ fulfills
\begin{equation}\label{eq:TheFlesh4BivPt2}
r \in  \exp (\pt_{\tilde u}(w)) \quad  \text{ for some } \quad  w \in \log_{u}y.	
\end{equation}
\end{lem}

\begin{proof}
	The present proof essentially employs the techniques already used 
	in the proof of Lemma~\ref{lem:PtVersionLscUniv}. For this reason we keep the following arguments rather short.
	The sequence $\bar w^n := \pt_{u}(w^n)$ is bounded since parallel transport is an isometry. 
	So, by going to a subsequence $\bar{w}^{n_l}$ of $\bar{w}^{n},$
    we get that $\bar{w}^{n_l} \to w$ for a tangent vector $w$ sitting in $u.$
	We have that the  limit  $ w \in \log_u(y),$ since the exponential map $T\M \to \M $ is differentiable 
	which implies $\exp_u w = \lim_{l \to \infty} \exp_{u^{n_l}} w^{n_l} = \lim_{l \to \infty} y^{n_l} = y.$
	For each $l$ we choose a distance-minimizing geodesic $\gamma^{n_l}$ connecting 
	$\gamma^{n_l}(0)$ $= u^{n_l}$ with $\gamma^{n_l}(1)$ $= \tilde u^{n_l}.$   
	Then we use Lemma~\ref{lem:stability_geodesics} to choose a subsequence of geodesics $\gamma^{n_k}$ of $\gamma^{n_l}$ 
	(suppressing iterated subindices) which uniformly converges to a geodesic $\gamma:[0,1]\rightarrow \M$ connecting $\gamma(0) = u$ with 
	$\gamma(1)=\tilde u$.
    Then, since parallel transport along nearby geodesics is differentiably dependent on the variation of the geodesics and by an argument similar to the one used for the convergence in Equation \eqref{eq:ParCommu}, 
	$\pt_{\tilde u^{n_k} ,\tilde u} \pt^{\gamma^{n_k}}_{u^{n_k},\tilde u^{n_k}} w^{n_k}$ $\to \pt^ \gamma_{u,\tilde u} w$
	 as $k \to \infty.$ 	
	Then, by the continuity of the exponential map,  
	$r = \lim_{k \to \infty} r^{n_k}$ $= \lim_{k \to \infty} \exp (\pt^{\gamma^{n_k}}_{u^{n_k},\tilde u^{n_k}}(w^{n_k})) =  \exp (\pt^ \gamma_{u,\tilde u}(w)),$
	i.e.,  $r^{n_k}$ converges and \eqref{eq:TheFlesh4BivPt2} holds true which was the assertion to show.
\end{proof}
Finally, the proof of Lemma \ref{lem:pt_ex_lsc} follows as consequence.

\begin{proof}[Proof of Lemma \ref{lem:pt_ex_lsc}]
For $\dpt{}$, this is the assertion of Lemma \ref{lem:PtVersionLscUniv}. For $\dpts$, we apply Proposition \ref{prop:existence_basis_general} with $C:\M^8\rightarrow \mathcal{P}(\M^2)$,
$C(u^1,v^1,u^2,v^2,x^1,y^1,x^2,y^2):=$ 
$\{\exp(\pt_{x^2}w) \,|\, w \in \log_{u^1}(v^1)\}$ $\times \{\exp(\pt_{x^1}w) \, |\, w \in \log_{u^2}(v^2)\} ,
$
and $G:\M^ 8 \times \M^ 2\rightarrow \R$ according to
$G((u^1,v^1,u^2,v^2,x^1,$ $y^1,x^2,y^2),(r^1,r^2)):=\dpt{}([r^1,y^2],[y^1,r^2]) .$
The Lemmata  \ref{lem:PtVersionLscUniv} and \ref{lem:PtVersionLscBivAux} ensure that $C$ and $G$ satisfy the assumption of Proposition \ref{prop:existence_basis_general}. 
\end{proof}

\subsection{Proofs for Section 4}\label{sec:proofs_section4}

\begin{proof}[Proof of Lemma~\ref{lem:GradDSwrtui}]
	Let us consider the Schild's ladder mapping 
	$ 
	u_{i-1} \mapsto S(u_{i-1},y_{i-1},u_{i}) = [u_{i-1},[u_{i},y_{i-1}]_{\frac{1}{2}}]_2
	$
	as a function of $u_{i-1}$.
	Since the points $y_{i-1},u_{i}$ are fixed, so is their midpoint $m=[u_{i},y_{i-1}]_{\frac{1}{2}}.$
	Now $S(u_{i-1},y_{i-1},u_{i})$ is obtained by evaluating the geodesic $\gamma:[0,1]\rightarrow \M$ connecting
	$u_{i-1}=\gamma(0)$ and $m=\gamma(1)$ at time $2,$ thus considering $\gamma(2).$
	Hence, the differential $T_1$ of $S$ w.r.t. $u_{i-1}$  is related to the Jacobi fields along $\gamma$ 
	as follows: consider those Jacobi fields $\mathcal J$ along $\gamma$ 
	for which $J(1)=0$ which means that they belong to geodesic variations leaving $m$ fixed. 
	Then the adjoint of the mapping 
	\begin{equation}\label{eq:DiffInTermJFDsui}
	J(0) \mapsto J(2),\qquad J \in \mathcal J
	\end{equation}
	equals $T_1$.
	
	If the manifold is a Riemannian symmetric space, then the mapping 
	$ 
	u_{i-1} \mapsto S(u_{i-1},y_{i-1},u_{i}) = [u_{i-1},[u_{i},y_{i-1}]_{\frac{1}{2}}]_2
	$
	is a geodesic reflection, see, e.g., \cite{eschenburg1997lecture}.
	We consider an orthonormal basis $(v_n)_n$ of eigenvectors of the self-adjoint Jacobi operator 
	$J \mapsto R(\frac{\gamma'(1)}{ |\gamma'(1) |} ,J) \frac{\gamma'(1)}{|\gamma'(1) |}$  with corresponding eigenvalues associated $\lambda_n,$ 
	and $v_1$ tangent to $\gamma'(1).$
	Then the mapping
	$J(0) \mapsto J(2)$, $ J \in \mathcal J,$
	in \eqref{eq:DiffInTermJFDsui} can be written as 
	\begin{equation}\displaystyle \label{eq:MapWIsIdentity}
	\sum\nolimits_n \alpha_n f(\lambda_n) \pt_{m,u_{i-1}}  v_n \mapsto 
	- \sum\nolimits_n \alpha_n f(\lambda_n) \pt_{m,S(u_{i-1},y_{i-1},u_{i})} v_n,
	\end{equation}
	where the $\alpha_n$ are the coefficients of the corresponding basis representation and 
	the function $f,$ depending on the sign of $\lambda_n$, is given 
	by $f(\lambda_n) = d,$  if $\lambda_n = 0,$
	by $f(\lambda_n) = (\sqrt{\lambda_n} )^{-1} \sin (\sqrt{\lambda_n} d) ,$	if $\lambda_n > 0$ and 
	by $f(\lambda_n) = (\sqrt{-\lambda_n} )^{-1} \sinh (\sqrt{-\lambda_n} d) ,$	if $\lambda_n < 0,$
	where $d$ is the distance between $m$ and $u_{i-1}$
	(which equals the distance between $m$ and $S(u_{i-1},y_{i-1},u_{i}).$)
	This results from the fact that $f(\lambda_n) = x(0),$ the evaluation at $0$ of the solution $x(t)$ of the scalar initial value problem $x'' = - \lambda_n d^2 x, \quad x(1) = 0, x'(1)= -d.$
	We observe that \eqref{eq:MapWIsIdentity} states that the mapping
	$J(0) \mapsto J(2)$, $ J \in \mathcal J,$ equals the identity multiplied with $-1$
	(up to parallel transport).
	(We note that this can also be concluded from the derivations in  \cite{eschenburg1997lecture} near Theorem 1.)
	Hence the adjoint of the mapping $J(0)\mapsto J(2)$ coincides with the parallel transport multiplied by $-1$ as in \eqref{eq:NablaUISym}, in particular, it is an isometry, and the proof is complete.
\end{proof}

\begin{proof}[Proof of Lemma~\ref{lem:GradDSwrtyi}]
	We observe that the concatenation of the mappings $m \mapsto  [u_{i-1},m]_2$ and 
	$y_{i-1} \mapsto  [u_{i},y_{i-1}]_{\frac{1}{2}}$ equals the mapping $y_{i-1} \mapsto [u_{i-1},[u_{i},y_{i-1}]_{\frac{1}{2}}]_2$.
	Hence, we conclude from the discussion leading to \eqref{eq:InitDiscussion2} 
	in connection with the chain rule the validity of
	\eqref{eq:diffDswrtyi}.
	
	It remains to express $T_3,T_4$ in terms of Jacobi fields.
	Concerning $T_3$ we consider the  geodesic $\gamma$ connecting
	$u_{i-1}=\gamma(0)$ and $m=\gamma(1).$ 
	$T_3$ is related to the Jacobi fields along $\gamma$ 
	as follows: consider those Jacobi fields $\mathcal J$ along $\gamma$ 
	for which $J(0)=0$ which means that they belong to geodesic variations leaving $u_{i-1}$ fixed. 
	Then, the adjoint of the mapping 
	\begin{equation}\displaystyle \label{eq:DiffInTermJFDsyi1}
	J(1) \mapsto J(2),\qquad J \in \mathcal J
	\end{equation}
	equals $T_3.$ 
	Concerning $T_4,$ we consider the  geodesic $\xi$ connecting
	$y_{i-1}=\xi(0)$ and $u_{i}=\xi(1).$ 
	$T_4$ is related to the Jacobi fields $\mathcal J'$ along $\xi$ 
	for which $J(1)=0$ by 
	the mapping 
	\begin{equation}\displaystyle \label{eq:DiffInTermJFDsyi2}
	J(0) \mapsto J(\tfrac{1}{2}),\qquad J \in \mathcal J'.
	\end{equation}	
	
\end{proof}

\begin{proof}[Proof of Lemma~\ref{lem:T3T4explicit}]
	We consider the assertion of Lemma~\ref{lem:T3T4explicit} for the mapping $T_3.$ 
	By the proof of Lemma~\ref{lem:GradDSwrtyi}, we have to
	determine the adjoint of the mapping given by \eqref{eq:DiffInTermJFDsyi1} more explicitely.
	Since the covariant derivative of the curvature tensor equals zero in a symmetric space,
	the differential equation for the Jacobi fields $J$ in terms of the scalar coefficient $x$ of the vector field
	$\pt_{m,[m,u_{i-1}]_t}v_n$ obtained by parallel transport of the eigenvector $v_n$ along $\gamma$
	reads  
	$x'' = - \lambda_n d^2 x, \quad x(0) = 0.$
	Solving this scalar ODE shows \eqref{eq:JacobiSymZentrStr}.
	Then, $f(\lambda_n)$ corresponds to the value $x(2)$ of the solution of the scalar ODE of the previous line with 
	additional boundary condition $x'(1)=1$.
	Solving this scalar ODE yields \eqref{eq:f4T3}.
We have already derived the corresponding statement for $T_4$ in \cite{bavcak2016second}.  
\end{proof}

Our next goal is to show Lemma~\ref{FSymAndImplications4Der}.
To this end, we introduce the mapping  $F_t$ which is a slight extension of  $\dpt{}$ for different parameters as follows.
For $t \in [0,1],$ we consider the mapping $F_t: \M \times \M \times \M \times \M \to [0,\infty),$ given by 
\begin{align}\displaystyle \label{eq:DefFt}
F_t(u_i,u_{i-1},y_i,y_{i-1}) = \big | \pt_{0,t} \log_{u_i}y_i - \pt_{1,t} \ \log_{u_{i-1}}y_{i-1} \big |,	
\end{align}
where, for the definition of $F_t$ and in the following Lemma, we use the shorthand notation $\pt_{s,t}$ to denote the parallel transport from $[u_i,u_{i-1}]_s$ to $[u_i,u_{i-1}]_t$ and the norm on the right hand side is the one introduced by the Riemannian scalar product in the point $[u_i,u_{i-1}]_t$.
Note that $F_0(u_i,u_{i-1},y_i,y_{i-1}) = \dpt([u_i,y_i],[u_{i-1},y_{i-1}])$, so in order to obtain the derivative of $\dpt{}$ it suffices to differentiate $F_0$ w.r.t. its four arguments. The following lemma shows that we can also consider $F_t$ instead.

\begin{lem} \label{lem:secDiffIndependentOft}
	The function $F_t$ given by \eqref{eq:DefFt} is independent of $t.$ 
\end{lem}

\begin{proof}
	The Riemannian scalar product is invariant under parallel transport,
	i.e., for any tangent vectors $v,v'$ sitting in the arbitrary point $x,$ and the parallel transport along any curve $\gamma$ with any points $y,z$ sitting on that curve,
	we have 
	$	
	\langle \pt_{x,y}v , \pt_{x,y} v' \rangle_y  = \langle \pt_{x,z}v , \pt_{x,z} v' \rangle_z.
	$
	Hence, for all $s,t \in [0,1],$
	\begin{align}\displaystyle \label{eq:EqParL41}
	\big | \pt_{0,s} \log_{u_i}y_i - \pt_{1,s} \ \log_{u_{i-1}}y_{i-1} \big |^2 
	&=		\big | \pt_{t,s}  \pt_{0,t} \log_{u_i}y_i - \pt_{t,s} \pt_{1,t} \ \log_{u_{i-1}}y_{i-1} \big |^2 \notag	\\
	&=	     \big | \pt_{0,t} \log_{u_i}y_i - \pt_{1,t} \ \log_{u_{i-1}}y_{i-1} \big |^2.
	\end{align}
	Since the first expression in \eqref{eq:EqParL41} equals the square of $F_s,$
	and the last expression equals the square of $F_t,$ this shows that  $F_s=F_t$ for all $s,t \in [0,1].$	 
\end{proof}

As pointed out below Lemma~\ref{FSymAndImplications4Der}, Lemma~\ref{FSymAndImplications4Der} is obtained by specifying $F=F_0$ in the following lemma.

\begin{lem} \label{FtSymAndImplications4Der}
	The function $F_t$ given by \eqref{eq:DefFt} is symmetric with respect to interchanging $(u_i,y_i)$ with $(u_{i-1},y_{i-1})$ i.e.,
	$
	F_t(u_i,u_{i-1},y_i,y_{i-1})  = F_t(u_{i-1}, u_i,y_{i-1},  y_i). 
	$
	In particular, for points with $F_t \neq 0,$ the gradient of the function $F_t$ w.r.t.\ the third variable $y_{i},$
	is given by 
	$
	\nabla_{y_{i}} F_t(u_i,u_{i-1},y_i,$ $y_{i-1})  = 
	\nabla_{y_{i-1}} F_t(u_{i-1}, u_i,y_{i-1},  y_i),
	$
	where $\nabla_{y_{i-1}} F_t$ denotes the derivative of $F_t$ w.r.t. the fourth argument.
	Further, again for points with $F_t \neq 0$, the gradient of the function $F_t$ w.r.t.\ the first component variable $u_{i},$
	is given by 
	$
	\nabla_{u_{i}} F_t(u_i,u_{i-1},y_i,y_{i-1})  = 
	\nabla_{u_{i-1}} F_t(u_{i-1}, u_i,y_{i-1},  y_i),
	$
	where $\nabla_{u_{i-1}} F_t$ denotes the derivative of $F_t$ w.r.t. the second argument.   	 
\end{lem}

\begin{proof}
	By the definition of $F_t$ in \eqref{eq:DefFt}, we have 
	$F_t(u_i,u_{i-1},y_i,y_{i-1})  = F_{1-t}(u_{i-1}, u_i,y_{i-1},  y_i).$	
	By Lemma~\ref{lem:secDiffIndependentOft}, $F_t$ is independent of $t.$
	Together, this implies the identity 
	$F_t(u_i,u_{i-1},y_i,y_{i-1})  = F_t(u_{i-1}, u_i,y_{i-1},  y_i)$
	which is the first assertion of the lemma.
	The following two assertions 
	on the gradients are then an immediate consequence of 
	this symmetry property.
\end{proof}

\begin{proof} [Proof of Lemma~\ref{lem:GradF1wrtyip1}]
	We note that $F$ agrees with the function $F_0$ with $F_t$ as in Equation \eqref{eq:DefFt} by definition and, consequently, also with $F_1$ by Lemma \ref{lem:secDiffIndependentOft}.
	For fixed $u_i$, $u_{i-1}$ and $y_i$, we decompose the mapping $y_{i-1} \mapsto F_1(u_i,u_{i-1},y_i,y_{i-1})$ into the mappings $G,H$, i.e., $F_1 = H \circ G,$ where 
	\begin{equation}\displaystyle \label{eq:DefG}
	G(y_{i-1}) = \log_{u_{i-1}} y_{i-1} 
	\end{equation}  
	locally maps the manifold $\M$ to the tangent space at $u_{i-1},$  and 
	\begin{equation}
	H(w) = \left| w - z\right|, \qquad \text{ where } z = \pt_{u_{i},u_{i-1}} \log_{u_i} y_i,
	\end{equation} 
	maps from the tangent space at $u_{i-1}$ to the positive real numbers.
	The differential of $H$ (as a map defined on the tangent bundle) is given by 
	$d_\xi H (\eta) = \langle \frac{\xi - z }{|\xi - z |} ,  \eta  \rangle, $ 
	for $\xi \neq z.$ Here $\xi$ is a point in the tangent space at $u_{i-1},$
	and $\eta$ is the direction in the tangent space with respect to which the differentiation is performed.
	Hence, the gradient of $H$ at $\xi$ equals $\frac{\xi - z }{ |\xi - z |}.$
	So the gradient of $H$ at the point $\xi = \log_{u_{i-1}} y_{i-1}$ equals 
	\begin{equation}
	\nabla H (\log_{u_{i-1}} y_{i-1})= {\left(\log_{u_{i-1}} y_{i-1}-z \right)}/{\left| \log_{u_{i-1}} y_{i-1} - z \right|}. 
	\end{equation} 
	In order to get the gradient of $F_1,$ we have to multiply $\nabla H (\log_{u_{i-1}} y_{i-1})$
	with the adjoint of the differential of $G.$
	The mapping $G$ is related to the Jacobi fields along the geodesic $\gamma:[0,1]\rightarrow \M$ connecting the points $\gamma(0)= u_{i-1}$ and $\gamma(1) = y_{i-1}$
	as follows. Consider the collection $\mathcal J$ of Jacobi fields $J$ along $\gamma$ 
	for which $J(0)=0$, which means that they belong to geodesic variations leaving $u_{i-1}$ fixed. 
	Then the mapping 
	\begin{equation}\label{eq:JFforLog}
	J(1) \mapsto J'(0),\quad  J \in  \mathcal J.   
	\end{equation} 
	equals the differential of $G.$ 
\end{proof}

\begin{proof}[Proof of Lemma~\ref{lem:Texplicit}]
	We use the basis $\{v_n\}$ to express the operator $T$ of Lemma~\ref{lem:GradF1wrtyip1} 
	(which is given as the adjoint of the derivative of $G$ defined by \eqref{eq:DefG} evaluated at $y_{i-1}$). 
	We use the expression 
	\eqref{eq:JFforLog} for the derivative of $G.$
	Since the covariant derivative of the curvature tensor equals zero in a symmetric space,
	the differential equation for the Jacobi fields $J$ in terms of the scalar coefficient $x$ of the vector field
	$\pt_{u_{i-1},[u_{i-1},y_{i-1}]_t}v_n$ obtained by parallel transport of the eigenvector $v_n$ along $\gamma$
	reads  
    $$x'' = - \lambda_n d^2 x, \quad x(0) = 0.$$
    This shows \eqref{eq:DiffOfLogInSym}.
	Further, $f(\lambda_n)$ corresponds to the value $x'(0)$ of the solution of the scalar ODE of the previous line with 
	additional boundary condition $x(1) = 1$.
    Solving this scalar ODE yields \eqref{eq:functionalCalculusDiffOfLogInSym}.	
\end{proof}

\begin{proof} [Proof of Lemma~\ref{lem:GradF1wrtuip1}]
	We again note that $F$ agrees with the function $F_1$ introduced in Equation~\eqref{eq:DefFt}.
	So our task is to determine the gradient of the function $F_1$ for points with $F_1 \neq 0$ given by \eqref{eq:DefFt} w.r.t.\ the variable $u_{i-1}.$
	We analyze the structure of $F_1$ as a function of $u_{i-1}.$
	To this end, we consider the two vector fields 
	$
	L: u_{i-1} \mapsto \log_{u_{i-1}}y_{i-1} 
	$
	and 
	$
	B: u_{i-1} \mapsto \pt_{u_{i},u_{i-1}} z,$ where 
	$z = \log_{u_{i}}y_{i}$ introduced above.
	Remember that $\pt_{u_{i},u_{i-1}}$ denotes the parallel transport from the point $u_{i}$ to the (varying) point $u_{i-1}$
	along a shortest geodesic connecting these points. We note that the parallel transport here depends on the varying end point $u_{i-1}$. Further note that $z = \log_{u_{i}}y_{i}$ does not depend on $u_{i-1}$ and is therefore fixed.
	Using this notation we may write \eqref{eq:DefFt} as 
	\begin{align} \label{eq:F1asFuncOfuip1Rewriten}
	F_1 (u_i,\cdot,y_i,y_{i-1}): u_{i-1} \mapsto \big |L(u_{i-1}) - B (u_{i-1}) \big | .
	\end{align}
	In order to differentiate \eqref{eq:F1asFuncOfuip1Rewriten} w.r.t.\ $u_{i-1}$
	we need some more preparation.
	Recall that the Levi-Civita connection is metric. 
	Hence, for any two vector fields $V_t,W_t$ along a geodesic $\gamma,$
	\begin{equation}\label{eqref:LeviChevitaIsMetric}
	\tfrac{d}{dt}\langle V_t, W_t \rangle = \langle \tfrac{D}{dt} V_t, W_t  \rangle + \langle  V_t, \tfrac{D}{dt} W_t \rangle.
	\end{equation}
	Recall that we use the symbol $\frac{D}{dt}$ to denote the covariant derivative of the corresponding vector field along the curve $\gamma.$
	Thus, for any two vector fields $V_t,W_t$ with $V_t \neq W_t,$ we have 
	\begin{align}\tfrac{d}{dt} \big | V_t - W_t \big | &= \tfrac{1}{ 2 \big | V_t - W_t \big |} \cdot 
	2 \langle  V_t- W_t,  \tfrac{D}{dt} V - \tfrac{D}{dt} W  \rangle  \notag\\
	& =  \langle   \tfrac{V_t- W_t}{\big | V_t - W_t \big |} ,  \tfrac{D}{dt} V - \tfrac{D}{dt} W  \rangle.
	\label{eqref:DifferentiateDifferenceOfVectorFields}
	\end{align}

	Since we have chosen the $v_n$ to be an orthogonal basis of the tangent space at $u_{i-1},$
	the coordinate representation of the gradient in this basis is given as the directional derivatives
	w.r.t. the basis vectors. The curves $t \mapsto \exp_{u_{i-1}}tv_n$ precisely yield these tangent vectors. This explains \eqref{eq:ExpandGradinONB}, \eqref{eqref:defAlphaNVecF},
	as well as the first identity in \eqref{eqref:defAlphaN}.
	The second identity in \eqref{eqref:defAlphaN} is a consequence of \eqref{eqref:DifferentiateDifferenceOfVectorFields}.		 
\end{proof}

\begin{proof} [Proof of Lemma~\ref{lem:GradLn}]	
	We notice that the proof of this lemma uses well-known facts on the connection of Jacobi fields and the exponential map (see for instance the books \cite{spivak1975differential,do1992riemannian}) which is the reason why we kept it rather short.
	We consider the Jacobi field $J^n$ associated with the following geodesic variation 
	\begin{equation}\label{eq:conJFandDerL}
	f^n(s,t) = \exp_{c^n(t)} \left(s \log_{c^n(t)}y_{i-1}\right), \quad \text{ where }c^n(t) = \exp_{u_{i-1}}tv_n.  	  
	\end{equation}
	Then the desired covariant derivative is connected with the Jacobi field $J^n$ associated with the geodesic variation $f^n$ by
	\begin{equation}\label{eq:theJacFToComp}
	\tfrac{D}{dt}|_{t=0} L_t^n = \tfrac{D}{ds}|_{s=0} J^n(s), \text{ where }
	\quad J^n(s) = \tfrac{d}{dt}|_{t=0}f^n(s,t).  	
	\end{equation}
	This identity can be seen as follows. By the definition of $f^n,$ we have
	$L_t^n =  \log_{c(t)}y_{i-1} = \tfrac{d}{ds}|_{s=0} f^n(s,t).$
	Hence, 
	\begin{align}
	\tfrac{D}{dt}|_{t=0} L_t^n &= \tfrac{D}{dt}|_{t=0} \tfrac{d}{ds}|_{s=0} f^n(s,t)\notag \\
	& =  	\tfrac{D}{ds}|_{s=0} \tfrac{d}{dt}|_{t=0} f^n(s,t) = \tfrac{D}{ds}|_{s=0} J^n(s),
	\end{align} 
	which shows \eqref{eq:theJacFToComp}.
	We further notice that 
	\begin{equation}\label{eq:JforLat0}
	J^n(0) = v_n,  \quad  J^n(1) = 0.  	
	\end{equation}
	The first part of \eqref{eq:JforLat0} is a consequence of
	\begin{align}
	J^n(0) = \tfrac{d}{dt}|_{t=0} f^n(0,t) = \tfrac{d}{dt}|_{t=0}c^n(t) = v_n.
	\end{align}
	The second equality of \eqref{eq:JforLat0} is a consequence of the mapping $t \mapsto f^n(1,t)=y_{i-1},$
	being constant.
	
	We notice that \eqref{eq:JforLat0} determines $J^n$ uniquely which, in turn, 
	yields the desired derivative of $L^n$ via \eqref{eq:conJFandDerL} as being equal to $(J^n)'(0) := \frac{D}{ds}|_{s=0} J^n$. So it only remains to determine this uniquely determined Jacobi field $J^n$.
	
	Since $\mathcal M$ is a symmetric space, and thus the curvature tensor is parallel, and since $v_n$ is an eigenvector of the Jacobi operator,
	we end up with the problem of determining $x'(0)$ where $x$ is the solution of 
	the scalar boundary value problem 
	\begin{equation}\label{eq:JforLat0scalar}
	x''(s) + d^2 \lambda_n x(s) = 0, \quad           x(0) = 1, \   x(1) = 0,  
	\end{equation}
	where $d=d(y_{i-1},u_{i-1})= |\gamma'(t)|$ for all $t \in [0,1].$  	
	Here, $\lambda_n$ is the corresponding eigenvalue of the Jacobi operator. 
	Depending on the sign of $\lambda_n,$ the solution of \eqref{eq:JforLat0scalar} is given as follows.
	If $\lambda_n = 0,$ $x(s)= 1 - s,$ and so $x'(0) = -1.$ 
	If $\lambda_n > 0,$ 
	the general solution of the ODE is 
	$x(s)=$ $\alpha \cos( d \sqrt{\lambda_n} s ) + $ $\beta \sin( d \sqrt{\lambda_n} s ) $ 
	with real parameters $\alpha,\beta.$
	Then $x(0) = 1$ implies $\alpha=1$ which, in turn, yields using $0 =x(1) =$
	$ \cos( d \sqrt{\lambda_n} ) + $ $\beta \sin( d \sqrt{\lambda_n})$ that 
	$\beta = - \cos( d \sqrt{\lambda_n} )/\sin( d \sqrt{\lambda_n}).$
	Hence, $x'(0) =  - \cos( d \sqrt{\lambda_n} )/\sin( d \sqrt{\lambda_n})$  $\cdot \ d \sqrt{\lambda_n}.$  
	If $\lambda_n < 0,$  an analogous calculation replacing the trigonometric polynomials by their hyperbolic analogues yields 
	that  $x'(0) =  - \cosh( d \sqrt{-\lambda_n} )/$ $\sinh( d \sqrt{-\lambda_n})$  $\cdot \ d \sqrt{-\lambda_n}.$ 
	This shows \eqref{eq:ExplicitDiffOfL} and completes the proof.
\end{proof}

\begin{proof}[Proof of Lemma~\ref{lem:GradBnSphere}]
	In order to covariantly differentiate the mapping 	
	\begin{equation}
	t \mapsto B^n_t = \pt_{u_{i},\exp_{u_{i-1}}tv_n} z,
	\end{equation}
	we differentiate the mapping 
	\begin{equation}\label{eq:toDiffHol}
	t \mapsto  P^n_t z:= \pt_{\exp_{u_{i-1}}tv_n,u_{i-1}}  \pt_{u_{i},\exp_{u_{i-1}}tv_n} z
	\end{equation}
	in the tangent space  $T_{u_{i-1}}\mathcal M$ of $u_{i-1}.$ This follows from the relation between parallel transport and the covariant derivative, see for instance \cite{spivak1975differential}.
	
	If $v_n$ is parallel to $ \log_{u_{i-1}}u_i$, then the mapping in \eqref{eq:toDiffHol}
	is constant, and therefore, its derivative is $0$ which shows the first statement in 
	\eqref{eq:StatementParTransS2}.
	
	We show the second part of \eqref{eq:StatementParTransS2}. We may assume that 
	$v_n $ is orthogonal to $ \exp^{-1}_{u_{i-1}}u_i.$
	We have to differentiate the mapping in \eqref{eq:toDiffHol}, which means calculating 
	$\lim_{t \to 0} \frac{1}{t}(P^n_t-P^n_0).$
	Since the parallel transport is an isometry, the differential of \eqref{eq:toDiffHol} is an infinitesimal rotation (up to the identity) applied to $z.$
	We start out to calculate  $P^n_t-P^n_0$  in the basis of $T_{u_{i-1}}\mathcal M.$
	We note that $P^n_t-P^n_0$ corresponds to the holonomy along the (spherical) triangle $\Delta$ connecting the points $u_{i},$ $\exp_{u_{i-1}}tv_n$ and $u_{i-1}.$ 
	We observe that the rotation angle $\alpha_t$ of the rotation $P^n_t-P^n_0$ equals the spherical excess $A_t+B_t+C_t-\pi$
	of the triangle $\Delta_t,$ where $A_t$ is the angle at $u_{i},$ 
	$C_t$ is the angle at $u_{i-1}$ and $B_t$ is the angle at $\exp_{u_{i-1}}tv_n$ of the  triangle $\Delta_t.$
	Hence, 
	\begin{equation}\label{eq:intermedPtSphere}
	\lim_{t \to 0} \frac{1}{t}(P^n_t-P^n_0) =  
	\begin{pmatrix}
	0 & \lim_{t \to 0} \frac{\alpha_t}{t}  \\ -\lim_{t \to 0} \frac{\alpha_t}{t}  & 0
	\end{pmatrix}\pt_{u_i,u_{i-1}},   \quad \text{where } \
	\alpha_t = A_t+B_t+C_t-\pi.
	\end{equation}
	Here, the first identity is a consequence of the chain rule combined with $\alpha_0=0.$
	Since $\frac{\sin t}{t} = 1+ o(1),$ and since $C_t = \pi/2$ by the orthogonality of $v_n $ and $\log_{u_{i-1}}u_i,$ we get
	\begin{equation}\label{eq:secifyOmega}
	\omega =
	\lim_{t \to 0} \frac{\alpha_t}{t} =
	\lim_{t \to 0} \tfrac{A_t+B_t+C_t-\pi}{\sin t} 
	= \lim_{t \to 0}  \left( \tfrac{A_t}{\sin t} + \tfrac{B_t-\pi/2}{\sin t} \right).
	\end{equation}
	We recall that $d = d(u_{i},u_{i-1})$
	and use the following identities for spherical triangles with an angle of $\pi/2$ (cf. \cite{todhunter1863spherical})
	\begin{equation}
	A_t = \arctan\left(\tfrac{\tan t}{\sin d}\right), \quad
	B_t = \arctan\left(\tfrac{\tan d}{\sin t}\right).
	\end{equation}
	Using the Taylor expansion of the $\arctan$ function w.r.t. $0$ we obtain that
	\begin{equation}\label{eq:diffAt}
	\lim_{t \to 0}  \frac{A_t}{\sin t} = 
	\lim_{t \to 0} \frac{\frac{\tan t}{\sin d} + O\left(\left(\frac{\tan t}{\sin d} \right)^3\right)}{\sin t} =
	\lim_{t \to 0} \frac{1}{\cos t\sin d} =  \frac{1}{\sin d}.
	\end{equation}
	Further, by  L'Hospital's rule,
	\begin{align}\notag
	\lim_{t \to 0} \frac{B_t-\pi/2}{\sin t} &=
	\lim_{t \to 0} \frac{\arctan\left(\frac{\tan d}{\sin t}\right)-\pi/2}{\sin t} 
	=\lim_{t \to 0} \frac{-\frac{\tan d}{\sin^2 t} \cdot \cos t }{1+\left(\frac{\tan d}{\sin t}\right)^2} \cdot \frac{1}{\cos t}\\
	&= \lim_{t \to 0} \left(-\tfrac{1}{\tan d}\right) \cdot 
	\left( 1- \frac{1}{1+\left(\tfrac{\tan d}{\sin t}\right)^2} \right) = -\tfrac{1}{\tan d}.
	\label{eq:diffBt}
	\end{align}
	Now, we combine \eqref{eq:diffAt} with \eqref{eq:diffBt} and conclude, using \eqref{eq:secifyOmega},
	that $\omega   =\tfrac{1}{\sin d} - \tfrac{1}{\tan d}.$  
	Together with \eqref{eq:intermedPtSphere}, this shows \eqref{eq:StatementParTransS2}.
	To see \eqref{eq:StatementParTransS2GeneralVn}, we notice that the connection is linear w.r.t.\ the direction of differentiation.
	Therefore, \eqref{eq:StatementParTransS2GeneralVn} is a consequence of \eqref{eq:StatementParTransS2} and the 
	expression $\langle  v_n, w \rangle$ equals the coefficient of the corresponding linear combination. 
	
	If  $u_i = u_{i-1},$ the mapping in \eqref{eq:toDiffHol} is constant; hence differentiation of this mapping yields zero 
	which implies that the differential $\frac{D}{dt}|_{t=0} B_t^n = 0.$ 
	This shows the last assertion and completes the proof.
\end{proof}

\begin{proof}  [Proof of Lemma~\ref{lem:GradBnPos}] 
	Since the space of positive matrices is a Riemannian symmetric space representable as quotient of matrix Lie groups there are explicit formulae for the objects of Riemannian geometry such as the $\exp$ mapping and the parallel transport available. The corresponding formulae may be found in, e.g., \cite{sra2015conic}.
	Our first task is to explicitly express the mapping $B^n_t$ in the space of positive matrices which form a symmetric space.
	We use the notation $\gamma_t:[0,1]\rightarrow \M$ to denote the geodesic starting in $u_{i-1}$ with direction $v,$ i.e.,
	\begin{equation}\displaystyle
	\gamma_t := \exp_{u_{i-1}}tv = 
	u_{i-1}^{\tfrac{1}{2}}
	\mathrm{Exp} \left( u_{i-1}^{-\tfrac{1}{2}}  \ tv \  u_{i-1}^{-\tfrac{1}{2}}   \right)
	u_{i-1}^{\tfrac{1}{2}}.
	\end{equation}
	Here, $\mathrm{Exp}$ denotes the ordinary matrix exponential. 
	Then, $B^n_t$ may be expressed in the space of positive matrices by 
	\begin{equation}\displaystyle \label{eq:BntExplicitInPos}
	B^n_t = \pt_{u_i,\gamma_t} z = 
	u_{i}^{\tfrac{1}{2}} \
	\bar{\gamma_t}^{\tfrac{1}{2}}\
	\bar{z}\
	\bar{\gamma_t}^{\tfrac{1}{2}}\
	u_{i}^{\tfrac{1}{2}}
	\end{equation}
	where
	\begin{equation}\displaystyle
	\bar{\gamma_t} =  u_{i}^{-\tfrac{1}{2}}  \ \gamma_t \  u_{i}^{-\tfrac{1}{2}},
	\quad \text{and}\quad 
	\bar{z} =  u_{i}^{-\tfrac{1}{2}}  \ z \  u_{i}^{-\tfrac{1}{2}}.
	\end{equation}
	(We refer for instance to \cite{sra2015conic} for the corresponding formulae for the parallel transport.) 
	The covariant derivative in the space of positive matrices may be expressed in terms of the ordinary derivative of a curve in the vector space of matrices plus some ``correction terms''(as for instance explained in \cite{sra2015conic}). 
	In our situation, we have
	\begin{equation}\displaystyle\label{eq:CovDerForulaPosMat}
	\tfrac{D}{dt}|_{t=0} B_t^n =  \tfrac{d}{dt}|_{t=0} B_t^n 
	- \tfrac{1}{2} \left(v u_{i-1}^{-1}   \pt_{u_i,u_{i-1}} z +  
	\pt_{u_i,u_{i-1}} z u_{i-1}^{-1}  v \right).    
	\end{equation}
	We denote the terms in brackets in \eqref{eq:CovDerForulaPosMat} by
	\begin{equation}\displaystyle \label{eq:writeRasSSt}
	S + S^\top,  \qquad  S :=  v u_{i-1}^{-1}  \pt_{u_i,u_{i-1}} z.   
	\end{equation} 
	We further have that, by \eqref{eq:BntExplicitInPos},
	\begin{equation}\displaystyle 
	\pt_{u_i,u_{i-1}} z = B^n_0 = 	u_{i}^{\tfrac{1}{2}} \
	\bar u_{i-1}^{\tfrac{1}{2}}\
	\bar{z}\
	\bar u_{i-1}^{\tfrac{1}{2}}\
	u_{i}^{\tfrac{1}{2}}, \quad \text{ where} \quad 
	\bar u_{i-1} = u_{i}^{-\tfrac{1}{2}}  \ u_{i-1} \  u_{i}^{-\tfrac{1}{2}},
	\end{equation}
	and so we may explicitly express $S$ in terms of matrix multiplications by 
	\begin{equation}\displaystyle\label{eq:defMatSpos}
	S :=  v u_{i-1}^{-1} 
	u_{i}^{\tfrac{1}{2}} \
	\bar u_{i-1}^{\tfrac{1}{2}}\
	\bar{z}\
	\bar u_{i-1}^{\tfrac{1}{2}}\
	u_{i}^{\tfrac{1}{2}}.
	\end{equation}   
	In view of \eqref{eq:writeRasSSt} and \eqref{eq:CovDerForulaPosMat}, we have to compute 
	$\frac{d}{dt}|_{t=0} B_t^n$ in order to derive an explicit representation of 
	$\frac{D}{dt}|_{t=0} B_t^n$ in terms of matrices. Differentiating \eqref{eq:BntExplicitInPos}, we have that  
	\begin{equation}\displaystyle\label{eq:OrdDerInPos}
	\tfrac{d}{dt}|_{t=0} B_t^n = 
	u_{i}^{\tfrac{1}{2}} \
	\left(\tfrac{d}{dt}|_{t=0} \bar{\gamma_t}^{\tfrac{1}{2}}\right)\
	\bar{z}\
	\bar u_{i-1}^{\tfrac{1}{2}}\ 
	u_{i}^{\tfrac{1}{2}} 
	+
	u_{i}^{\tfrac{1}{2}} \
	\bar u_{i-1}^{\tfrac{1}{2}}\
	\bar{z}\
	\left(\tfrac{d}{dt}|_{t=0} \bar{\gamma_t}^{\tfrac{1}{2}}\right)\ 
	u_{i}^{\tfrac{1}{2}}.   
	\end{equation} 
	Introducing the notation 
	\begin{equation}\displaystyle\label{eq:defTandX}
	X := \tfrac{d}{dt}|_{t=0} \bar{\gamma_t}^{\tfrac{1}{2}},\quad 
	T:=  u_{i}^{\tfrac{1}{2}} \ X \bar{z}\ \bar u_{i-1}^{\tfrac{1}{2}} \ u_{i}^{\tfrac{1}{2}}, 
	\end{equation}
	the derivative $\frac{d}{dt}|_{t=0} B_t^n$ in \eqref{eq:OrdDerInPos} may be rewritten as 
	\begin{equation}
	\tfrac{d}{dt}|_{t=0} B_t^n =  T + T^\top.
	\end{equation}
	We express $X$ more explicitly now. 
	To that end, let $f(A)= A^{1/2}$ be the matrix square root function (which is unambiguously defined for positive matrices). We have, by the inverse function theorem, that
	$\mathrm{d}f_{A^2}(Z) = Y,$  where $A Y+Y A = Z,$
	meaning that, at the point $A$, the directional derivative of $f$ in direction $Z$
	is given by the solution $Y$ of the right-hand Sylvester equation.
	Hence, in order to get $\frac{d}{dt}|_{t=0} \bar{\gamma_t}^{\tfrac{1}{2}},$   
	we have to solve the following Sylvester equation for $X$,
	\begin{equation}
	\bar{\gamma_0}^{\tfrac{1}{2}} X + X \bar{\gamma_0}^{\tfrac{1}{2}} = \bar v 
	\quad \text{where} \quad \bar v := u_{i}^{-\tfrac{1}{2}} v u_{i}^{-\tfrac{1}{2}}. 
	\end{equation}      
	Summing up, 
	\begin{equation}\label{eq:SyvEqToSolve}
	X = \tfrac{d}{dt}|_{t=0} \bar{\gamma_t}^{\tfrac{1}{2}} 
	\quad \text{ is the solution of  } \quad
	\bar u_{i-1}^{\tfrac{1}{2}} X + X \bar u_{i-1}^{\tfrac{1}{2}} = \bar v. 
	\end{equation}   
	Plugging \eqref{eq:SyvEqToSolve} and \eqref{eq:defTandX} together with \eqref{eq:defMatSpos}  into \eqref{eq:CovDerForulaPosMat} shows that
	$
	\frac{D}{dt}|_{t=0} B_t^n = (S-\tfrac{1}{2}T) + (S-\tfrac{1}{2}T)^\top
	$
	which completes the proof.
\end{proof}

\end{document}

%% file: definitions.tex
\DeclareMathOperator{\TV}{TV}

\DeclareMathOperator{\pt}{pt}

\DeclareMathOperator{\BV}{BV}
\DeclareMathOperator{\TGV}{TGV}

\newcommand{\PTTGV}{\text{PT-TGV}^2_\alpha}
\newcommand{\STGV}{\text{S-TGV}^2_\alpha}
\newcommand{\MTGV}{\text{M-TGV}^2_\alpha}
\newcommand{\dpt}{\text{D}_{\text{pt}}}
\newcommand{\dpts}{\text{D}_{\text{pt}}^{\text{sym}}}
\newcommand{\ds}{\text{D}_{\text{S}}}
\newcommand{\dss}{\text{D}_{\text{S}}^{\text{sym}}}
\newcommand{\dsym}{\text{D}^{\text{sym}}}

\newcommand{\TGVat}{\TGV_\alpha ^2}

\newcommand{\uc}{{u_\circ}}
\newcommand{\um}{{u_-}}
\newcommand{\up}{{u_+}}

\newcommand{\ucc}{{u_{\circ,\circ}}}
\newcommand{\ucm}{{u_{\circ,-}}}
\newcommand{\umc}{{u_{-,\circ}}}
\newcommand{\ucp}{{u_{\circ,+}}}
\newcommand{\upc}{{u_{+,\circ}}}

\newcommand{\upm}{{u_{+,-}}}
\newcommand{\ump}{{u_{-,+}}}

\newcommand{\ycco}{{y^1_{\circ,\circ}}}
\newcommand{\ycmo}{{y^1_{\circ,-}}}

\newcommand{\ycct}{{y^2_{\circ,\circ}}}

\newcommand{\ymct}{{y^2_{-,\circ}}}

\DeclareMathOperator{\BD}{BD}

\DeclareMathOperator{\dive}{div}

\DeclareMathOperator{\loc}{loc}

\newcommand{\st}{ \, |\, }

\newcommand{\R}{\mathbb{R}}
\newcommand{\N}{\mathbb{N}}

\newcommand{\Du}{\mathrm{D} u }

\newcommand{\symgrad}{\mathcal{E}}

\newcommand{\Wrt}{\mathrm{D}}

\newcommand{\beq}{\begin{equation}}
\newcommand{\eeq}{\end{equation}}

\newcommand{\M}{\mathcal{M}}

\DeclareMathOperator{\prox}{prox}
\DeclareMathOperator{\argmin}{argmin}

%% file: experiments/compare_with_TV_2D_Pos3_BZ4_2/TV_alpha.txt
0.1

%% file: experiments/compare_with_TV_2D_Pos3_BZ4_2/deltaSNR_TV.txt
4.5

%% file: experiments/compare_with_TV_2D_S2_5/TV2_beta.txt
8.6

%% file: experiments/compare_with_TV_2D_S2_5/deltaSNR_TV2.txt
8.4

%% file: experiments/pos3_camino_data_5/TGV_r.txt
0.1

%% file: experiments/compare_with_TV_2D_Pos3_BZ4_2/deltaSNR_TGV.txt
4.9

%% file: experiments/compare_with_TV_2D_Pos3_BZ4_2/TV_TV2_alpha.txt
0.02

%% file: experiments/compare_with_TV_2D_Pos3_BZ4_2/TV_TV2_beta.txt
0.1

%% file: experiments/compare_with_TV_2D_Pos3_BZ4_2/deltaSNR_TV_TV2.txt
4.4

%% file: experiments/pos3_camino_data_5/TGV_s.txt
0.3